\documentclass[10.9pt,a4paper]{amsart}
\usepackage{a4wide}
\usepackage[english]{babel}
\usepackage[normalem]{ulem}

\usepackage[square,sort,comma,numbers]{natbib}

\usepackage[all]{xy}
\usepackage{pstricks}
\usepackage{enumerate}
\usepackage{amsfonts,amssymb,amsmath,pinlabel,array,hhline}
\usepackage{slashed}
\usepackage{tabulary}
\usepackage{fancyhdr}
\usepackage{color}
\usepackage{bbm}
\usepackage[position=b]{subcaption}
\usepackage[colorlinks,
    linkcolor={blue!50!black},
    citecolor={blue!50!black},
    urlcolor={blue!80!black}]{hyperref}

    \usepackage{tikz}
\usetikzlibrary{
  cd,
  calc,
  positioning,
  fit,
  arrows,
  decorations.pathreplacing,
  decorations.markings,
  shapes.geometric,
  backgrounds,
  bending
}
\usepackage{tikzsymbols}

\usepackage{overpic}

\usepackage{todonotes}

\newtheorem{theorem}{Theorem}[section]

\newtheorem{question*}{Question}
\newtheorem*{lemma*}{Lemma}
\newtheorem{corollary}[theorem]{Corollary}
\newtheorem{lemma}[theorem]{Lemma}
\newtheorem{proposition}[theorem]{Proposition}

\theoremstyle{definition}
\newtheorem{definition}[theorem]{Definition}
\newtheorem{construction}[theorem]{Construction}
\newtheorem{notation}[theorem]{Notation}
\newtheorem*{notation*}{Notation}
\newtheorem*{convention}{Convention}
\newtheorem{terminology}[theorem]{Terminology}
\newtheorem{example}[theorem]{Example}

\newtheorem{remark}[theorem]{Remark}
\newtheorem*{remark*}{Remark}
\theoremstyle{remark}
\newtheorem*{claim*}{Claim}

\definecolor{bettergreen}{rgb}{0.0, 0.5, 0.0}

\newcommand{\Z}{\mathbb{Z}}
\newcommand{\Q}{\mathbb{Q}}

\newcommand{\C}{\mathbb{C}}

\newcommand{\bbT}{\mathbb{T}}
\newcommand{\bsm}{\left(\begin{smallmatrix}}
\newcommand{\esm}{\end{smallmatrix}\right)}
\newcommand{\id}{\operatorname{id}}
\newcommand{\Bl}{\operatorname{Bl}}
\newcommand{\ev}{\operatorname{ev}}
\newcommand{\BS}{\operatorname{BS}}
\newcommand{\PD}{\operatorname{PD}}

\newcommand{\coker}{\operatorname{coker}}

\newcommand{\Ext}{\operatorname{Ext}}

\newcommand{\im}{\operatorname{im}}

\newcommand{\Hom}{\operatorname{Hom}}

\newcommand{\lk}{\ell k}
\newcommand{\Ad}{\operatorname{Ad}}
\newcommand{\that}{\hat{t}}
\newcommand{\shat}{\hat{s}}

\newcommand{\sgn}{\operatorname{sgn}}

\newcommand{\T}{{\operatorname{T}}}
\newcommand{\Ann}{\operatorname{Ann}}

\def\aug{\operatorname{aug}}

\newcommand{\onu}{\overline{\nu}}
\newcommand{\Ima}{\operatorname{Im}}

\DeclareSymbolFont{EulerScript}{U}{eus}{m}{n}
\DeclareSymbolFontAlphabet\mathscr{EulerScript}

\title{Algebraic concordance of links}

\author[D.~Cimasoni]{David Cimasoni}
\address{Universit\'e de Gen\`eve,  Suisse}
\email{david.cimasoni@unige.ch }
\author[A.~Conway]{Anthony Conway}
\address{The University of Texas at Austin, Austin TX}
\email{anthony.conway@austin.utexas.edu}
\author[G.~Simian]{Ga\"etan Simian}
\address{Universit\'e de Gen\`eve,  Suisse}
\email{gaetan.simian@unige.ch}

\begin{document}

\maketitle

%
%
\begin{abstract}
Algebraic concordance of knots can be understood from the perspective of Seifert matrices,  Blanchfield forms,  and homology surgery.
We initiate a systematic study of algebraic concordance for links from each of these viewpoints.
The present article is concerned with algebraic concordance from the perspective of homology surgery and Blanchfield forms, whereas a companion article by the third named author focuses on C-complexes and generalised Seifert matrices.
The outcome of the present work consists of two obstructions to~$\mu$-component links being concordant.
The first obstruction, called the {\em homology surgery invariant\/}, takes values in the Witt group of hermitian forms over the field of fractions~$Q$ of~$\Z[\Z^\mu]$.
 The second obtruction, called the {\em Blanchfield invariant\/}, takes values in a Witt group of~$Q/\Z[\Z^\mu]$-valued hermitian linking forms.
For~$\mu\le 2$, we describe these invariants in terms of generalised Seifert matrices.
\end{abstract}

\section{Introduction}

Two knots~$K_1,K_2 \subset S^3$ are \emph{concordant} if they cobound a locally flat annulus in~$S^3 \times [0,1]$.
A knot is \emph{slice} if it bounds a locally flat disc in the~$4$-ball or,  equivalently,  if it is concordant to the unknot.
Invariants of knot concordance can be constructed using a variety of methods ranging from algebraic topology  to gauge theory, Khovanov homology and Heegaard-Floer homology; see e.g.~\cite{OzsvathSzaboTau,Rasmussen,
KronheimerMrowkaGaugeS,HomSurvey}.
However, one of the oldest results in the area states that if a knot is slice, then it is algebraically slice~\cite{LevineKnotCob}, meaning that any of its Seifert matrices is \emph{metabolic} (i.e. congruent to a matrix of the form~$\bsm 0&B\\C &D\esm$, where~$B,C,D$ are square matrices).
In fact, given a knot~$K$, the following conditions are equivalent (see~\cite[Theorem 1.1]{CochranOrrTeichner} and~\cite{KeartonCobordism}).
\begin{enumerate}
\item The knot~$K$ is algebraically slice.
\item For any spin~$4$-manifold~$W$ with boundary the $0$-surgery~$M_K:=S^3_0(K)$ such that the inclusion induces an isomorphism~$H_1(M_K)\cong H_1(W)$,  the~$\Z[\Z]$-intersection form~$\lambda_W$ vanishes on a submodule of~$H_2(W;\Z[\Z])$ whose image in~$H_2(W)$ is a Lagrangian.
\item The Blanchfield form $\Bl_{M_K}$  is metabolic.
\end{enumerate}
For characterisations 
of algebraic sliceness
 involving gropes and Whitney towers, see~\cite{CochranOrrTeichner}.

There are concordance invariants that capture each of these obstructions:  the Witt class~$[A]$ of any Seifert matrix of~$K$, the class~$\lambda(K):=[\lambda_W] \in L^4_S(S^{-1}\Z[\Z] )/L^4(\Z[\Z] )$ of the~$\Z[\Z]$-intersection form of any (spin)~$\Z$-filling~$W$ of the $0$-surgery~$M_K$,  and the Witt class~$[\Bl_{M_K}] \in L^4(\Z[\Z],S)$ for a suitable multiplicative subset $S \subset \Z[\Z]$.
Here and in what follows,  $L^4$ denotes symmetric~$L$-theory; these  groups are reviewed in Appendix~\ref{sub:Ltheory}. 
The details of these constructions are provided in Section~\ref{sub:History},  as is motivation,  historical context,  justification for the terminology, and comparisons with prior work in the area.

The goal of this work and of~\cite{GaetanPhD,Simian} is to initiate a systematic study of algebraic concordance of links. 
Indeed, whereas many ideas in knot concordance have been generalised to links,  algebraic concordance appears to have fallen through the cracks (for algebraic concordance of boundary links, see~\cite{CappellShanesonLinkCobordism, Ko2,LeDimetThesis, SheihamThesis}, for link concordance in the context of the solvable filtration of~\cite{CochranOrrTeichner}, see e.g.~\cite{HarveyHomology,ChaSymmetric}).
The present article is concerned with algebraic concordance from the perspective of homology surgery and Blanchfield forms, whereas~\cite{Simian} focuses on the Seifert matrix theoretic approach.

\subsection{The homology surgery invariant}

Generalising the case of knots, two~$\mu$-component links~$L$ and~$L'$ are \emph{concordant} if they cobound $\mu$ disjoint locally flat annuli in~$S^3 \times [0,1]$.
Here and in what follows, links are assumed to be ordered and oriented.
Set~$M_{L,L'}:=X_L \cup X_{L'}$ (see Construction~\ref{cons:MLL'} for a more precise definition) and call a homomorphism~$\varphi \colon H_1(M_{L,L'}) \to \Z^\mu$ \emph{meridional} if it extends the maps~$H_1(X_L) \to \Z^\mu$ and~$H_1(X_{L'}) \to \Z^\mu$ that send the~$i$-th meridian to the~$i$-th canonical basis element.
When~$L=K$ is a knot and~$L'$ is the unknot, then~$M_{L,L'}$ is the~$0$-surgery on~$K$ and there is a unique meridional homomorphism.
For links,  there is in general no canonical choice of~$\varphi$ and, as a consequence, our invariants will be invariants of triples~$(L,L',\varphi)$.
In favourable cases,  independence of~$\varphi$ can be established.

Our first result concerns a generalisation of the invariant~$\lambda(K)$.
In what follows, given a~$4$-manifold~$W$ together with a map~$\psi \colon H_1(W) \to \Z^\mu$, 
we write~$\lambda_W^\Q$ for its~$\Q$-intersection form and~$\lambda_W^Q$ for its $Q$-intersection form,
where~$Q:=\Q(t_1,\ldots,t_\mu)$ is the field of fractions of the group ring~$\Lambda:=\Z[\Z^\mu]=\Z[t_1^{\pm 1},\ldots,t_\mu^{\pm 1}]$.
As described in more detail in Section~\ref{sec:HomologySurgeryInvariant}, given two~$\mu$-component links~$L,L'$,  a meridional homomorphism~$\varphi \colon H_1(M_{L,L'}) \to \Z^\mu$,  and a~$\Z^\mu$-filling~$W$ of~$M_{L,L'}$,  one can consider the difference
$$\lambda(L,L',\varphi):=[\lambda_W^Q]-[\lambda_W^\Q] \in L^4(Q).$$ 
A filling $W$ as above exists if and only if~$L$ and~$L'$ have the same linking numbers and {\em total Milnor invariant\/}~\cite{DavisNagelOrsonPowell},  a refinement of the set of triple Milnor numbers.
The fact that~$\lambda(L,L',\varphi)$ does not depend on the choice of the filling is proved in Proposition~\ref{prop:WellDef}; the argument is a variation on~\cite[Section 2]{ChaHirzebruch} and~\cite[Proposition 2.10]{ConwayNagelToffoli}.
Our results on~$\lambda(L,L',\varphi)$ can be summarised as follows.

\begin{theorem}
\label{thm:ConcordanceInvarianceGeneralIntro}
Given two~$\mu$-component links~$L$ and~$L'$ and a meridional $\varphi \colon H_1(M_{L,L'}) \to \Z^\mu$,  the \emph{homology surgery class}~$\lambda(L,L',\varphi) \in L^4(Q)$ satisfies the following properties.
\begin{enumerate}
\item (Definedness of the invariant) The class~$\lambda(L,L',\varphi) \in L^4(Q)$ is defined if and only if~$L$ and~$L'$ have the same linking numbers and total Milnor invariant.
\item (Concordance invariance) If~$L$ and~$L'$ are concordant~$\mu$-component links, then there exists a meridional homomorphism~$\varphi\colon H_1(M_{L,L'}) \to \Z^\mu$ such that 
$\lambda(L,L',\varphi)=0$.
\item 
(Independence of the choice of $\varphi$ for~$\mu \leq 2$)
If $\mu \leq 2$,  then~$\lambda(L,L',\varphi) \in L^4(Q)$ does not depend on the choice of $\varphi$ and we write
$$ \lambda(L,L'):=\lambda(L,L,\varphi). $$
\item (Calculation for~$\mu=1$) If~$L=K$ and~$L'=K'$ are knots,  then 
$$\lambda(K,K')=[H_F \oplus -H_{F'}]\,,$$
where the matrix~$H_F$ associated to a Seifert surface~$F$ for a knot~$K$ is defined in Construction~\ref{cons:MatrixH}.
In particular, the class~$[H_F] \in L^4(Q)$ is a concordance invariant.
\item (Calculation for~$\mu=2$)
If~$\mu=2$, then 
$$\lambda(L,L')=[H_F \oplus -H_{F'}]\,,$$
where the matrix~$H_F$ associated to a nice C-complex~$F$ for a~$2$-component link~$L$ is defined in Construction~\ref{cons:MatrixH}.
In particular,  the class~$[H_F] \in L^4(Q)$ is a concordance invariant.
\end{enumerate}
\end{theorem}

The first item is a consequence of the definition of $\lambda(L,L',\varphi)$ and~\cite[Theorem 1.1]{DavisNagelOrsonPowell}.
The second item follows relatively promptly from the definition of the invariant, see Theorem~\ref{thm:ConcordanceInvarianceGeneral}.
The third item is proved in Proposition~\ref{prop:mu=2Indep}. The fourth item has not appeared explicitly in print but can be seen as a consequence of work of Litherland~\cite{LitherlandCobordism} and Ko~\cite{Ko2}
(even though our geometric construction is different from theirs; see Remark~\ref{rem:KoComparison}).
Our main contribution is the fifth item of the theorem.
The proof of these latter two statements can be found in Theorem~\ref{thm:Calculation}.

Our computation of~$\lambda(L,L')$ is obtained by calculating a~$\Lambda$-intersection form of auxiliary manifolds~$W$ and~$W'$ with $\pi_1(\partial W)  \to \pi_1(W) \cong \Z^\mu$ surjective and applying Novikov-Wall additivity to obtain the intersection form over~$Q$.
This strategy does not apply for~$\mu > 3$ because for such~$4$-manifolds, the~$\Lambda$-module~$H_2(W;\Lambda)$ is free if and only if~$\mu \leq 3$ (see Lemma~\ref{lem:H_2IsFree}).
We also note that the dependence of the homology surgery class on $\varphi$ is reminiscent of the theory of Casson-Gordon invariants~\cite{CassonGordonCobordism} as well as higher order signatures~\cite{CochranOrrTeichner}.

\subsection{The Blanchfield invariant}

Next, we generalise the invariant~$\Bl(K):=[\Bl_{M_K}]$.
For a knot~$K$,  the Alexander module~$H_1(M_K;\Lambda)$ is torsion and~$\Bl_{M_K}$ is nonsingular.
On the other hand, given two $\mu$-component  links $L$ and $L'$ and a meridional homomorphism~$\varphi \colon H_1(M_{L,L'}) \to \Z^{\mu}$,  the Blanchfield form~$\Bl_{M_{L,L'}}^\varphi$ on the torsion submodule $TH_1(M_{L,L'};\Lambda)$ of the Alexander module might be singular.
A remedy (that goes back to Blanchfield~\cite{Blanchfield}) involves taking the quotient of~$TH_1(M_{L,L'};\Lambda)$ by its {\em maximal pseudonull submodule\/}~$zH_1(M_{L,L'};\Lambda)$,  a notion that we review in Appendix~\ref{sub:PseudoNull}.
The arguments of Blanchfield show that the form~$\Bl_{M_{L,L'}}^\varphi$ descends to~$\hat{t}H_1(M_{L,L'};\Lambda):=TH_1(M_{L,L'};\Lambda)/zH_1(M_{L,L'};\Lambda)$ leading to a nondegenerate linking form
$$ 
\hat \Bl_{M_{L,L'}}^\varphi \colon \hat{t}H_1(M_{L,L'};\Lambda) \times \hat{t}H_1(M_{L,L'};\Lambda) \to Q/\Lambda.
$$
In particular,~$\hat \Bl_{M_{L,L'}}^\varphi$ determines an element in the Witt group~$L^4_\textit{nd}(\Lambda,\Lambda\setminus \lbrace 0 \rbrace)$ of nondegenerate linking forms.
For brevity, we set
$$\Bl(L,L',\varphi):=[\hat \Bl_{M_{L,L'}}^\varphi] \in L^4_\textit{nd}(\Lambda,\Lambda\setminus \lbrace 0 \rbrace).$$
We emphasise that~$L^4_\textit{nd}(\Lambda,\Lambda\setminus \lbrace 0 \rbrace)$ does not agree with the Witt group~$L^4(\Lambda,\Lambda\setminus \lbrace 0 \rbrace)$ of nonsingular linking forms over torsion~$\Lambda$-modules admitting a projective resolution of length~$1$ (i.e.  admitting a square presentation matrix).
As we recall in Section~\ref{sub:Ltheory}, this latter Witt group is the one that occurs when~$\mu=1$, and that is involved in~$L$-theory.

\begin{notation}
A homomorphism~$\varphi \colon H_1(Y) \to \Z^\mu$ is an element of~$\Hom_\Z(H_1(Y),\Z^\mu) \cong [Y,\mathbb{T}^\mu]$.
When~$Y$ is an oriented closed~$3$-manifold and~$\mu=3$,  the \emph{degree of $\varphi$}, denoted $\deg(\varphi)\in \Z$, refers to the degree of the associated homotopy class of a map~$f_Y^\varphi \colon Y\to\mathbb{T}^3$,  i.e.~$\deg(\varphi)$ is the unique integer such that~$(f_Y^\varphi)_*([Y])=\deg(\varphi)[\bbT^3]$.
\end{notation}

Our results on the Blanchfield form can be summarised as follows.
In this statement and in the rest of the article, we denote by~$\Delta_L\in\Lambda=\Z[\Z^\mu]$ the multivariable Alexander
polynomial of a~$\mu$-component link~$L$, an invariant that is well-defined up to multiplication by units of~$\Lambda$.
\begin{theorem}
\label{thm:BlClassProperties}
Given two~$\mu$-component links~$L$ and~$L'$ and a meridional~$\varphi \colon H_1(M_{L,L'}) \to \Z^{\mu}$,  the \emph{Blanchfield class}~$\Bl(L,L',\varphi) \in L^4_\textit{nd}(\Lambda,\Lambda\setminus \{0\})$ is always defined, and
satisfies the following properties.
\begin{enumerate}
\item (Concordance invariance) If~$L$ and~$L'$ are concordant~$\mu$-component links, then there exists a meridional homomorphism~$\varphi\colon H_1(M_{L,L'}) \to \Z^\mu$ such that 
$\Bl(L,L',\varphi)=0$.
\item (Independence of the choice of $\varphi$ for $\mu \leq 2$)
If $\mu \leq 2$,  then~$\Bl(L,L',\varphi)$ does not depend on the choice of $\varphi$ and we write
$$ \Bl(L,L'):=\Bl(L,L',\varphi). $$
\item (Existence of square presentation matrices for $\Bl(L,L,\varphi)$)
Assume~$\Delta_L,\Delta_{L'} \neq 0$.
If~$\mu\le 2$, then~$\Bl^\varphi_{M_{L,L'}}$ (and thus~$\hat{t} H_1(M_{L,L'};\Lambda)$) can be represented by a square hermitian presentation matrix.
If~$\mu>3$, and if~$\mu=3$ and~$\deg(\varphi)\neq \pm 1$, then
 the module~$\that H_1(M_{L,L'};\Lambda)$ (and thus~$\Bl^\varphi_{M_{L,L'}}$) cannot be represented by a square presentation matrix.
\item (Nonsingularity) 
Assume that~$\Delta_L,\Delta_{L'} \neq 0$.
If either~$\mu \neq 3$,  or $\mu =  3$ and~$\deg(\varphi) \neq 0$,
then the Blanchfield form~$\hat \Bl_{M_{L,L'}}^\varphi$ is nonsingular.
Therefore,  if $\mu \leq 2$ and $\Delta_L,\Delta_{L'} \neq 0$, then the Blanchfield class determines an element~$\Bl(L,L') \in L^4(\Lambda,\Lambda\setminus \{0\})$.
\item (Calculation for~$\mu=1$) If~$L=K$ and $L'=K'$ are knots,  then
$$ \Bl(K,K') =\Bl(K) \oplus -\Bl(K').$$
Here~$\Bl(K)$ denotes the Witt class of the Blanchfield form~$\Bl_{M_K}$ and the latter is represented by any hermitian matrix~$H_F(t)$ as in Construction~\ref{cons:MatrixH}.
\item (Calculation for~$\mu=2$) If~$L$ is a $2$-component link with linking number one and~$L'=\mathcal{H}$ is the Hopf link, then the linking form~$\Bl_{M_{L,\mathcal{H}}}^\varphi$ is represented by~$H_F$ for a nice C-complex~$F$ for~$L$.
\end{enumerate}
\end{theorem}

The first item is well known for knots~\cite{KeartonCobordism} and for links when the coefficients are taken in the localisation~$\Lambda_S:=\Z[t_1^{\pm 1},\dots,t_\mu^{\pm 1},(t_1-1)^{\pm 1},\dots,(t_\mu-1)^{\pm 1}]$ of~$\Lambda$; see e.g.~\cite{Hillman}.
To the best of our knowledge, there is no published proof involving~$\Lambda$ coefficients; we provide one in Theorem~\ref{thm:ConcordanceInvarianceGeneral}.
The second item is proved in Proposition~\ref{prop:mu=2Indep}.
The third item is a consequence of Proposition~\ref{prop:BlanchfieldRepresentation} for~$\mu\le 2$ and
of Proposition~\ref{prop:NoSQP4} for~$\mu>2$.
A proof of the fourth item can be found in Proposition~\ref{prop:NonSingTorsionY}. 
The direct sum formula in the fifth item is proved in Proposition~\ref{prop:FLNP}; the fact~$\Bl_{M_K}$ is represented by~$H_F$ is implicit in the combination of work of Litherland~\cite{LitherlandCobordism} and~\cite{Ko2} but also follows from our work.
The sixth item is proved in Corollary~\ref{cor:CalculationBlmu=2}.

\begin{remark}
We comment on the relevance of items (3) and (4).
The long exact sequence in~$L$-theory requires one to work with even linking forms~$(T,b)$ where~$T$ has projective dimension one and~$b$ is nonsingular (see Appendix~\ref{sub:Ltheory}).
As will become apparent below,  knowing whether the Blanchfield form satisfies these conditions is therefore relevant to deciding whether the Blanchfield class is related to the homology surgery class.
\end{remark}

\begin{remark}\label{rem:triple-deg}
If~$L$ and~$L'$ are~$3$-component links with vanishing pairwise linking numbers, then the integer~$\deg(\varphi)$ appearing in items~(3) and~(4) coincides with the difference~$\overline{\mu}_{L'}(123)-\overline{\mu}_{L}(123)$ of triple linking numbers, see Example~\ref{ex:triple-deg}.
\end{remark}

By contrast with the first item of Theorem~\ref{thm:BlClassProperties} (which concerns $[\hat \Bl_{M_{L,L'}}^\varphi] \in L^4_\textit{nd}(\Lambda,\Lambda\setminus \lbrace 0 \rbrace)$),  when~$\hat \Bl_{M_{L,L'}}^\varphi$ is nonsingular and is defined on a module that admits a square presentation matrix, we are not able to prove that~$\Bl(L,L',\varphi) \in L^4(\Lambda,\Lambda\setminus \{0\})$ is a concordance invariant: the issue lies in showing that the metaboliser has projective dimension~$1$.
When~$\mu\le 2$, we are able to circumvent this issue using Theorem~\ref{thm:RelatingInvariantsIntro} below.
The case where~$\mu=3$ and~$\deg(\varphi)= \pm 1$ remains open.

\subsection{Relating the invariants}

We describe how the invariants~$\lambda(L,L',\varphi)$ and~$\Bl(L,L',\varphi)$ are related.
For~$\mu=1$,  this is a folklore result,  see e.g.~\cite{LitherlandCobordism} as well as the discussion in Section~\ref{sub:History}.
We refer to Theorem~\ref{thm:RelatingInvariants} for the proof.

\begin{theorem}
\label{thm:RelatingInvariantsIntro}
Fix $\mu \leq 2$ and let $L$ and $L'$ be two $\mu$-component links with equal linking numbers and~$\Delta_L,\Delta_{L'} \neq 0$.
The natural homomorphism~$L^4(Q) \to L^4(\Lambda,\Lambda \setminus \{0\})$ whose definition is recalled in Appendix~\ref{sub:Ltheory}
takes the homology surgery invariant~$\lambda(L,L')$ to the Blanchfield invariant~$\Bl(L,L')$.
In particular, if~$\lambda(L,L')$ vanishes in~$L^4(Q)$,  then so does~$\Bl(L,L')$ in~$L^4(\Lambda,\Lambda \setminus \{0\})$.
\end{theorem}

The second point of Theorem~\ref{thm:ConcordanceInvarianceGeneralIntro}  together with the first point of Theorem~\ref{thm:RelatingInvariantsIntro}
immediately yield the following corollary.

\begin{corollary}
\label{cor:Bl=0}
Fix $\mu \leq 2$ and let~$L$ and~$L'$ be two $\mu$-component links with equal linking numbers and~$\Delta_L,\Delta_{L'} \neq 0$.
If~$L$ and~$L'$ are concordant, then~$\Bl(L,L')$ vanishes in~$L^4(\Lambda,\Lambda \setminus \{0\})$.\qed
\end{corollary}

\begin{remark}
\label{rem:OtherlambdaIntro}
As we discuss in Section~\ref{sec:HomologySurgeryInvariant}, for~$\mu \leq 2$,  it is also possible to consider a variation on the invariant~$\lambda(L,L')$.
Namely,  if the linking numbers of $L$ and $L'$ agree, then one can pick a filling~$(W,\psi)$ of~$(M_{L,L'},\varphi)$ over~$\Z^\mu$ with $\psi$ an isomorphism and set
$$\widetilde{\lambda}(L,L'):=[\lambda_W^Q] \in L^4(Q)/\im(L^4(\Lambda) \to L^4(Q)).$$
This invariant satisfies all the same properties as~$\lambda(L,L')$ with the same proofs (in fact the proofs are slightly simpler since~$\lambda_W^\Q$ need not be taken into account).
The difference between~$\widetilde{\lambda}(L,L')$ and~$\lambda(L,L')$ is that,  since there is an injection~$L^4(Q)/\im(L^4(\Lambda) \to L^4(Q)) \hookrightarrow L^4(\Lambda,\Lambda\setminus \{0\})$,  the class~$\widetilde{\lambda}(L,L')$ is trivial if and only if~$\Bl(L,L')$ is trivial.
One drawback of~$\widetilde{\lambda}(L,L')$ is that we are only able to define it when~$\mu \leq 2$, see Remark~\ref{rem:mu=1Trick}.
\end{remark}

\subsection{Generalised Seifert matrices}
\label{sub:gen-seifert}

Concordant knots have Witt equivalent Seifert matrices.
One of the initial motivations for this project was to understand whether a similar statement holds for the generalised Seifert matrices~$A^\varepsilon$ of~\cite{CooperThesis,Cim04, CimasoniFlorens} or for the matrix $H(t_1,\ldots,t_\mu)$ involved in the definition of the multivariable signature~$\sigma_L\colon (S^1\setminus\{1\})^\mu\to\Z$.
A positive result in this direction is given by the following corollary of Theorem~\ref{thm:ConcordanceInvarianceGeneralIntro}.
\begin{corollary}
\label{cor:IntroSeifert}
Let $\mu \leq 2$ and let $L$ be a $\mu$-component link with~$\Delta_L \neq 0$.
For any nice C-complex~$F$ for $L$, the class~$[H_F] \in L^4(Q)$ is a concordance invariant such that the signature of~$H_F(\omega)$ coincides with the value at~$\omega$ of the extended multivariable
signature~$\sigma_L\colon (S^1)^\mu\to\Z$ defined in~\cite{CMP}.
\end{corollary}

However, it is not true that concordant links have Witt equivalent generalised Seifert matrices.
Actually, a given link can admit two C-complexes yielding generalised Seifert matrices which are not Witt equivalent,
as demonstrated by the following example.

\begin{figure}[h]
\centering
\begin{overpic}[width=7cm]{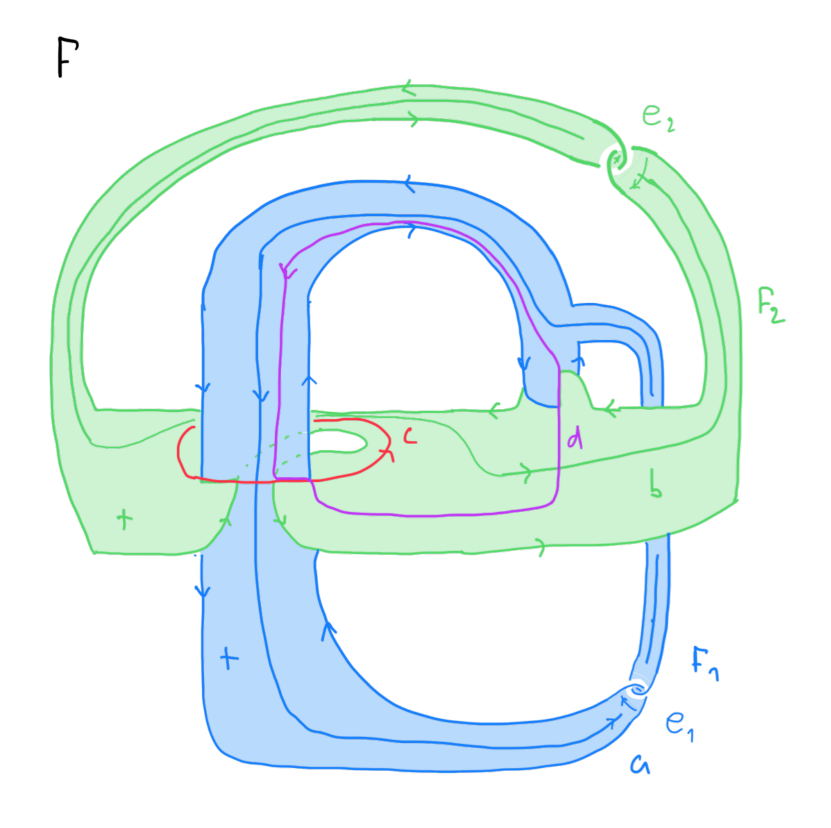}\end{overpic}
\caption{A~C-complex for the Hopf link.}
\label{fig:Counterexample}
\end{figure}

\begin{example}
\label{ex:Hopf}
Consider the~C-complex~$F$ for the Hopf link shown in Figure \ref{fig:Counterexample}. 
In the basis~$\{a, e_1, b, e_2, c, d\}$ depicted therein, the generalized Seifert form~$\alpha_F^{--}\colon H_1(F)\times H_1(F)\to\Z$ defined by~$\alpha_F(x,y)=\lk(x^{--},y)$ is represented by the matrix
$$A_F^{--} = \begin{pmatrix}
0&0&1&0&-1&0\\
1&-1&0&0&0&0\\
1&0&0&0&0&0\\
0&0&1&-1&0&0\\
0&0&0&0&0&0\\
0&0&1&0&-1&0
\end{pmatrix}.$$
One readily check that the real symmetric matrix~$A_F^{--} + (A_F^{--})^{\T}$
is non degenerate and has signature~$-2$, so~$\alpha_F^{--}$ is not metabolic.
On the other hand, the Hopf link also admits a contractible~C-complex~$F'$, whose associated form~$\alpha_{F'}^{--}$ is clearly metabolic.
As a consequence, the generalised Seifert forms~$\alpha_F^{--}$ and~$\alpha_{F'}^{--}$ are not Witt equivalent.
\end{example}

This contrasts with double concordance:
any strongly doubly slice link admits a collection of generalised Seifert matrices~$\{A^\varepsilon\}$ where~$A^\varepsilon$ is hyperbolic for each $\varepsilon$~\cite[Theorem 1.13]{ConwayOrson}.
For a more thorough discussion of generalised Seifert matrices in the context of link concordance, we refer to~\cite{Simian,GaetanPhD}.

\subsection{History and further motivation}
\label{sub:History}

In this section, we describe the geometric content of the homology surgery obstruction and justify the chosen terminology.
For motivation, it is helpful to first recall some history in the case of knots.

\subsubsection*{Algebraic concordance via Seifert matrices}
We begin by recalling the motivation behind the definition of algebraic concordance~\cite{LevineKnotCob}.
The idea is that for~$n\ge 2$,  if an~$n$-dimensional knot~$K \subset S^{2n+1}$ bounds a Seifert hypersurface that admits a metabolic Seifert matrix, then it is possible to ambiently surger this Seifert surface into a disc, thus proving that~$K$ is slice. 
In high dimensions,  the Witt class of the Seifert form is therefore a complete obstruction to sliceness whereas in the classical dimension, it leads to the notion of algebraic concordance.
This is the definition that motivates the search for an analogue of Levine's condition in terms of generalised Seifert matrices, leading to~\cite{Simian} and to the results summarised in Section~\ref{sub:gen-seifert}.

\subsubsection*{Blanchfield forms}
In the seventies, Kearton recast Levine's notion of algebraic concordance using Blanchfield forms~\cite{KeartonCobordism}.
As we have seen,  it is possible to generalise this version of algebraic concordance to links at the price of using a different and less amenable Witt group.
This was the content of Theorem~\ref{thm:BlClassProperties}.

\subsubsection*{Homology surgery}
We now describe Cappel-Shaneson's homology surgery theoretic approach to algebraic knot concordance~\cite{CappellShaneson}.
This is the most technical approach,  so we take the time to survey some of the main ideas in order to provide context and motivation for Theorem~\ref{thm:ConcordanceInvarianceGeneralIntro}.

We begin with knots.
Using Freedman's work~\cite{Freedman,FQ90}, a knot~$K \subset S^3$ is seen to be slice if and only if its~$0$-surgery~$M_K$ bounds a~$4$-manifold~$W$ with~$\pi_1(W)$ normally generated by a meridian and~$H_2(W)=0$.
One incarnation of the homology surgery theory programme consists of first finding a degree one normal map~$f \colon (W,M_K) \to (S^1 \times D^3,S^1 \times S^2)$ with~$f|_{M_K}$ a homology equivalence,
and then studying the obstruction to~$f$ itself being normal bordant to the desired~$\Z$-homology equivalence,
an obstruction denoted by
$$\sigma(f) \in \Gamma_4(\Z[\Z] \to \Z).$$
Here and in what follows,  $\Gamma_4$ (resp.~$L_4$) denotes quadratic~$\Gamma$-groups (resp. quadratic~$L$-theory).
Performing surgeries if necessary,  one can assume that $f$ is a $\pi_1$-isomorphism. 

In high dimensions,  $\sigma(f)$ is the only obstruction for $f$ to be normal bordant to a homology equivalence and it therefore leads to a calculation of the knot concordance group~\cite{CappellShaneson,CappellShanesonLinkCobordism} (though the programme is carried out a little differently and involves a relative~$\Gamma$-group, see e.g.~\cite[Lemma~3.1 and Theorem~4.1]{CappellShanesonLinkCobordism}).

In the classical dimension,  these two steps are part of an infinite process~\cite{CochranOrrTeichner}: as noted in~\cite[Proposition 23.19]{DET}, finding such a degree one normal map is equivalent to requiring that~$\operatorname{Arf}(K)=0$ (this is~$0$-solvability); on the other hand, the vanishing of the surgery obstruction corresponds (in a way that we will outline below) to algebraic concordance (this is~$0.5$-solvability).

\begin{remark}
Roughly speaking, the obstruction~$\sigma(f)$ is represented by the restriction of the~$\Z[\Z]$-intersection form of~$W$ to the surgery kernel~$\ker(H_2(W;\Z[\Z])\to H_2(S^1 \times D^3;\Z[\Z]))$.
The relation to the opening paragraph of this article now becomes apparent: if~$\operatorname{Arf}(K)=0$, then~$W$ can be taken to be spin, and
since~$H_2(S^1 \times D^3;\Z[\Z])=0$, the surgery obstruction is represented by the equivariant intersection form of~$W$.
Here and in what follows, we omit the quadratic data both for simplicity and because it is typically determined by~$\lambda_W$.
\end{remark}

The challenge is then to convert the surgery obstruction~$\sigma(f)=[\lambda_W]$ into a concordance invariant, i.e.  to show that $[\lambda_W]$ does not depend on the choice of spin~$\Z$-filling for~$M_K$; to do so, natural ideas involve considering~$[\lambda_W^Q]-[\lambda_W^\Q] \in L^4(Q)$ or~$[\lambda_W^{\Z[\Z]}] \in \Gamma_4(\Z[\Z] \to \Z)/L_4(\Z[\Z])$. 
Following~\cite{LitherlandCobordism,
CochranOrrTeichner,ChaHirzebruch}, this is what motivates the definition of our homology surgery obstruction (although we note that Cappell-Shaneson proceed differently~\cite[Proposition 3.2]{CappellShaneson}).
This concludes the motivation of the definition of the invariant~$\lambda(K)$ and explains our choice of terminology for it.

\begin{remark}
\label{rem:RelationToAlgConc}
We now briefly comment on the relation between~$\Gamma_4(\Z[\Z] \to \Z)/L_4(\Z[\Z])$ and the algebraic concordance group following~\cite{Ranicki, SheihamThesis}, as it motivates our work in the multivariable setting.
Consider the multiplicative subset~$\Sigma=\{ p(t) \in \Z[t^{\pm 1}] \mid p(1)=\pm 1\}$ and observe that there is a canonical isomorphism
$\Gamma_*(\Z[\Z] \to \Z) = \Gamma_*(\Z[\Z] \to \Sigma^{-1}\Z[\Z])$.
The following long exact sequence of~$\Gamma$ groups~\cite{CappellShaneson,Ranicki} implies that the aforementioned quotient is isomorphic to a Witt group of linking forms:
$$
\xymatrix{
0 \ar[r]& L_4(\Z[\Z]) \ar[r]& \Gamma_4(\Z[\Z]\to \Sigma^{-1}\Z[\Z]) \ar[r]& \Gamma_4(\mathcal{F}) \ar[r]& 0 \\
0 \ar[r]& L_4(\Z[\Z]) \ar[r]\ar[u]^=& L_4(\Sigma^{-1}\Z[\Z]) \ar[r]\ar[u]^\cong& L_4(\Z[\Z],\Sigma) \ar[u]^\cong \ar[r]& 0\,.
}
$$
Here~$\Gamma_4(\mathcal{F})$ denotes the relative~$\Gamma$ group corresponding to the diagram
$$
\xymatrix{
\Z[\Z] \ar[r]^{\id}\ar[d]^{\id}& \Z[\Z]\ar[d]^{\aug} \\
\Z[\Z] \ar[r]^{\aug}& \Z.
}
$$
The final subtlety is that for this localisation, there is no difference between the quadratic and symmetric theories, see~\cite[Proposition 7.9.2 (ii), see also page 831]{Ranicki}.
Since all these diagrams also hold for symmetric $L$-theory, the obstruction can be considered in
\[
\Gamma^4(\Z[\Z] \to \Z)/L^4(\Z[\Z])=\Gamma^4(\Z[\Z] \to \Sigma^{-1}\Z[\Z])/L^4(\Z[\Z])\cong L^4(\Sigma^{-1}\Z[\Z])/L^4(\Z[\Z])\cong L^4(\Z[\Z] ,\Sigma)\,.
\]
If one simplifies the situation by taking~$\Sigma=\Lambda \setminus \{0 \}$ in this final group,  then one obtains what we called the homology surgery theory class.
\end{remark}
%

We now turn to link concordance.
%
For motivation, we begin with a naive generalisation of the process described above, referring to~\cite{LevineOrrSurvey,LeDimetThesis} for more sophisticated approaches.
In order to study concordance to a link $L'$,  one can attempt to build a degree one normal map~$X_L \to X_{L'}$,  extend it to a degree one normal map~$M_{L,L'} \to M_{L',L'}$,  consider the obstruction to finding a degree one normal bordism~$f \colon (W,M_{L,L'}) \to (X_{L'} \times I,M_{L',L'})$ with~$f|_{M_{L,L'}}$  a meridian-preserving homology equivalence,  and finally consider the obstruction 
$$\sigma(f) \in \Gamma_4(\Z[\pi_1(X_{L'})]\to \Z).$$
Multiple roadblocks present themselves: to name but one such challenge,  contrarily to the case where~$L'$ is the unknot, there might not be a degree one map~$X_L \to X_{L'}$, e.g. because such a map is~$\pi_1$-surjective.
For example,  for~$L'$ the $\mu$-component unlink, if $\pi_1(X_L) \to \pi_1(X_{L'}) \cong F_\mu$ is meridian preserving,  then~$L$ is a boundary link.
For the study of (boundary) concordance of boundary links from the perspective of homology surgery, we refer to~\cite{CappellShanesonLinkCobordism, Ko2,LeDimetThesis, SheihamThesis}.

When~$L=K_1 \cup K_2$ is a~$2$-component link with linking number one and~$L'$ is the Hopf link, then~$\pi_1(X_{L'}) \cong \Z^2$,  and there is more hope to purse the homology surgery programme.
In fact, Davis~\cite{DavisConcordant} showed that if the Arf invariants of~$K_1$ and~$K_2$ vanish, then there is a meridian preserving degree one normal map 
$$f \colon (W,M_{L,L'}) \to (X_{L'} \times I,M_{L',L'}) \cong (T^2\times I \times I,T^3).$$
Furthermore, Davis showed that if $\Delta_L(t_1,t_2) \doteq 1$, then $L$ is concordant to $L'$.
Later work approaching concordance to the Hopf link using higher order invariants includes~\cite{FriedlPowellHopf,ChaSymmetric,KimWhitney}, but the analogue of algebraic concordance has not been considered (besides the study of the multivariable signature in~\cite{ConwayNagelToffoli}).
One can actually repeat the discussion from the unknot case: since~$H_2(X_{L'} \times I;\Z[\Z^2])=0$, the homology surgery obstruction takes the form 
$$ \sigma(f)=[\lambda_W] \in \Gamma_4(\Z[\Z^2] \to \Z).$$ 
This is essentially the definition of our homology surgery invariant.

\begin{remark}
Consider the multiplicative subset~$\Sigma=\{ p(t_1,t_2) \in \Z[t_1^{\pm 1},t_2^{\pm 1}] \mid p(1,1)=\pm 1\}$ and note the isomorphism~$\Gamma_*(\Z[\Z^2] \to \Z) = \Gamma_*(\Z[\Z^2] \to \Sigma^{-1}\Z[\Z^2]).$
Replacing $\Sigma^{-1}\Z[\Z^2]$ by the field of fractions $Q$, the definition of $\sigma(f) \in L_4(Q)$, leads to (a quadratic analogue of) our homology surgery obstruction $\lambda(L,L') \in L^4(Q).$
Just as for knots,  this group fits in an exact sequence
\[
L^4(\Z[\Z^2])\to L^4(Q)\to L^4(\Z[\Z^2],\Z[\Z^2]\setminus\{0\})\,,
\]
where the third group is a Witt group of linking forms, see Proposition~\ref{prop:LTheoryMap}.
With the exception of~$L^4(\Z[\Z^2])\cong\Z$ (see e.g.~\cite{MilgramRanicki}),
we do not know how to calculate these groups.
\end{remark}

As mentioned above,  the situation seems to be more complicated when $L'$ is not the Hopf link.
From the point of view of surgery theory, the algebraic concordance group should perhaps be a (quotient of) $\Gamma_4(\Z[\pi_1(X_{L'})]\to \Z)$ or the corresponding relative $\Gamma$ group.
On the other hand, since algebraic concordance typically involves abelian invariants,  just as for the Hopf link, one could consider the multiplicative system~$S=\Z[\Z^\mu] \setminus \{0\}$ and declare the algebraic concordance group to be a (quotient of)
$$\Gamma_4(\Z[\Z^\mu] \to S^{-1}\Z[\Z^\mu])=L_4(Q).$$
Forgetting the quadratic data for simplicity, this is the approach we took in this article.

\subsection{Comparison with $0.5$-solvable cobordism }

Inspired by~\cite{CochranOrrTeichner}, Cha defined a notion of~$0.5$-solvable cobordism for links~\cite[Definition 2.8]{ChaSymmetric} that generalises the notion of~$0.5$-solvability for knots.
We will not repeat the definition  here but note that the argument from~\cite[Proposition 5.10]{ConwayNagelToffoli} (which applies with~$Q$ coefficients equally as well as with~$\C^\omega$ coefficients) implies that~$0.5$-solvable cobordant links have vanishing homology surgery invariant (the converse appears to be unlikely).
This should be compared with the fact that~$0.5$-solvable cobordant links have (generically) equal multivariable signature and nullity~\cite[Theorem 1.5]{ConwayNagelToffoli}.
Relatedly, Kim proved that if two links are
$1$-solvable cobordant,  then their first non-zero Alexander polynomials agree up to norms and the Blanchfield forms of their exteriors (considered over the localised ring~$S^{-1}\Lambda$ where~$S \subset \Lambda$ is the multiplicative system generated by the~$t_i-1$) are Witt equivalent~\cite{KimWhitney}.
The corresponding statements for~$0.5$-solvability remain unknown.
We also do not know whether links with vanishing homology surgery invariant (or Blanchfield class) have (generically) equal multivariable signatures for~$\mu>2$ (recall Corollary~\ref{cor:IntroSeifert} for the case~$\mu\le 2$).
Indeed,  while it appears likely that a~$Q$-lagrangian for a~$\Z^\mu$-filling~$W$ of~$M_{L,L'}$ gives rise to a~$\C^\omega$-lagrangian,  we do not know how to force~$\sigma(W)=0$ without imposing additional constraints on~$\lambda_W^\Z$ (e.g. via the presence of~$0$-duals).

\subsection*{Organisation}
This article is organised as follows.
In Sections~\ref{sec:HomologySurgeryInvariant} and~\ref{sec:Blanchfield} we introduce the homology surgery invariant and Blanchfield invariant in the broader context of closed $3$-manifolds $Y$ endowed with an epimorphism $\varphi \colon H_1(Y) \to \Z^\mu$.
In Section~\ref{sec:LinkConcordance} we specify to the case $Y=M_{L,L'}$ and introduce our concordance invariants.
Section~\ref{sec:Fillings} builds the $4$-manifolds that are required to calculate the homology surgery invariant.
Section~\ref{sec:Relating} relates the two invariants.
The last two sections are slightly technical: Section~\ref{sec:ProjDim} analyses the projective dimension of the Alexander module $H_1(M_{L,L'};\Lambda)$ and the nonsingularity of the corresponding Blanchfield form, whereas Section~\ref{sec:IntersectionForm} carries out the intersection form calculations that are required to prove the main result of Section~\ref{sec:Fillings}.
The article concludes with Appendix~\ref{sub:Ltheory} on $L$-theory and Appendix~\ref{sub:PseudoNull} on pseudonull submodules.

\subsection*{Acknowledments}
The authors thank Peter Feller and Lukas Lewark for sharing parts of their upcoming paper~\cite{FL26}.
DC and GS were supported by the Swiss NSF grant 200021-212085.
AC was partially supported by the NSF grant DMS~2303674.

\subsection*{Conventions}

We work in the topological category.
Manifolds are assumed to be compact, connected and oriented unless otherwise specified.
Links are assumed to be ordered and oriented.
We write~$\overline{\nu}(N)$ (resp.~$\nu(N)$) for a closed (resp. open) tubular neighbourhood of a submanifold~$N$.
We set $\Lambda:=\Z[\Z^\mu]=\Z[t_1^{\pm 1},\ldots,t_\mu^{\pm 1}]$ for the ring of~$\mu$ variable Laurent polynomials
and~$Q:=\Q(t_1,\ldots,t_\mu)$ for its field of fractions.

\section{The homology surgery invariant of a~$3$-manifold over~$\Z^\mu$}
\label{sec:HomologySurgeryInvariant}

This section introduces the homology surgery invariant in the broader context of closed $3$-manifolds over $\Z^\mu$.
We assume some familiarity with the Witt group $L^4(Q)$ of hermitian forms over the field~$Q:=\Q(t_1,\ldots,t_\mu)$; see Appendix~\ref{sub:Ltheory} for some recollections.

\medbreak

Let~$Y$ be a closed~$3$-manifold and let~$\varphi \colon H_1(Y) \to \Z^\mu$ be a homomorphism.
Assume that~$(Y,\varphi)$ bounds over~$\Z^\mu$,  meaning that there is a~$4$-manifold~$W$ and a homomorphism~$\psi \colon H_1(W) \to \Z^\mu$ with~$\partial (W,\psi)=(Y,\varphi)$.
Recall that the adjoint of the~$\Q$-valued intersection form of~$W$ can be defined as the composition
\[
H_2(W;\Q)\to H_2(W,Y;\Q)\xrightarrow{\PD,\cong} H^2(W;\Q)\xrightarrow{\ev,\cong}\Hom_\Q(H_2(W;\Q),\Q)\,,
\]
whose kernel is~$\ker(H_2(W;\Q)\to H_2(W,Y;\Q))=\im(H_2(Y;\Q) \to H_2(W;\Q))$. As a consequence, we can consider the nonsingular intersection form
\[
\lambda_W^{\Q,\textit{ns}} \colon \frac{H_2(W;\Q)}
{\im(H_2(Y;\Q) \to H_2(W;\Q))} \times 
\frac{H_2(W;\Q)}
{\im(H_2(Y;\Q) \to H_2(W;\Q))} 
\to \Q\,.
\]
Similarly, working with (co)homology with twisted coefficients in~$Q$, we have a nonsingular form
\[
\lambda_W^{Q,\textit{ns}} \colon \frac{H_2(W;Q)}
{\im(H_2(Y;Q) \to H_2(W;Q))} \times 
\frac{H_2(W;Q)}
{\im(H_2(Y;Q) \to H_2(W;Q))} 
\to Q\,.
\]
Since both forms are nonsingular and hermitian, they define elements in the Witt group~$L^4(Q)$ of~$Q$-valued hermitian forms over~$Q$.
Taking the difference of the two classes leads  to
$$ 
[\lambda_W^{Q,\textit{ns}}]-[\lambda_W^{\Q,\textit{ns}}] \in L^4(Q)\,.
$$
The next lemma, whose proof involves a combination of ideas from~\cite[Lemma 2.3]{ChaHirzebruch} and~\cite[Corollary~2.11]{ConwayNagelToffoli},  proves that this difference does not depend on the choice~$(W,\psi)$.
When~$\mu \leq 3$, this result follows from work of Cha~\cite[Lemma 2.3]{ChaHirzebruch}.

\begin{proposition}
\label{prop:WellDef}
Let~$Y$ be a closed~$3$-manifold and let~$\varphi \colon H_1(Y) \to \Z^\mu$ be a homomorphism.
If~$(W,\psi)$ and~$(W',\psi')$ are two~$\Z^\mu$-fillings of~$(Y,\varphi)$, then 
$$ 
[\lambda_{W'}^{Q,\textit{ns}}]-[\lambda_{W'}^{\Q,\textit{ns}}]
=
[\lambda_W^{Q,\textit{ns}}]-[\lambda_W^{\Q,\textit{ns}}]
\in L^4(Q).
$$
\end{proposition}
\begin{proof}
We claim that if~$(V,\phi)$ is a closed~$4$-manifold over~$\Z^\mu$, then~$[\lambda_V^{Q,\textit{ns}}]-[\lambda_V^{\Q,\textit{ns}}]=0$. 
A half lives half dies argument shows that if~$V$ bounds a~$5$-manifold, then~$[\lambda_V^{\Q,\textit{ns}}]=0$ and similarly,   if~$V$ 
bounds a~$5$-manifold over~$\Z^\mu$, then~$[\lambda_V^{Q,\textit{ns}}]=0$.
We deduce that the assignment~$(V,\phi) \mapsto [\lambda_V^{Q,\textit{ns}}]-[\lambda_V^{\Q,\textit{ns}}]$ induces a well defined map
\begin{align*}
 \Omega_4(\Z^\mu) &\to L^4(Q) \\
(V,\phi) &\mapsto [\lambda_V^{Q,\textit{ns}}]-[\lambda_V^{\Q,\textit{ns}}].
\end{align*}
The claim is equivalent to the assertion that this map is trivial.
The remainder of the proof is now a slight modification of~\cite[Proposition 2.10]{ConwayNagelToffoli} so we proceed briskly. 
An Atiyah-Hirzebruch spectral sequence argument shows that~$\Omega_4(\Z^\mu) \cong \Omega_4 \oplus H_4(T^\mu;\Z)$, where the first summand is generated by~$\C P^2$ and the second summand by the $\Z^\mu$-manifolds given by the~${\mu}\choose{4}$ subtori~$T^4 \to T^\mu$.
It suffices to prove that the map vanishes on these~$\Z^\mu$-manifolds. 
Indeed since~$\C P^2$ is simply-connected, the twisted intersection form agrees with the untwisted one, so their difference is metabolic.
For~$T^4=T^3 \times S^1$, since~$H_2(T^4;Q)=0$,  the twisted form is zero.
The usual intersection form of~$T^4$ is a direct sum of~$3$ hyperbolics forms, so it vanishes in~$L^4(\Q)$ and therefore in~$L^4(Q)$ as well,  proving the claim.

We conclude.
Consider the closed $4$-manifold~$V:=W \cup -W'$ over $\Z^\mu$ and use Novikov addivity, which holds for Witt classes because $Q$ and $\Q$ are both fields:
\begin{align*}
[\lambda_V^{Q,\textit{ns}}]&=[\lambda_{W}^{Q,\textit{ns}}]-[\lambda_{W'}^{Q,\textit{ns}}] \in L^4(Q)  \\
[\lambda_V^{\Q,\textit{ns}}]&=[\lambda_{W}^{\Q,\textit{ns}}]-[\lambda_{W'}^{\Q,\textit{ns}}]  \in L^4(\Q)\,.
\end{align*}
Taking the difference of the first equality with the image of the second one via~$L^4(\Q)\to L^4(Q)$ yields the conclusion.
\end{proof}

Proposition~\ref{prop:WellDef} leads to the main definition of this section.

\begin{definition}
\label{def:WittInvariant3Manifold}
Let~$Y$ be a closed~$3$-manifold and let~$\varphi \colon H_1(Y) \to \Z^\mu$ be a homomorphism such that~$(Y,\varphi)$ bounds over~$\Z^\mu$.
The \emph{homology surgery invariant} of~$(Y,\varphi)$ is 
$$ 
\lambda(Y,\varphi):=[\lambda_W^{Q,\textit{ns}}]-[\lambda_W^{\Q,\textit{ns}}] \in L^4(Q)\,,
$$
where $(W,\psi)$ is a $\Z^\mu$-filling of $(Y,\varphi)$.
\end{definition}

Next,  we discuss the variation on~$\lambda(Y,\varphi)$ that was mentioned in Remark~\ref{rem:OtherlambdaIntro}.
For brevity, given a space $X$ and a map $\phi \colon H_1(X) \to \Z^\mu$, we write $\phi$ instead of $\phi \circ \operatorname{ab} \colon \pi_1(X) \to \Z^\mu$.

\begin{proposition}
\label{prop:WellDef}
Fix $\mu \leq 2$.
Let~$Y$ be a closed~$3$-manifold and let~$\varphi \colon H_1(Y) \to \Z^\mu$ be an epimorphism.
If~$(W,\psi)$ and~$(W',\psi')$ are two~$\Z^\mu$-fillings of~$(Y,\varphi)$ such that $\psi$ and $\psi'$ are isomorphisms on $\pi_1$, then 
$$ 
[\lambda_{W'}^{Q,\textit{ns}}]
=
[\lambda_W^{Q,\textit{ns}}]
\in L^4(Q)/\im(L^4(\Lambda) \to L^4(Q)).
$$
\end{proposition}
\begin{proof}
Consider the closed $4$-manifold~$V:=W \cup -W'$ over $\Z^\mu$.
Since $\varphi$ is an epimorphism and~$\psi,\psi'$ are isomorphisms,  we have~$\pi_1(V) \cong \Z^\mu$.
Novikov addivity yields
$$ [\lambda_V^{Q,\textit{ns}}]=[\lambda_{W}^{Q,\textit{ns}}]-[\lambda_{W'}^{Q,\textit{ns}}] \in L^4(Q)\,,$$
and it remains to check that $[\lambda_V^{Q,\textit{ns}}]$ lies in~$\im(L^4(\Lambda) \to L^4(Q))$. 
To do so,  note that since~$\mu \leq 2$,  the form $\lambda_V^\Lambda$ descends to a nonsingular hermitian form on the~$\Lambda$-module~$H_2(V;\Lambda)/\operatorname{rad}(\lambda_V^\Lambda)$, which is free~\cite[Theorem~B]{HambletonKreckTeichner}.
\end{proof}

Observe that if there is a $\Z^\mu$-filling of~$(Y,\varphi)$, then there is one such that the map to~$\Z^\mu$ is an isomorphism on~$\pi_1$: connect sum with copies of~$S^1 \times S^3$ to obtain a~$\pi_1$-surjection and then surger loops representing normal generators of the kernel.

\begin{definition}
\label{def:WittInvariant3ManifoldSimplified}
Fix~$\mu\le 2$.
 Let~$Y$ be a closed~$3$-manifold and let~$\varphi \colon H_1(Y) \to \Z^\mu$ be an epimorphism such that~$(Y,\varphi)$ bounds over~$\Z^\mu$.
The \emph{simplified homology surgery invariant} of~$(Y,\varphi)$ is 
$$ 
\widetilde{\lambda}(Y,\varphi):=[\lambda_W^{Q,\textit{ns}}] \in L^4(Q)/\im(L^4(\Lambda) \to L^4(Q)),
$$
where $(W,\psi)$ is a $\Z^\mu$-filling of $(Y,\varphi)$ with $\psi$ an isomorphism on $\pi_1$.
\end{definition}

\begin{remark}
\label{rem:mu=1Trick}
This idea does not generalise to the case where~$\mu>2$ since,  for a closed~$4$-manifold with~$\pi_1(V) \cong \Z^\mu$, the~$\Lambda$-module~$H_2(V;\Lambda)$ might not be free and~$\lambda^{\Lambda}_V$ might be singular.
The issue with nonsingularity 
is due to the following exact sequence involving $H_2(V;\Lambda) \cong H^2(V;\Lambda)$ that can be deduced from the universal coefficient spectral sequence (UCSS):
$$
H^2(\Z^\mu;\Lambda)
\to
H^2(V;\Lambda)
\xrightarrow{\ev} 
\Hom(H_2(V;\Lambda),\Lambda)
\to
H^3(\Z^\mu;\Lambda).
$$
\end{remark}

We conclude with the following immediate yet useful property of $\lambda(Y,\varphi)$ and~$\widetilde{\lambda}(Y,\varphi)$.

\begin{proposition}
\label{prop:WittCobordismInvariant}
Let~$Y$ be a closed~$3$-manifold and let~$\varphi \colon H_1(Y) \to \Z^\mu$ be a homomorphism.
If~$(Y,\varphi)$ bounds a~$4$-manifold~$(W,\psi)$ such that both~$\lambda_W^{Q,\textit{ns}}$ and~$\lambda_W^{\Q,\textit{ns}}$ are metabolic, then 
$$ \lambda(Y,\varphi)=0.$$
If $\mu \leq 2$ and $\lambda_W^{Q,\textit{ns}}$ is metabolic, then the same holds for $\widetilde{\lambda}(Y,\varphi)$.
\end{proposition}
\begin{proof}
This follows immediately from Definitions~\ref{def:WittInvariant3Manifold} and~\ref{def:WittInvariant3ManifoldSimplified}.
\end{proof}

\section{The Blanchfield form of a~$3$-manifold over~$\Z^\mu$}
\label{sec:Blanchfield}

This section introduces the Blanchfield class in the broader context of $3$-manifolds over $\Z^\mu$.
Section~\ref{sub:BlanchfieldDef} reviews the definition of the Blanchfield form whereas Section~\ref{sub:CobordismInvarianceBlanchfield} is concerned with its cobordism invariance over $\Z^\mu$.
Background material on multivariable Blanchfield forms can be found in~\cite{Hillman,BorodzikFriedlPowell}.

\subsection{The definition of the Blanchfield form}
\label{sub:BlanchfieldDef}

This section recalls the definition of the Blanchfield form of a~$3$-manifold over~$\Z^\mu$ and proves the~$\Z^\mu$-cobordism invariance of its Witt class under suitable assumptions on the bounding $4$-manifold.

\begin{construction}
\label{cons:Blanchfield}
Let~$Y$ be an~$n$-manifold and let~$\varphi\colon H_1(Y)\to\Z^\mu$ be an epimorphism.
For every~$0 \leq k \leq n$,  there is a \emph{relative Blanchfield pairing} 
\begin{align*}
\label{eq:Bl}
\Bl_Y^\varphi \colon TH_{n-k-1}(Y, \partial Y;\Lambda) \times TH_k(Y; \Lambda) &\longrightarrow Q \slash \Lambda \\
(x,y) &\longmapsto  \Ad_{\Bl}(y)(x),
\end{align*}
where the adjoint~$\Ad_{\Bl}$ is the composition of the following three maps: 
\[
\begin{tikzcd}
TH_k(Y; \Lambda) \arrow[r,"\mathrm{PD}^{-1}","\cong"'] \arrow[rrr,bend right=13,"\Ad_{\Bl} "]& 
TH^{n-k}(Y, \partial Y; \Lambda)&
\frac{H^{n-k-1}(Y, \partial Y; Q \slash \Lambda)}{\ker(\BS)}\arrow[r,"\ev"]
\arrow[l,"\BS"',"\cong"] 
&\overline{\Hom(TH_{n-k-1}(Y, \partial Y ; \Lambda) ; Q\slash \Lambda)}\,.
\end{tikzcd}
\]
Here, the first map is the inverse of the Poincar\'e duality isomorphism restricted to torsion submodules, the third is the evaluation map,  and the second arises from the Bockstein homomorphism associated to the exact sequence
of coefficients~$0 \to \Lambda \to Q \to Q\slash \Lambda \to 0$. 
In more details,  observe that by exactness 
$$TH^{n-k}(Y, \partial Y; \Lambda)=\ker(H^{n-k}(Y, \partial Y; \Lambda) \to H^{n-k}(Y, \partial Y;Q))=\im(\BS)$$
and use the first isomorphism theorem to see that~$\BS$ induces the announced isomorphism. 
In order to see that elements of~$\ker(\BS)$ evaluate to zero on elements of~$TH_{n-k-1}(Y,\partial Y;\Lambda)$, observe that elements of~$\ker(\BS)=\im(H^{n-k-1}(Y,\partial Y;Q) \to H^{n-k-1}(Y,\partial Y;Q/\Lambda))$ are represented by cocycles that factor through~$Q$-valued homomorphisms and therefore vanish on torsion elements. 
\end{construction}

In what follows, we make use of maximal pseudonull submodules; we refer to Appendix~\ref{sub:PseudoNull} for background on this notion.
First note that the maximal pseudonull submodule~$z H_1(Y;\Lambda)$ is contained in the kernel of~$\Ad_{\Bl}$.
This is a general fact about linking forms that follows from~\cite[Theorem~3.9~(3)]{Hillman} but is also a direct consequence of Lemmas~\ref{lemma:PnMapInduction} and~\ref{lemma:HomPN}(i) (together with Example~\ref{ex:NoPN}).
Since,  for a module~$H$, we set~$\that H:=TH/zH$, it follows that the relative Blanchfield pairing descends to a form 
$$
\hat{\Bl}_Y^\varphi \colon \that H_{n-k-1}(Y,\partial Y;\Lambda) \times \that H_k(Y;\Lambda)  \to Q/\Lambda\,.
$$
One of the main results of Blanchfield's original article implies that~$\hat{\Bl}_Y^\varphi$ is nondegenerate ``in both directions"~\cite[Theorems 3.2 and 3.3]{Blanchfield},  i.e. that~$\ker(\Ad_{\Bl_Y^\varphi})=zH_k(Y;\Lambda)$ and similarly for the other adjoint.
Restricting to the case of closed~$3$-manifolds leads to the following definition.

\begin{definition}
\label{def:BlanchfieldY}
Let~$Y$ be a closed~$3$-manifold and let~$\varphi \colon H_1(Y) \to \Z^\mu$ be an epimorphism.
The \emph{Blanchfield form} of~$(Y,\varphi)$ is the pairing
\[
\hat{\Bl}_Y^\varphi \colon \that H_1(Y;\Lambda) \times \that H_1(Y;\Lambda)  \to Q/\Lambda\,.
\] 
\end{definition}

The Blanchfield form is hermitian (see e.g.~\cite[Theorem 1(d)]{Powell}) and, as mentioned above, it is nondegenerate.
When~$H_1(Y;\Lambda)$ is torsion,  it can be shown that~$\hat{\Bl}_Y^\varphi$ is nonsingular
if and only if~$\mu\neq 3$ or~$\mu=3$ and~$\varphi$ has non-vanishing degree, see Proposition~\ref{prop:NonSingTorsionY}.
We refer to Remark~\ref{rem:BlanchfieldEven} for a condition ensuring that $\hat{\Bl}_Y^\varphi$ is even.
Before focusing on the cobordism invariance of the Blanchfield form, we make a brief remark on the choice of $\varphi$.

\begin{remark}
\label{rem:BlChoicePhi}
In general the linking form~$\hat{\Bl}_Y^\varphi$ depends on the choice of the epimorphism~$\varphi$.
However,  given two epimorphisms~$\varphi,\varphi' \colon H_1(Y) \to \Z^\mu$,  if there exists an orientation-preserving homeomorphism~$h \colon Y \to Y$ with $\varphi'=\varphi \circ h_*$, then~$h$ lifts to an isometry of~$\that H_1(Y;\Lambda)$ with respect to the Blanchfield forms~$\hat{\Bl}_Y^\varphi $ and~$\hat{\Bl}_Y^{\varphi'}$; the proof is analogous to the one from~\cite[Proposition~3.7]{ConwayPowell}.
\end{remark}

\subsection{Cobordism invariance over~$\Z^\mu$}
\label{sub:CobordismInvarianceBlanchfield}

The following theorem is stated in~\cite[Proposition~2.7]{BorodzikFriedlPowell} over~$\Lambda_S:=\Z[t_1^{\pm 1},\dots,t_\mu^{\pm 1},(t_1-1)^{\pm 1},\dots,(t_\mu-1)^{\pm 1}]$ and is proved in~\cite[Theorem~B]{KimWhitney}, also over~$\Lambda_S$.
This latter reference mostly follows~\cite[Theorem 2.4]{Hillman} but corrects a minor error.
We also note that the ideas of the proof go back to work of Letsche~\cite[Proposition 2.8]{Letsche}.
Since the passage from~$\Lambda_S$ to~$\Lambda$ is crucial to this work,  we recall the details of the argument.

\begin{theorem}
\label{theorem:Invariance3ManifoldVersion}
Let~$Y$ be a closed~$3$-manifold and let~$\varphi \colon H_1(Y) \to \Z^\mu$ be an epimorphism.
If~$(Y,\varphi)$ bounds a~$4$-manifold~$(Z,\varphi)$ such that the sequence~$
TH_2(Z, \partial Z; \Lambda) \stackrel{\partial}{\longrightarrow} TH_1(\partial Z; \Lambda) \stackrel{\iota}{\longrightarrow} TH_1(Z; \Lambda)$ is exact,
then~$\hat \Bl_Y^\varphi$ is metabolic.
\end{theorem}
\begin{proof}
During this proof,  we will omit the~$\varphi$ superscripts for brevity.
The strategy of the proof is to show that~$\Bl_{\partial Z}=\Bl_Y$ is metabolic and to then deduce,  by a purely algebraic argument,  that~$\hat \Bl_{\partial Z}$ is metabolic.
During the proof of this first step,  we will use the hypothesis that the exact sequence of the pair gives rise to the exact sequence~$TH_2(Z, \partial Z; \Lambda) \stackrel{\partial}{\longrightarrow} TH_1(\partial Z; \Lambda) \stackrel{\iota}{\longrightarrow} TH_1(Z; \Lambda)$.
The following claim carries out the first step of our plan.
\begin{claim*}
For~$K:=\ker(\iota)=\Ima(\partial)$, the following is a metabolizer for~$\Bl_{\partial Z}$:
\[
P:=K^\perp=\{ x\in TH_1(\partial Z;\Lambda)\mid \Bl_{\partial Z}(x,y)=0 \text{ for all } y\in K\}. 
\]
\end{claim*}
\noindent To prove this claim, consider the Blanchfield pairing
$
\Bl_Z \colon TH_2(Z,\partial Z; \Lambda) \times TH_1(Z;\Lambda) \to Q/\Lambda.
$
A now standard diagram chase yields the inclusion $K \subset K^\perp$, which implies~$P^\perp=K^{\perp\perp}\subset K^\perp=P$.

We now check the opposite inclusion~$K^{\perp}\subset K^{\perp\perp}$,
including the details to verify that the argument carries over from~$\Lambda_S$ to~$\Lambda$.
To do so, first observe that for any~$y\in K^\perp$, its image~$\iota(y)\in\iota(K^\perp)$ satisfies~$\Bl_Z(X,\iota(y))=\Bl_{\partial Z}(\partial X,y)=0$
for all~$X\in TH_2(Z,\partial Z;\Lambda)$, since~$\partial X$ belongs to~$K=\Ima(\partial)$ and~$y$ to~$K^\perp$.
Hence,~$\iota(y)$ lies in the kernel of~$\Ad_{\Bl_{Z}}$, which by Blanchfield's theorem is the maximal pseudonull submodule of~$TH_1(Z;\Lambda)$. Since we know that the inclusion~$K\subset K^\perp$ holds, this implies
\[
K^\perp/K=K^\perp/\ker(\iota)\cong\iota(K^\perp)\subset zH_1(Z;\Lambda)\,.
\]
Assume that~$x$ belongs to~$K^\perp$. By the equation displayed above together with Corollary~\ref{cor:max-pn},
the class~$[x]\in K^\perp/K$ can be written as~$[x]=\sum_i[x_i]$ where for each~$i$, there exists coprime
elements~$\alpha_i,\beta_i\in\Lambda$ such that~$\alpha_i[x_i]=\beta_i[x_i]=0\in  K^\perp/K$; in other words,
the elements~$\alpha_i x_i$ and~$\beta_i x_i$ belong to~$K$. Hence, for all~$i$ and all~$y\in K^\perp$, we have~$\alpha_i\Bl_{\partial Z}(x_i,y)=\Bl_{\partial Z}(\alpha_ix_i,y)=0$
and similarly~$\beta_i\Bl_{\partial Z}(x_i,y)=0$. Since~$\alpha_i$ and~$\beta_i$ are coprime, Corollary~\ref{cor:max-pn}
and Example~\ref{ex:NoPN} imply~$\Bl_{\partial Z}(x_i,y)=0$ for all~$i$ and all~$y\in K^\perp$.
This yields~$\Bl_{\partial Z}(x,y)=0$ for all~$y\in K^\perp$, meaning that~$x$ belongs to~$K^{\perp\perp}$.
This shows that~$P=K^\perp=K^{\perp\perp}=P^\perp$, thus proving the claim.

Let~$\that\colon TH_1(\partial Z;\Lambda)\to \that H_1(\partial Z;\Lambda)$ denote the canonical projection.
We now prove that~$\that(P)$ is a metabolizer for~$\hat \Bl_{\partial Z}$.
By definition, note that~$\hat \Bl_{\partial Z}(\that(x),\that(y)) = \Bl_{\partial Z}(x,y)$
for all~$x,y\in TH_1(\partial Z;\Lambda)$.
For any submodule~$N\subset TH_1(\partial Z;\Lambda)$, this equation readily implies 
\begin{equation}
\label{eq:hat-perp}
\that(N)^\perp=\that(N^\perp).
\end{equation}
The claim establishes that the submodule~$P=K^\perp\subset TH_1(\partial Z;\Lambda)$ satisfies~$P^\perp=P$. By~\eqref{eq:hat-perp} applied to~$N=P$, its image~$\that(P)\subset \that H_1(\partial Z;\Lambda)$
satisfies
\[
\that(P)^\perp=\that(P^\perp)=\that(P)\,.
\]
In other words, the submodule~$\that(P)\subset \that H_1(\partial Z;\Lambda)$ is a metabolizer for the nondegenerate pairing~$\hat \Bl_{\partial Z}$. 
This concludes the proof.
\end{proof}

\begin{remark}
\label{rem:WittNonSing}
We conclude this section with some remarks on Theorem~\ref{theorem:Invariance3ManifoldVersion}.
\begin{itemize}
\item Since we noted that Blanchfield forms are nondegenerate after taking the quotient by the maximal pseudonull submodule, the conclusion of Theorem~\ref{theorem:Invariance3ManifoldVersion} can be rephrased as saying that~$\hat \Bl_Y^\varphi$ is trivial in the Witt group~$L_\textit{nd}^4(\Lambda,\Lambda \setminus \{0\})$ of nondegenerate linking forms.
\item The most commonly used Witt group of linking forms~$(T,b)$ requires that the torsion module~$T$ has projective dimension~$1$,  that the form~$b$ is even and nonsingular, and that lagrangians also have projective dimension~$1$; see Appendix~\ref{sub:Ltheory}. 
\item There are (exceptional) situations where~$\hat \Bl_Y^\varphi$ is singular, see Proposition~\ref{prop:NonSingTorsionY} and Example~\ref{ex:SingBl}.
\item Even when~$\hat \Bl_Y^\varphi$ is nonsingular and $\hat{t}H_1(Y;\Lambda)$ has projective dimension $1$,  at this level of generality, we have not been able to prove that the metabolizer~$P$ from Theorem~\ref{theorem:Invariance3ManifoldVersion} also has projective dimension~$1$.
As a consequence, the conclusion of Theorem~\ref{theorem:Invariance3ManifoldVersion} generally does not take place in the Witt group that is most amenable to calculations.
\end{itemize}
\end{remark}

\section{Invariants of link concordance}
\label{sec:LinkConcordance}

The aim of this section is to prove items (2) and (3) of Theorem~\ref{thm:ConcordanceInvarianceGeneralIntro} and items (1) and (2) of Theorem~\ref{thm:BlClassProperties}. These results concern two link invariants whose definitions were only sketched in the introduction. We now give the precise constructions.

\medbreak

\begin{construction}
\label{cons:MLL'}
Let~$L$ and~$L'$ be~$\mu$-component links.
The zero framing induces homeomorphisms~$g\colon \partial X_L \to \bigsqcup_{i=1}^\mu T^2$ and~$g'\colon \partial X_{L'} \to \bigsqcup_{i=1}^\mu T^2$.
Define
$$M_{L,L'}:=X_L\cup_{g' \circ g^{-1}}  -X_{L'}.$$
Abelianisation gives rise to epimorphisms~$\varphi \colon \pi_1(X_L) \to \Z^\mu$ and~$\varphi \colon \pi_1(X_{L'}) \to \Z^\mu$ and these extend non-uniquely to a~$\varphi \colon \pi_1(M_{L,L'}) \to \Z^\mu$; the indeterminacy is closely related to that of triple linking numbers~\cite{DavisNagelOrsonPowell}.
Thus~$M_{L,L'}$ is a~$3$-manifold over~$\Z^\mu$ but not canonically so for~$\mu>1$.
We call a homomorphism~$\varphi \colon H_1(M_{L,L'}) \to \Z^\mu$ \emph{meridional} if it extends the maps~$H_1(X_L) \to \Z^\mu$ and~$H_1(X_{L'}) \to \Z^\mu$ that send the~$i$-th meridian to the~$i$-th canonical basis element.
\end{construction}

If~$L$ and~$L'$ have the same linking numbers and total Milnor invariant, then work of Davis-Orson-Nagel-Powell ensures that there is a meridional homomorphism~$\varphi \colon H_1(M_{L,L'}) \to \Z^\mu$ such that~$(M_{L,L'},\varphi)$ bounds a~$4$-manifold over~$\Z^\mu$~\cite[Theorem 1.1]{DavisNagelOrsonPowell}.

\begin{definition}
\label{def:WittInvariant}
Let~$L$ and~$L'$ be~$\mu$-component links, and let~$\varphi \colon H_1(M_{L,L'}) \to \Z^\mu$ be a meridional homomorphism.
\begin{itemize}
\item If~$L$ and~$L'$ have the same linking numbers and total Milnor invariant,  and if~$\varphi$ is a meridional homomorphism that extends over a~$\Z^\mu$-nullbordism for~$M_{L,L'}$, then the \emph{homology surgery invariant} of~$(L,L',\varphi)$ is defined as
$$
\lambda(L,L',\varphi):=[\lambda(M_{L,L'},\varphi)] \in L^4(Q).
$$
\item The \emph{Blanchfield class} of~$(L,L',\varphi)$ is defined as
$$
\Bl(L,L',\varphi):=[\hat \Bl_{M_{L,L'}}^\varphi] \in L_\textit{nd}^4(\Lambda,\Lambda \setminus \{0\}).
$$
\end{itemize}
\end{definition}

The next proposition shows that when~$\mu \leq 2$,  these invariants do not depend on the choice of~$\varphi$, thus proving the third item of Theorem~\ref{thm:ConcordanceInvarianceGeneralIntro} and the second item of Theorem~\ref{thm:BlClassProperties}.

\begin{proposition}
\label{prop:mu=2Indep}
Fix $\mu \leq 2$ and let $L$ and $L'$ be $\mu$-component links.
For any two meridional homomorphisms~$\varphi,\varphi' \colon H_1(M_{L,L'}) \to \Z^\mu$, there exists an orientation-preserving homeomorphism $h \colon M_{L,L'} \to M_{L,L'}$ satisfying $\varphi'=\varphi\circ h_*$.
In particular,  the Blanchfield class $\Bl(L,L',\varphi)$ and the homology surgery class $\lambda(L,L',\varphi)$ do not depend on the choice of $\varphi$.
\end{proposition}
\begin{proof}
If~$\mu=1$, then~$H_1(M_{L,L'})$ admits a unique meridional homomorphism and the statement is readily seen to hold.
We can therefore assume that~$\mu=2$.
The Mayer-Vietoris sequence for~$M_{L,L'}=X_L \cup X_{L'}$ gives rise to a split short exact sequence
$$
0 \to H_1(X_L) \oplus H_1(X_{L'}) \to H_1(M_{L,L'}) \to \Z  \to 0\,.
$$
A \emph{big loop} in $M_{L,L'}$ refers to a representative of the image of a generator of $\Z$ under an arbitrary splitting $s \colon \Z \to H_1(M_{L,L'})$.
By definition of the connecting homomorphism,  a big loop consists of an union of two arcs (one in $X_L$, the other in $X_{L'}$) whose two endpoints lie on the two tori that make up the boundary of the link exteriors.
There is no canonical choice of big loop,  but given any such choice,
a meridional homomorphism~$ H_1(M_{L,L'})\to \Z^2$ is well defined up to fixing its value on the chosen big loop.

Write $L=K_1 \cup K_2$  and consider a collar neighborhood of $\partial X_{K_1} \subset X_L$.
The zero framing leads to an identification of this neighborhood with $S^1 \times S^1 \times [0,1]$, where (say) the first $S^1$-factor corresponds to the Seifert longitude and the second to the meridian.
Perform a Dehn twist along the meridian in this collar resulting in a homeomorphism of~$S^1 \times S^1 \times [0,1]$ (i.e.  perform a $2$-dimensional Dehn twist $\tau$ along the annulus $S^1 \times [0,1]$ and extend it to $\id_{S^1} \times \tau$ on the aforementioned collar).
This homeomorphism is the identity on both boundary components of the collar,  so it can be extended by the identity on $X_L$ and $X_{L'}$  to a homeomorphism of $M_{L,L'}$.

Precomposing $\varphi$ by this homeomorphism changes the first component of its value on the big loop by $t_1^{\pm 1}$.
Dehn twisting along a meridian of $K_2$ leads to a similar change in the second coordinate.
The required homeomorphism $h$ can therefore be taken to be the composition of appropriate powers of these Dehn twists and their inverses.

We conclude by explaining how this implies that the invariants~$\lambda(L,L',\varphi)$ and~$\Bl(L,L',\varphi)$ do not depend on the choice of~$\varphi$.
For~$\lambda(L,L',\varphi)$,  this follows either from the definition of the invariant (use~$h$ to glue a collar to a filling of~$(M_{L,L'},\varphi)$ and observe that this does not affect the invariant) or from the cobordism invariance of Proposition~\ref{prop:WittCobordismInvariant} (if~$(M_{L,L'},\varphi)$ and~$(M_{L,L'},\varphi')$ are homeomorphic, then they are cobordant).
For~$\Bl(L,L',\varphi)$,  this follows from Remark~\ref{rem:BlChoicePhi}.
\end{proof}

The next result establishes the concordance invariance of these invariants, thus proving the second item of Theorem~\ref{thm:ConcordanceInvarianceGeneralIntro} and the first item of Theorem~\ref{thm:BlClassProperties}.

\begin{theorem}
\label{thm:ConcordanceInvarianceGeneral}
If~$L$ and~$L'$ are concordant~$\mu$-component links, then there exists a meridional homomorphism~$\varphi\colon H_1(M_{L,L'}) \to \Z^\mu$ such that 
$$\lambda(L,L',\varphi)=0 \text{ and } \Bl(L,L',\varphi)=0.$$
When $\mu \leq 2$, this holds for any choice of $\varphi$.
\end{theorem}
\begin{proof}
Let~$C \subset S^3\times [0,1]$ be a concordance between~$L$ and~$L'$ and let~$X_C$ denotes its exterior.
Note the homeomorphism~$\partial X_C\cong M_{L,L'}$.
The abelianisation~$H_1(X_C) \to \Z^\mu$ restricts to a meridional homomorphism~$\varphi \colon H_1(M_{L,L'}) \to \Z^\mu$.
Since~$H_i(X_C,X_L)=0$,  the chain homotopy lifting argument from~\cite[Proposition~2.10]{CochranOrrTeichner} (see also~\cite[Lemma 2.16]{ConwayNagelToffoli}) gives~$H_i(X_C,X_L;Q)=0.$
As a consequence, the inclusion induced map~$H_2(X_L;Q)\to H_2(X_C;Q)$ is an isomorphism, and~$H_2(\partial X_C;Q)\to H_2(X_C;Q)$ an epimorphism.
It follows that~$\lambda_{X_C}^Q$ is the zero form, and~$\lambda_{X_C}^{\Q}$ as well for similar reasons.
For~$\lambda(L,L',\varphi)$, the result now follows from Proposition~\ref{prop:WittCobordismInvariant}, so we focus on the Blanchfield form.
Since~$H_2(X_C, X_L;Q)=0$, it follows that~$H_2(X_C,\partial X_C;Q)=0$ and thus~$H_2(X_C,\partial X_C;\Lambda)$ is torsion.
The exact sequence of the pair is then seen to yield the exact sequence~$TH_2(X_C,\partial X_C;\Lambda) \to TH_1(\partial X_C;\Lambda) \to TH_1(X_C;\Lambda)$, and Theorem~\ref{theorem:Invariance3ManifoldVersion} implies that~$\hat\Bl_{M_{L,L'}}^\varphi$ is metabolic.
When~$\mu\le 2$, Proposition~\ref{prop:mu=2Indep} implies that this holds for any~$\varphi$.
\end{proof}

\begin{remark}
\label{rem:SimplifiedLinkInvariant}
When~$\mu \leq 2$,  using the same notation and assumptions as in Definition~\ref{def:WittInvariant},  one can define $\widetilde{\lambda}(L,L'):=\widetilde{\lambda}(M_{L,L'},\varphi)$, where $\widetilde{\lambda}$ denotes the simplified homology surgery obstruction from Definition~\ref{def:WittInvariant3ManifoldSimplified}.
The arguments from Proposition~\ref{prop:mu=2Indep} and Theorem~\ref{thm:ConcordanceInvarianceGeneral} respectively show that~$\widetilde{\lambda}(L,L')$ does not depend on $\varphi$ and that it is a concordance invariance.
\end{remark}

We conclude by noting the following calculation,  which shows that when~$\mu=1$, the concordance invariance 
of~$\Bl(K,K')$ follows from that of~$\Bl(K)$.
Recall that~$\Bl(K)$ denotes the class of~$\Bl^\varphi_{M_K}$ (with~$\varphi$ the abelianisation).

\begin{proposition}
\label{prop:FLNP}
For knots $K$ and $K'$,  there is an inclusion induced isometry~$\Bl_{M_K}\oplus -\Bl_{M_{K'}} \cong \Bl_{M_{K,K'}}$ and
the following equality holds in~$L^4(\Lambda,\Lambda\setminus\{0\})$:
$$\Bl(K,K') =\Bl(K) \oplus -\Bl(K').$$
\end{proposition}
\begin{proof}
A short Mayer-Vietoris argument shows that the inclusion induces an isometry $H_1(X_K;\Lambda)\cong H_1(M_K;\Lambda)$ with respect to the nonsingular hermitian pairings~$\Bl^\varphi_{X_K}$, yielding~$\Bl(K)=[\Bl^\varphi_{X_K}]$. Consider the Mayer-Vietoris decomposition for~$M_{K,K'}=X_K \cup_T -X_{K'}$ with $\Lambda$-coefficients:
$$
\overbrace{H_1(T;\Lambda)}^{\cong\Z} \xrightarrow{i} H_1(X_K;\Lambda)\oplus H_1(X_{K'};\Lambda)
\to H_1(M_{K,K'};\Lambda) \xrightarrow{\partial} \overbrace{H_0(T;\Lambda)}^{\cong \Z}\,.
$$
Since~$H_1(T;\Lambda)$ is generated by a Seifert longitude which bounds in~$H_1(X_K;\Lambda)$,
the map~$i$ is the zero map. 
Furthermore, the map $\partial$ is clearly injective.
It follows that the inclusion induces an isomorphism~$H_1(X_K;\Lambda)\oplus H_1(X_{K'};\Lambda)
\cong H_1(M_{K,K'};\Lambda)$.
That this isomorphism is an isometry follows for example from~\cite[Theorem~1.1]{FriedlLeidyNagelPowell} because $H_1(T;Q)=0$.
The fact that the claimed equality holds in the Witt group~$L^4(\Lambda,\Lambda\setminus\{0\})$ follows from
its definition: the~$\Lambda$-module~$H_1(X_K;\Lambda)$ is torsion, of projective dimension~$1$ since the Alexander module of a knot admits a square presentation matrix, and the nonsingular hermitian form~$\Bl^\varphi_{X_K}$ is even since~$H_1(X_K;\Lambda)$ is annihilated by~$\Delta_K\in\Lambda$ with~$\overline{\Delta}_K=\Delta_K$, see Remark~\ref{rem:BlanchfieldEven}.
\end{proof}

We conclude this section with a discussion of the possible extension of this result to links of~$\mu>1$ components.

\begin{remark}
\label{rem:Bl-add}
In general,  the Blanchfield form of~$X_L$ fails to be nondegenerate because the inclusion induced map~$H_1(\partial X_L;\Lambda) \to H_1(X_L;\Lambda)$ fails to be injective; this is why many articles in the area work over localised coefficients; see e.g.~\cite{KimWhitney,BorodzikFriedlPowell,CFT18,ConwayBlanchfield}.
However, given a~$2$-component link~$L$ with nonzero linking number~$\ell$, one has~$H_1(\partial X_L;\Lambda)=0$ (see e.g.~\cite[Equation~(10)]{LevineModule2}) and the Blanchfield form is seen to be nondegenerate over~$\Lambda$~\cite[Introduction,~B]{LevineModule2}.
In general, the map~$\partial$ in the Mayer-Vietoris argument above fails to be injective but this time, the intersection consists of the disjoint union~$T$ of two tori with an isomorphism~$H_0(T;\Lambda) \cong \Lambda/(t_1-1,t_2^\ell-1) \oplus \Lambda/(t_1^\ell-1,t_2-1)$ (see e.g.~\cite[Equation (9)]{LevineModule2}).
Thus for example, for the Hopf link~$\mathcal{H}$, one has~$H_1(X_\mathcal{H};\Lambda)=0$ but~$H_1(M_{\mathcal{H},\mathcal{H}};\Lambda) \cong \Z$.
It is therefore conceivable that the equality $\Bl(L,L') =\Bl(L) \oplus -\Bl(L')$ holds in some cases but the argument will not be as straightforward as in the case of knots.
\end{remark}

\section{Calculation of the homology surgery invariant}
\label{sec:Fillings}

The goal of this section is to prove the fourth and fifth items of Theorem~\ref{thm:ConcordanceInvarianceGeneralIntro}.
More precisely, we start in Section~\ref{subsec:plumbedmanifolds} with a quick review of plumbed~$3$-manifolds,  leading to the construction of
a~$\Z^\mu$-manifold~$M_L$ in Section~\ref{subsec:M_L}.
The most technical and lengthy part is Section~\ref{sub:WF}, where we build a~$\Z^\mu$-filling~$W_F$ of~$M_L$ assuming~$\mu\le 2$. This in turn allows us to construct a~$\Z^\mu$-filling~$W_{F,F'}$ of~$M_{L,L'}$ in Section~\ref{sub:WFF'}, and to prove the fourth and fifth items of Theorem~\ref{thm:ConcordanceInvarianceGeneralIntro} assuming intersection form calculations that will be carried out in Section~\ref{sec:IntersectionForm}.

\subsection{Plumbed manifolds}
\label{subsec:plumbedmanifolds}

In this section, we recall the definition of plumbed~$3$-manifolds, following~\cite[Section~4.5]{ConwayNagelToffoli} and~\cite[Section~2.5]{CMP}.
We also recall the definition of a meridional homomorphism on such a plumbed manifold, and analyse their existence and nonuniqueness.
These will be used in Section~\ref{subsec:M_L} to define the manifold~$M_L$, but also in Section~\ref{sub:WF} to describe the boundary of the exterior of a pushed in C-complex, and in Section~\ref{sub:WFF'} to build the manifold~$W_{F,F'}$.

\medskip

Let~$\Gamma=(V,E)$ be a finite unoriented graph without loops (multiple edges are allowed). Following classical references such as~\cite{Serre}, we denote
the set of {\em oriented\/} edges by~$E$, and the source and target maps by~$s,t\colon E\to V$. The graph~$\Gamma$ is unoriented in the sense that~$E$ is endowed with an involution~$e\mapsto \overline{e}$ such that~$\overline{e}\neq e$ and~$s(\overline{e})=t(e)$ for all~$e\in E$. Such a graph is a \emph{plumbing graph} if it is endowed with the following additional data: 
\begin{itemize}
\item each vertex~$i \in V$ is decorated with an oriented and connected surface~$F_i$;
\item each oriented edge~$e \in E$ comes with a sign~$\sgn(e)=\pm 1$ such that~$\sgn(\overline{e})=\sgn(e)$.
\end{itemize}

Given such a plumbing graph~$\Gamma$, the associated \emph{plumbed manifold} is the~$3$-manifold~$P(\Gamma)$
constructed as follows.
For each oriented edge~$e \in E$, let~$D_e$ be an open disk in~$F_{s(e)}$ such that the discs~$\{D_e\}_{e\in E}$ are disjoint. For each~$i\in V$, set
\[
F_i^\circ =  F_i\backslash \bigsqcup_{s(e)=i} D_e\,.
\]
Note that since~$F_i$ is oriented, so are~$F_i^\circ$ and~$D_e$, inducing opposite orientations on their shared boundary. By~$\partial D_e$, we mean the boundary of~$D_e$ endowed with the orientation induced by the orientation
of~$D_e$; hence,~$-\partial D_e$ is the same space endowed with the orientation induced by that of~$F_i^\circ$. Let~$S^1$ denote an oriented circle. The plumbed manifold~$P(\Gamma)$ is defined by
\[
P(\Gamma) = \left(\bigsqcup_{i\in V} F_i^\circ \times S^1\right) \slash \sim\,,
\]
where each pair of edges~$e,\overline{e}\in E$ yields the identification of~$F_{s(e)}^\circ\times S^1$ and~$F_{s(\overline{e})}^\circ\times S^1$ along one of their boundary tori; this identification is defined by
the maps
\begin{equation}
\label{eq:ident}
\begin{aligned}
(-\partial D_e)\times S^1 & \longrightarrow (-\partial D_{\overline{e}})\times S^1\\
( x, y )& \longmapsto (y^{-\textrm{sgn}(e)}, x^{-\textrm{sgn}(e)}).
\end{aligned}
\end{equation}
Since these maps are orientation reversing, the orientation on~$\bigsqcup_i F_i^\circ \times S^1$ extends to
an orientation on~$P(\Gamma)$.
Note that the boundary of~$P(\Gamma)$ consists of a (possibly empty) disjoint union of tori.

\begin{example}
\label{ex:P}
Consider the graph~$\Gamma$ given by two vertices, each decorated by a disc, and a single edge~$e$ linking them.
The associated plumbed manifold~$P(\Gamma)$ is given by two thickened tori glued along one of their toric boundary components via a map switching the coordinates, so it is homeomorphic to the exterior of the Hopf link, with the orientation of~$P(\Gamma)$ depending on~$\sgn(e)$.
\end{example}

\begin{construction}
\label{cons:PL}
Let~$L = K_1 \cup \ldots \cup K_\mu$ be a~$\mu$-component oriented link in the~3-sphere~$S^3$.
Let~$\Gamma_L$ denote the plumbing graph defined as follows: it has one vertex~$i$ for each component~$K_i$ of~$L$, each decorated by an oriented disc~$D_i$; for each pair~$i\neq j$, there are~$|\lk(K_i,K_j)|$ edges between the vertices~$i$ and~$j$, 
each with sign~$\sgn(\lk(K_i,K_j))$. We set
\[
P_L:=P(\Gamma_L)\,,
\]
an oriented~$3$-manifold with boundary given by~$\mu$ disjoint tori. Note that~$P_L$ only depends on the
linking matrix of~$L$.
\end{construction}

We now discuss coefficient systems on plumbed manifolds, assuming~$V=\{1,\dots,\mu\}$. 
Following~\cite{ConwayNagelToffoli}, we say that a homomorphism~$H_1(P(\Gamma)) \to \Z^\mu$ is {\em meridional\/} if, for any point~$p_i\in F^\circ_i$, it maps the class of~$\{p_i\}\times S^1$ to the element~$t_i$ of~$\Z^\mu$.
As we shall see, meridional homomorphisms always exist, but are in general not unique, due to the possible
existence of ``big loops'' in~$P(\Gamma)$.

We wish to understand this nonuniqueness in a precise way for the plumbed manifold~$P_L$ defined above.
To do so, let us consider an embedding~$\alpha\colon\Gamma_L\hookrightarrow P_L$ 
mapping~$i\in V$ to~$\alpha(i)\in D_i\times S^1$ and each unoriented edge~$e$ between~$i$ and~$j$ to a simple path~$\alpha(e)\subset (D_i^\circ \times S^1) \cup (D_j^\circ \times S^1)$ between~$\alpha(i)$ and~$\alpha(j)$ intersecting~$(\partial D_i^\circ \times S^1)\cup(\partial D_j^\circ \times S^1)$ only in the torus~$T_e:=\partial D_e \times S^1$ (identified with~$\partial D_{\overline{e}} \times S^1$ in~$P_L$).
Such embeddings clearly exist (but are not unique).
The point of this construction lies in the following result: a meridional homomorphism on~$H_1(P_L)$ is uniquely defined
up to its values on the ``big loops'' from~$H_1(\Gamma_L)$. Here is the precise statement.

\begin{lemma}
\label{lemma:MeridionalPL}
For any link~$L$, there exists a meridional homomorphism~$\varphi_P\colon H_1(P_L)\to \Z^\mu$.
Moreover, an embedding~$\alpha\colon\Gamma_L\hookrightarrow P_L$ as above induces an injection~$\alpha_\ast\colon H_1(\Gamma_L)\to H_1(P_L)$ such that:
\begin{enumerate}[(i)]
\item The subgroup~$\alpha_\ast(H_1(\Gamma_L))$ is  a free summand of~$H_1(P_L)$.
\item A meridional homomorphism is uniquely defined on~$H_1(P_L)/\operatorname{Im}(\alpha_\ast)$
and can take arbitrary values on~$\alpha_\ast(H_1(\Gamma_L))$, whose element we call {\em big loops\/}.
\end{enumerate}
\end{lemma}

\begin{proof}
Recall the plumbing graph~$\Gamma_L=(V,E)$ defined above. The Mayer-Vietoris sequence associated to the decomposition~$P_L=\bigcup_i D_i^\circ \times S^1$ (see e.g.~\cite[Lemma~3.3]{BFP} and~\cite[Lemma~4.7]{ConwayNagelToffoli}) yields the exact sequence 
\[
\bigoplus_{e\in E} H_1(T_e) \stackrel{\iota_t-\iota_s}{\longrightarrow} \bigoplus_{i\in V} H_1(D_i^\circ\times S^1) \longrightarrow H_1(P_L) \stackrel{\partial_*}{\longrightarrow} \bigoplus_{e\in E} H_0(T_e) \stackrel{d}{\longrightarrow} \bigoplus_{i\in V} H_0(D_i^\circ \times S^1),
\]
where~$T_e$ is the two dimensional torus appearing in the identification~\eqref{eq:ident} and~$\iota_t$ (resp.~$\iota_s$) is the homomorphism induced by the inclusion of~$T_e$ in~$D^\circ_{t(e)}\times S^1$ (resp.~$D^\circ_{s(e)}\times S^1$). This boils down to the split short exact sequence
\begin{equation}
\label{eq:MVes}
0 \longrightarrow \bigoplus_{i\in V} H_1(D_i^\circ\times S^1)/\operatorname{Im}(\iota_t-\iota_s)\longrightarrow H_1(P_L) \stackrel{\partial_*}{\longrightarrow} \ker(d)\longrightarrow 0\,.
\end{equation}
Note that a meridional homomorphism~$\varphi_P\colon H_1(P_L)\to \Z^\mu$ is uniquely defined
on~$H_1(D_i^\circ\times S^1)$, via~$\varphi_P([\{p_i\}\times S^1])=t_i$ for any~$p_i\in D^\circ_i$ and~$\varphi_P(\partial D_e)=\sgn(e)\, t_j$ for any~$e\in E$ with~$s(e)=i$ and~$t(e)=j$. 
By construction, this definition is coherent with the identifications coming from~\eqref{eq:ident}.
As a consequence, a meridional homomorphism~$\varphi_P$ 
is uniquely defined on the quotient appearing as the first group in~\eqref{eq:MVes}, but it can take arbitrary values
on the direct summand of~$H_1(P_L)$ corresponding to the image of~$\ker(d)$ by a section of this short
exact sequence.

To analyse the group~$\ker(d)$, note that any embedding~$\alpha\colon\Gamma_L\hookrightarrow P_L$ defined as above induces
a homomorphism~$\alpha_*\colon H_1(\Gamma_L)\to H_1(P_L)$ which fits in the commutative diagram
\[
\begin{tikzcd}
H_1(\Gamma_L)\arrow{r}{\iota} \arrow{d}{\alpha_*} & C_1(\Gamma_L) \arrow{r}{\partial} \arrow{d}{\cong}& C_0(\Gamma_L)\arrow{d}{\cong}\\
H_1(P_L)\arrow{r}{\partial_*}&\bigoplus_{e\in E} H_0(T_e) \arrow{r}{d} & \bigoplus_{i\in V} H_0(D^\circ_i \times S^1)\,,
\end{tikzcd}
\]
where~$\iota$ stands for the inclusion of~$H_1(\Gamma_L)=\ker(\partial)$ in~$C_1(\Gamma_L)$,
and the two unnamed vertical maps are the canonical isomorphisms. 
The commutativity of the left square uses the assumption on~$\alpha$.
These maps induce an isomorphism~$\varphi\colon H_1(\Gamma_L)=\ker(\partial)\stackrel{\cong}{\to}\ker(d)$ such that~$\partial_*\circ\alpha_*=\varphi$.
Thus, the map~$\alpha_*\circ\varphi^{-1}$ is a section of the short exact sequence~\eqref{eq:MVes}.
The lemma follows.
\end{proof}

\color{black}

\subsection{The generalised Seifert surgery}
\label{subsec:M_L}
This section recalls the construction of the generalised Seifert surgery,  denoted $M_L$, following~\cite{ToffoliThesis,Toffoli,CMP,CFP}; see also~\cite{NagelPowell} for the case~$\mu=1$.
\medskip

Let~$L=K_1\cup\dots\cup K_\mu$ be a~$\mu$-component oriented link in~$S^3$, and let~$X_L=S^3\setminus\nu(L)$ denote its exterior.
Recall that a {\em meridian\/} of~$K_i$ is an oriented simple closed curve~$m_i\subset \partial\onu(K_i)$ such that~$[m_i]=0\in H_1(\onu(K_i))$ and~$\lk(m_i,K_i)=1$.
Also, a {\em Seifert longitude\/} of~$K_i$ is an oriented simple closed curve~$\ell_i\subset \partial\onu(K_i)$ such that~$[\ell_i]=[K_i]\in H_1(\onu(K_i))$ and~$\lk(\ell_i,K_i)=0$.

Note that~$\partial X_L$ consists of~$\mu$ disjoint tori, and so does~$\partial P_L$. Therefore, it is possible to glue~$P_L$ and~$X_L$ along their boundaries. We do so using the homeomorphism~$\partial D_i\times S^1\cong \partial\onu(K_i)$ obtained by mapping~$\{p_i\}\times S^1$ (for some~$p_i\in\partial D_i$) to a meridian~$m_i$ and~$\partial D_i\times\{\ast\}$ (for some~$\ast\in S^1$) to a Seifert longitude~$\ell_i$.
Since the orientations on~$X_L$ and on~$P_L$ induce the same orientation on the boundary tori, we reverse the orientation of~$P_L$ and set~$M_L:=X_L\cup_\partial -P_L$.

\begin{definition}
\label{def:ML}
The closed~3-manifold~$M_L:=X_L\cup_\partial -P_L$ is the {\em generalised Seifert surgery\/} on~$L$.
\end{definition}

The main point of this construction is the following fact, see e.g.~\cite[Lemma~2.11]{CMP}:
the canonical isomorphism~$\varphi_X\colon H_1(X_L)\to\Z^\mu$ defined by~$\varphi_X(\gamma)=t_1^{\lk(\gamma,K_1)} \cdots t_\mu^{\lk(\gamma,K_\mu)}$ and any meridional
homomorphism~$\varphi_P\colon H_1(P_L)\to\Z^\mu$ extend to a map~$\varphi\colon H_1(M_L)\to\Z^\mu$.
Once again, we shall call such a map a {\em meridional\/} homomorphism, and it is in general not unique
due to the possible existence of (another type of) ``big loops''.

Due to the pervasive presence of this terminology in our work, we state it in a definition.
Here, recall from Construction~\ref{cons:PL} that $M_L=X_L \cup P_L$, with~$P_L=\bigsqcup_i (D_i^\circ \times S^1)/\sim$.

\begin{definition}
\label{def:meridional}
Let~$L=K_1\cup\dots\cup K_\mu$ be a~$\mu$-component link. A {\em meridional homomorphism for~$L$\/} is a homomorphism~$\varphi\colon H_1(M_L)\to\Z^\mu$
such that for all~$1\le i\le\mu$ and all~$p_i\in D^\circ_i$, it satisfies~$\varphi([\{p_i\}\times S^1])=t_i$.
\end{definition}

Note that a meridional homomorphism maps a~1-cycle~$\gamma\subset X_L$ to~$\varphi([\gamma])=t_1^{\lk(\gamma,K_1)} \ldots t_\mu^{\lk(\gamma,K_\mu)}$.

\subsection{The manifold~$W_F$}
\label{sub:WF}
Throughout the rest of this section, we assume~$\mu\le 2$.
As a first step towards constructing a filling of~$M_{L,L'}$ over~$\Z^\mu$, we construct a filling~$W_F$ of the generalised Seifert surgery~$M_L$.
The intersection forms of this~$4$-manifold will play a crucial role in the calculation of the homology surgery invariant~$\lambda(L,L'):=\lambda(M_{L,L'},\varphi).$

\medbreak

Fix~$\mu \leq 2$. 
Let~$L$ be a~$\mu$-component link and let~$F$ be a connected C-complex for~$L$.
We begin by introducing a size~$b_1(F)$ matrix~$H_F$ that will turn out to represent the equivariant intersection form of~$W_F$. This matrix requires a specific type of basis of~$H_1(F)$, hence some terminology.

\begin{terminology}
\label{ter:geo-symp-good}
~
\begin{enumerate}[(i)]
\item Let~$F$ be a Seifert surface for a knot~$K$. 
We say that a family of oriented simple closed curves $\{\alpha_1,\dots,\alpha_g,\beta_1,\dots,\beta_g\}$ forms a {\em good basis\/} for~$H_1(F)$ if is a geometric symplectic basis for~$H_1(F)$, i.e. if
it represents a basis for~$H_1(F)$ and satisfies the following properties.
\begin{itemize}
\item The curves~$\alpha_i$ are pairwise disjoint, and so are the curves~$\beta_i$.
\item The curves~$\alpha_i$ and~$\beta_j$ are disjoint for~$i\neq j$, while~$\alpha_i$ and~$\beta_i$ intersect geometrically once, with positive sign.
\end{itemize}
\item Let~$F = F_1 \cup F_2$ be a~C-complex for a~$2$-component link~$L=K_1\cup K_2$. 
Let us write~$\ell$ for the linking number~$\ell k(K_1,K_2)$ and~$g_i $ for the genus of~$F_i$, and set~$g:=g_1+g_2$. A \emph{good basis\/} for~$H_1(F)$ is a family of oriented simple closed curves
\[
\{\alpha_1,\ldots,\alpha_{g}, \beta_1,\ldots,\beta_{g}, \gamma_1,\ldots,\gamma_l, \delta_1, \ldots, \delta_n\}
\]
representing a basis of~$H_1(F)$, and satisfying the following properties.
\begin{itemize}
  \item The families~$ \{\alpha_1,\ldots,\alpha_{g_1}, \beta_1,\ldots,\beta_{g_1}\} $ and $ \{\alpha_{g_1+1},\ldots,\alpha_{g}, \beta_{g_1+1},\ldots,\beta_{g}\} $ are good bases for~$H_1(F_1) $ and~$H_1(F_2) $, respectively, and are disjoint from the clasps.
  \item The curves $\{\delta_1, \ldots, \delta_n\}$ are pairwise disjoint and disjoint from the~$\alpha_j$ and from the~$\beta_j$; each~$\delta_k$ passes through exactly two clasps, which are of opposite signs, so that each clasp with sign different from~$\sgn(\ell)$ is traversed by exactly one of the cycles~$\delta_k$.
   \item The curves $\gamma_1, \ldots, \gamma_l$ are disjoint from the~$\alpha_j$ and from the~$\beta_j$; each~$\gamma_k$ crosses exactly two clasps, both of sign~$\sgn(\ell)$.
\end{itemize}
Note that we have~$n=\frac{1}{2}(\#\{\text{clasps}\}-\vert\ell\vert)$ and~$l=\frac{1}{2}(\#\{\text{clasps}\}+\vert\ell\vert)-1$.
\item We say that a C-complex~$F$ is {\em nice\/} if~$H_1(F)$ admits a good basis.
\end{enumerate}
We occasionally refer to the $\delta_i$ as the \emph{cancelling curves}, to the $\gamma_i$ as the \emph{linking curves}, and to the~$\alpha_i$ and $\beta_i$ as the \emph{genus curves}.
\end{terminology}

\begin{remark} 
\label{remark:cancellingCurves}
~
\begin{enumerate}[(i)]
\item All C-complexes with~$\mu=1$ are nice, as any Seifert surface admits a good basis.
On the other hand, not all C-complexes with~$\mu=2$ are nice.
An example is illustrated below, where any two curves pairing clasps of opposite signs need to intersect.

\begin{center}
\begin{overpic}[width=4.5cm]{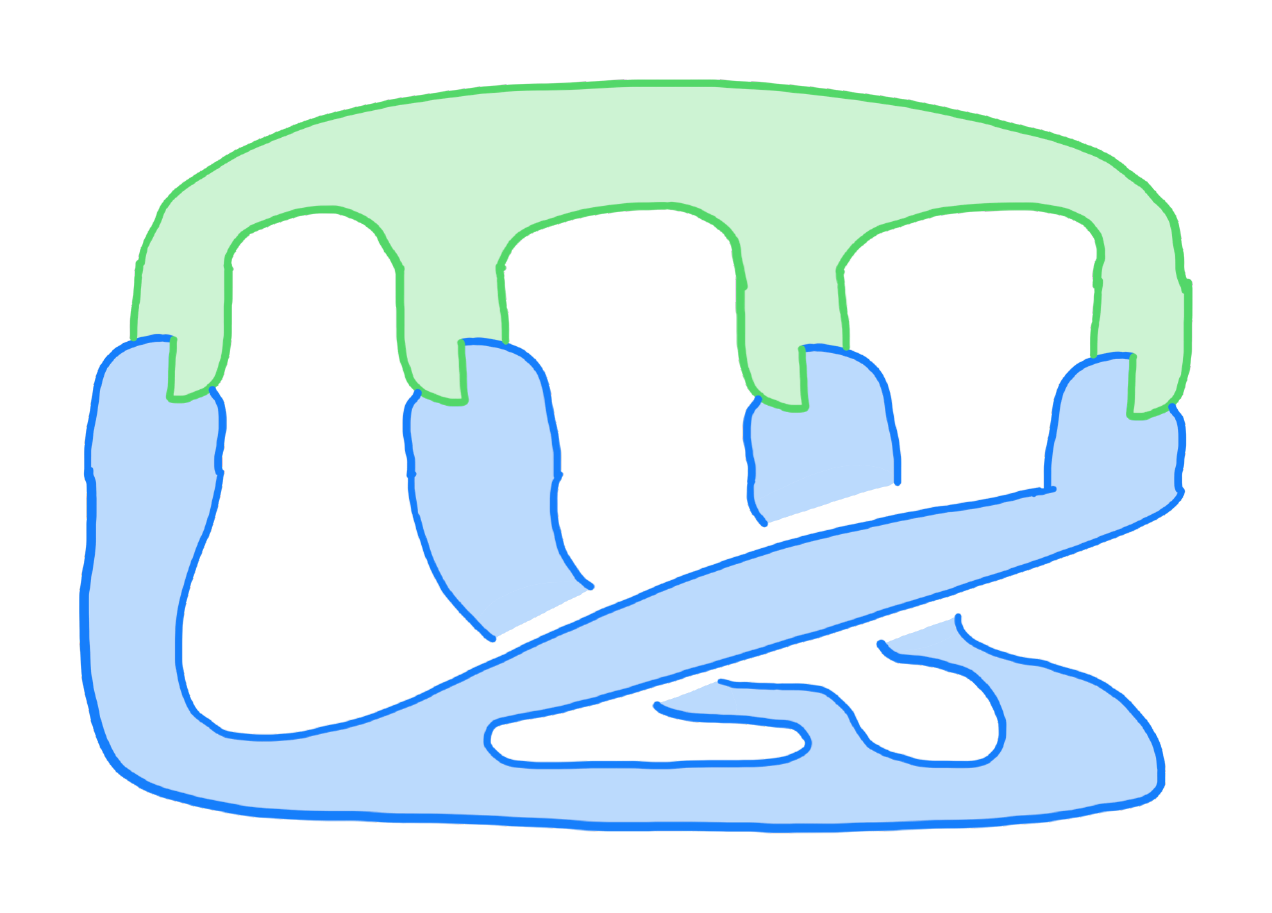}\end{overpic}
\end{center}

\item Any~$2$-component link admits a nice C-complex.
Indeed, if~$\ell\neq 0$, then~\cite[Theorem~2]{Davis} shows that there exists a~C-complex~$F$ for~$L$ with all clasps of same sign. 
One can then complete good bases for~$H_1(F_1)$ and~$H_1(F_2)$ (disjoint from the clasps) into a good basis for~$H_1(F)$ using oriented simple closed curves~$\gamma_1,\dots,\gamma_{\vert\ell\vert-1}$ that are disjoint from the genus curves.
(Note that we have~$n=0$ in this case.)
If~$\ell= 0$, the aforementioned result of~\cite{Davis} implies that there exists a C-complex with exactly two clasps.
 One can complete good bases for~$H_1(F_1)$ and~$H_1(F_2)$ into a good basis for~$H_1(F)$ using an oriented simple closed curve~$\delta_1$ passing through these clasps and following~$\partial F_i$ away from the clasps, hence remaining disjoint from the other curves. (Note that~$n=1$ and~$l=0$ in this case.)
 \item Theorem~\ref{thm:ExistenceOfW} below is still valid in the more general setting where we assume neither~$\beta_i\cap\beta_j=\emptyset$ for~$i\neq j$, nor~$\delta_k\cap\beta_j=\emptyset$, nor~$\gamma_k\cap\alpha_j=\gamma_k\cap\beta_j=\emptyset$ (see~\cite{GaetanPhD}).
However, since every link with~$\mu\le 2$ components admits a nice C-complex in the sense defined above,
we restrict ourselves to this setting.
\end{enumerate}
\end{remark}

\begin{construction}
\label{cons:MatrixH}
We define the matrices that feature in the fourth and fifth items of Theorem~\ref{thm:ConcordanceInvarianceGeneralIntro} as well as in the fifth and sixth items of Theorem~\ref{thm:BlClassProperties} and in Corollary~\ref{cor:IntroSeifert}.
We refer to Terminology~\ref{ter:geo-symp-good} for the definition of a nice C-complex.
\begin{enumerate}[(i)]
\item Let~$F$ be a Seifert surface, and let~$A$ denote the Seifert matrix (corresponding to the Seifert
form~$(x,y)\mapsto\alpha^-(x,y)=\lk(x^-,y)$ in a good basis for~$H_1(F)$. Set
\[
H_{F}:=D'\cdot(tA^{\T}-A)\cdot D\,,
\]
where~$D$ and~$D'$ are the following diagonal matrices:
\[
D'=\mathrm{I}_{g}\oplus(1-t)\mathrm{I}_{g}\,,\quad
D=\frac{1}{(t-1)}\mathrm{I}_{g}\oplus \frac{1}{t}\mathrm{I}_{g}\,.
\]
This matrix is analogous (to the one-variable version) of the one used by Ko~\cite[Section~4]{Ko2} in his study of concordance of boundary links; see also~\cite[Equation (2.3)]{BorodzikFriedl}.
\item Let~$F=F_1\cup F_2$ be a nice C-complex, and let~$A$ and~$B$ denote the generalized Seifert matrices corresponding to the generalized Seifert forms~$\alpha^{--}$ and~$\alpha^{-+}$ respectively, in a good basis for~$H_1(F)$. Set
\[
H_{F} := D' \cdot (t_1 t_2 A^{\T} - t_1 B^{\T} - t_2 B + A) \cdot D\,,
\]
where~$D$ and~$D'$ are the following diagonal matrices:
\small
\begin{align*}
D'&=\mathrm{I}_{g}\oplus(1-t_1)\mathrm{I}_{g_1}\oplus(1-t_2)\mathrm{I}_{g_2}\oplus(t_1-1)(t_2-1)\mathrm{I}_l\oplus\mathrm{I}_{n}\,,\\
D&=\frac{1}{(t_1-1)(t_2-1)} \mathrm{I}_{g}\oplus \frac{1}{t_1(t_2-1)}\mathrm{I}_{g_1}\oplus \frac{1}{t_2(t_1-1)}\mathrm{I}_{g_2}\oplus \frac{1}{t_1t_2}\mathrm{I}_l\oplus\frac{1}{(t_1-1)(t_2-1)}\mathrm{I}_{n}\,.
\end{align*}
\end{enumerate}
The fact that the entries of these matrices take values in $\Z[\Z^\mu]$ follows from Theorem~\ref{thm:ExistenceOfW} below.
\end{construction}

The main theorem of this section is the following.
\begin{theorem}
\label{thm:ExistenceOfW}
Fix $\mu \leq 2$.
For any~$\mu$-component link~$L$ with~$\Delta_L \neq 0$, there exists a~$4$-manifold~$W_F$, constructed
using a nice C-complex~$F$ for~$L$, which satisfies the following properties.
\begin{enumerate}
\item There is an isomorphism~$\Phi\colon \pi_1(W_F) \to \Z^\mu$ such that~$\partial(W_F,\Phi)=(M_L,\varphi)$
with~$\varphi$ a meridional homomorphism for~$L$.
\item The~$\Z$-intersection form of~$W_F$ decomposes as the direct sum of a zero form with a nonsingular form that is metabolic over $\Q$.
\item The~$\Z[\Z^\mu]$-intersection form of~$W_F$ is represented by~$H_F$.
\end{enumerate}
\end{theorem}

In this section, we construct~$W_F$ and prove the first item of Theorem~\ref{thm:ExistenceOfW}.
The proofs of the second and third items are postponed to Section~\ref{sec:IntersectionForm}.
Note that in the~$\mu=1$ case, the manifold~$W_F$ 
is distinct from Litherland and Ko's construction of a filling of $M_K$~\cite{LitherlandCobordism,Ko2} (see also~\cite{FL26}).
We also emphasise that Ko's work concerns~$\mu$-component \emph{boundary} links.

In a nutshell,~$W_F$ is obtained by attaching handles to the boundary of the exterior of a C-complex~$F$ pushed inside the~$4$-ball.
In order to describe to push offs of curves and discuss coefficient systems, 
we need an explicit parametrisation of this boundary as a plumbed~$3$-manifold,
which unfortunately requires rather tedious notation and technical arguments.
Hence, we divided this section into four parts. In the first, we define the exterior of a pushed-in C-complex and provide an explicit parametrization of part of its boundary as a plumbed manifold. In the second part, we construct the 4-manifold~$W_F$ by attaching suitable pieces to the exterior of the pushed in~C-complex along the aforementioned plumbed manifold. In the third part, we deal with the technical question of the extension of the meridional homomorphism
from~$M_L$ to~$W_F$.
Finally, the fourth part is devoted to the proof of Theorem~\ref{thm:ExistenceOfW} (1).

\subsubsection*{The exterior of a pushed-in C-complex}

We begin by defining the notion of a \emph{pushed-in C-complex\/} and by fixing an explicit parametrization of a neighborhood of such a space. 
These parametrizations will be helpful to calculate (equivariant) intersections in Section~\ref{sec:IntersectionForm}.
We then elaborate on the well-known result (see e.g.see~\cite{ConwayNagelToffoli,Toffoli}) that part of the boundary of this neighborhood is homeomorphic to a particular plumbed manifold.

\begin{terminology}
\label{terminology:PushedIn}
We introduce some terminology related to C-complexes and their neighborhoods in $B^4$.
Let~$F = F_1 \cup F_2$ be a C-complex for a~2-component link~$L = K_1 \cup K_2$ in~$S^3=\partial B^4$.
We use polar coordinates for~$B^4$, identifying it with~$S^3 \times [0,3]/(S^3 \times \{3\})$.
As in~\cite{CFT18}, we write~$X \star I$ instead of $X \times I$ for subspaces~$X\subset S^3$ and~$I\subset [0,3)$
to distinguish the depth coordinate from the others.
The C-complex~$F$ \emph{pushed-in}~$B^4$ is defined as~$\tilde{F} := \tilde{F}_1 \cup \tilde{F}_2\subset B^4$, ~where
\[
\tilde{F}_1 := \left(K_1 \star [0,1]\right) \cup \left(F_1 \star \{1\}\right) \quad\text{and}\quad 
\tilde{F}_2 := \left(K_2 \star [0,2]\right) \cup \left(F_2 \star \{2\}\right)\,.
\]
Note that each clasp in the C-complex~$F$ in~$S^3$ gives rise to a transverse intersection in~$\tilde{F}\subset B^4$.
This is illustrated in Figure~\ref{fig:NeighborhoodCombined}.

Next, given closed tubular neighborhoods~$\onu(\tilde{F}_1)$ and~$\onu(\tilde{F}_2)$ of~$\tilde F_1$ and~$\tilde F_2$ in~$B^4$,  
a neighborhood of $\widetilde{F}$ is obtained by setting~$\onu(\tilde{F}) := \onu(\tilde{F}_1) \cup \onu(\tilde{F}_2)$ and the \emph{exterior} of the pushed-in C-complex is then defined to be~$V_F := B^4 \setminus \nu(\tilde{F})$.
Note that the boundary of~$V_F$ decomposes as
\[
\partial V_F = X_L \cup \left(\partial \onu(\tilde{F}) \setminus \nu(L)\right)\,.
\]
\end{terminology}

It is known that~$\partial \onu(\tilde{F}) \setminus \nu(L)$ is homeomorphic to a plumbed manifold.
In order to describe to push offs of curves and discuss coefficient systems, we need an explicit realization of this homeomorphism. To that end, we now describe an explicit parametrization of~$\onu(\tilde{F})$.

\begin{figure}
\centering
\includegraphics[width=8cm]{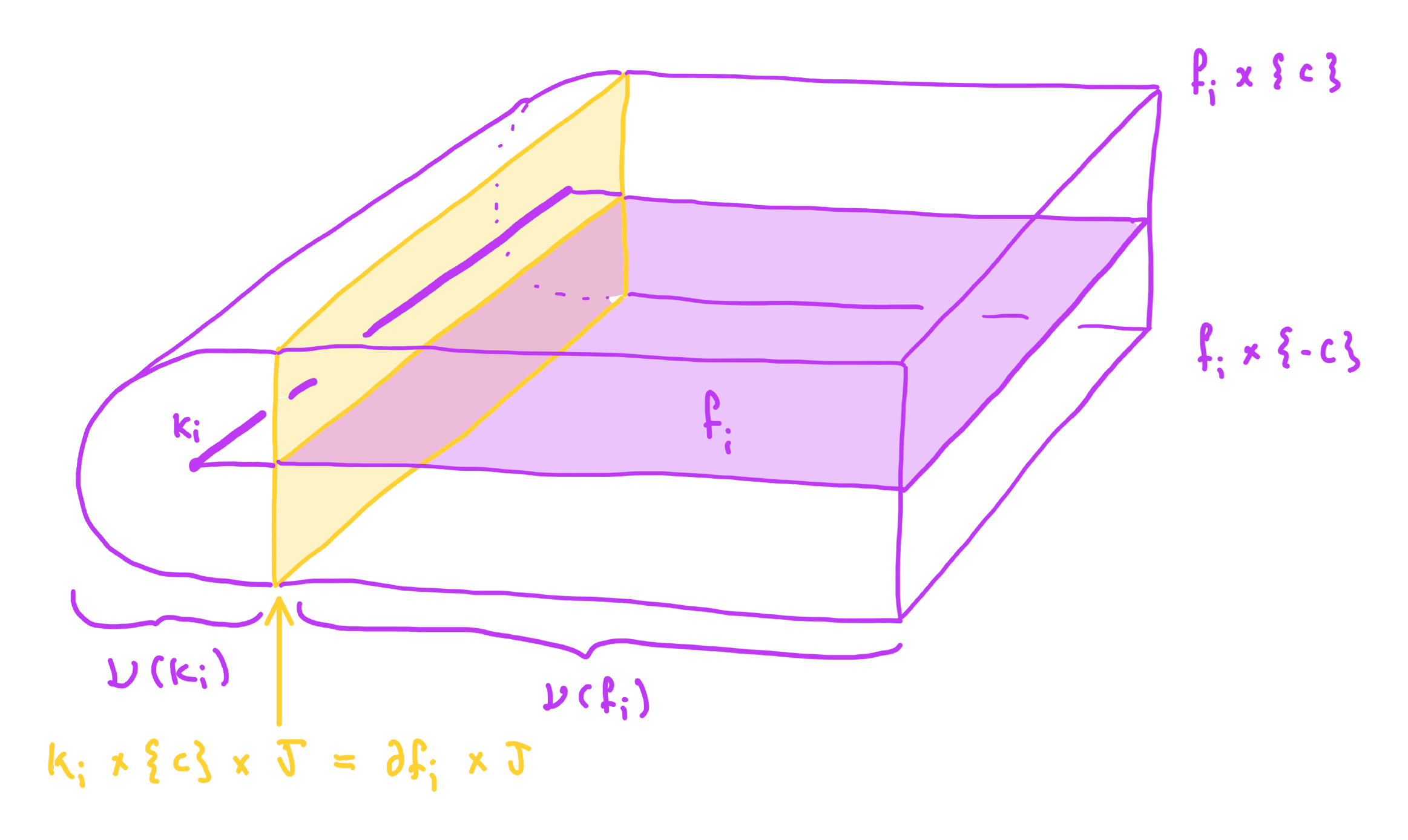}
\caption{Compatibility between $\onu(F_i)$ and $\onu(K_i)$.}
\label{fig:NeighborhoodCompatibility}
\end{figure}

\begin{construction}
\label{cons:parametrizationnu(F)}
Continuing with the notation from Terminology~\ref{terminology:PushedIn}, we describe a parametrization of a neighborhood of a pushed-in C-complex $\widetilde{F}$.
Fix~$0<c<1/4$ and set~$J := [-c,c]$. 
For~$i = 1,2$, fix a parametrization~$\onu(F_i)\cong F_i \times J$ of a closed bicollar neighborhood of~$F_i$ in~$S^3$ identifying~$F_i$ with~$F_i \times \{0\}$, so that the orientation of $J$ from~$-c$ to~$c$ matches the outward normal vector induced by the orientation of~$F_i$ via the ``outward vector first'' convention.
Additionally, fix a parametrization~$\onu(K_i) \cong K_i \times J^2$ of a tubular neighborhood of~$K_i$ in~$S^3$ identifying $K_i$ with~$K_i \times \{(0,0)\}$, and set~$f_i:= F_i\setminus\nu(K_i)$.
We require that the parametrized neighborhoods~$\onu(F_i)$ and~$\onu(K_i)$ are compatible in the sense that
\[
K_i \times \{c\} \times J = \partial f_i \times J\quad\text{and}\quad K_i\times [0,c)\times\{0\}=(F_i\setminus f_i)\times\{0\}\,,
\]
as illustrated in Figure~\ref{fig:NeighborhoodCompatibility}.
Using these data, we define the neighborhood~$\onu(\tilde{F}) := \onu(\tilde{F}_1) \cup \onu(\tilde{F}_2)$ of~$\tilde{F}$ in~$B^4$ via
\[
\onu(\tilde{F}_i) := \left(\nu(K_i) \star [0,i+c]\right) \cup \left(\nu(F_i) \star [i-c,i+c]\right)\,,
\]
see Figure~\ref{fig:NeighborhoodCombined}.
Throughout this section, color conventions in the figures remain consistent: blue represents~$F_1$, green represents~$F_2$, and purple may denote either.
\end{construction}

\begin{figure}
\centering
\begin{overpic}[width=0.7\textwidth]{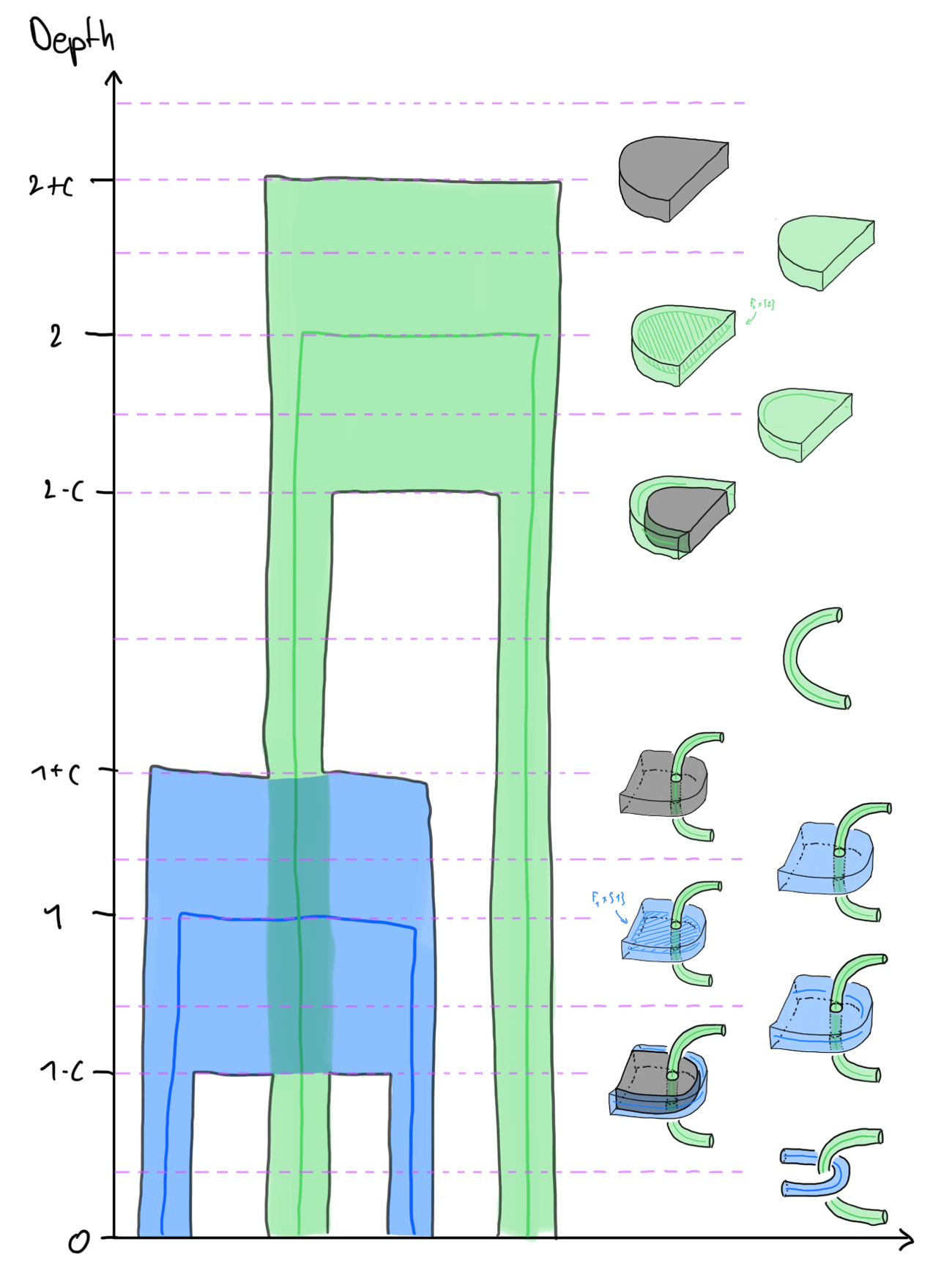}\end{overpic}
\caption{This figure shows~$\onu(\tilde F)$ in reduced dimension (left), and decomposed into three dimensional slices (right). The black and shaded-grey parts correspond to~$\partial\onu(\tilde F_1)\cup\partial\onu(\tilde F_2)$. The colored parts are interior points of~$\onu(\tilde F)$: light filled blue corresponds to~$\onu(\tilde F_1)$ and light filled green to~$\onu (\tilde F_2)$. The parts colored with maximal opacity correspond to~$\tilde F_1$ and~$\tilde F_2$.}
\label{fig:NeighborhoodCombined}
\end{figure}

The following result is well-known, see e.g.~\cite[Example~4.12]{ConwayNagelToffoli} and~\cite[Example~4.15]{Toffoli}.

\begin{lemma}
\label{lemma:HomeoPF}
There is a homeomorphism $h\colon P_F \to \partial \onu(\tilde{F}) \setminus \nu(L)$, where~$P_F$ is the plumbed~$3$-manifold associated to the graph~$\Gamma_F$ with:
\begin{itemize}
  \item two vertices labeled by the surfaces~$F_1$ and~$F_2$;
  \item one edge for each clasp between~$F_1$ and~$F_2$, with sign corresponding to  its contribution to the linking number.
\end{itemize}
\end{lemma}
\begin{construction}
\label{cons:Homeoh}
Later in this section and in Section~\ref{sec:IntersectionForm},  it will be helpful to make the homeomorphism~$h\colon P_F \to \partial \onu(\tilde{F}) \setminus \nu(L)$ more explicit. 
Using our model of $\overline{\nu}(F)$, this can be done via the decomposition
\[
\partial\onu(\tilde{F}) \setminus \nu(L) = 
\left(\partial \onu(\tilde{F}_1) \setminus \left(\nu(\tilde{F}_2)\cup\nu(K_1)\right)\right)
\cup 
\left(\partial \onu(\tilde{F}_2) \setminus \left(\nu(\tilde{F}_1)\cup\nu(K_2)\right)\right)
\]
and homeomorphisms~$\partial \onu(\tilde{F}_i) \setminus \left(\nu(\tilde{F}_j)\cup\nu(K_i)\right)\to F_i^\circ \times S^1$
for~$\{i,j\}=\{1,2\}$.
Indeed, the boundary of~$\onu(\tilde{F}_1)$ around~$f_1 \star \{1\}$ is parametrized by
\[
\partial\onu(\tilde{f}_1):=(f_1 \times J \star \{1 \pm c\}) \cup (f_1 \times \{\pm c\} \star [1-c,1+c])=f_1\star\partial(J\star[1-c,1+c])\,,
\]
which is homeomorphic to $f_1 \times S^1$.
Since each clasp yields a double point in~$(f_1 \star \{1\}) \cap (K_2 \star \{1\})$, the surface~$f^\circ_1:=f_1\setminus\nu(K_2)$ is homeomorphic to~$F_1^\circ$. Observe the equality
\[
\partial \onu(\tilde{f}_1) \setminus \nu(\tilde{F}_2) = 
(f_1^\circ \times J \star \{1 \pm c\}) \cup 
(f_1^\circ \times \{\pm c\} \star [1-c,1+c])\cong f^\circ_1\times S^1\,.
\]
Extending this to~$\partial \onu(K_1) \star [0,1+c]$ yields the explicit homeomorphism
\[
\partial \onu(\tilde{F}_1) \setminus \left(\nu(\tilde{F}_2)\cup\nu(K_1)\right)\cong f_1^\circ \times S^1\cong F_1^\circ \times S^1\,,
\]
which can be visualized in Figure~\ref{fig:NeighborhoodCombined} focusing on the boundary of the blue region.
Similarly, we have an explicit homeomorphism
\[
\partial \onu(\tilde{F}_2) \setminus \left(\nu(\tilde{F}_1)\cup\nu(K_2)\right)\cong F_2^\circ \times S^1\,,
\]
which can be seen in Figure~\ref{fig:NeighborhoodCombined} by inspecting the boundary of the green region.

Finally, each clasp in~$F\subset S^3$ yields a transverse double point in~$\tilde{F}\subset B^4$,
and the corresponding boundary tori in~$\partial \onu(\tilde{F}_1) \setminus \left(\nu(\tilde{F}_2)\cup\nu(K_1)\right)\cong F_1^\circ \times S^1$ and in~$\partial \onu(\tilde{F}_2) \setminus \left(\nu(\tilde{F}_1)\cup\nu(K_2)\right)\cong F_2^\circ \times S^1$ are glued via the identification~\eqref{eq:ident}.
\end{construction}

\medskip

Recall that in the case~$\mu=1$, a~C-complex~$F=F_1$ is nothing but a Seifert surface for~$L=K_1$.
We can therefore simply define the exterior of the pushed-in~C-complex to be~$V_F=B^4\setminus\nu(\tilde{F})$,
with~$\tilde{F}=\tilde{F}_1$ as above, and we have an explicit homeomorphism~$h\colon P_F=F_1\times S^1\to \partial\onu(\tilde{F})\setminus\nu(L)$.

\subsubsection*{Construction of $W_F$}
Starting from a nice~C-complex~$F$ for~$L=K_1\cup K_2$, we now construct the~$4$-manifold~$W_F$ appearing in Theorem~\ref{thm:ExistenceOfW}.  
The idea of this construction is the following.
By Lemma~\ref{lemma:HomeoPF}, the boundary of~$V_F$ decomposes as~$\partial V_F = X_L \cup_h P_F$, whereas our goal is to obtain a manifold whose boundary is~$M_L=X_L \cup P_L$.
There are two differences between~$P_F$ and~$P_L$: first, the graph~$\Gamma_F$ may contain more edges than~$\Gamma_L$ (when $F$ has clasps of opposite signs); second, the surfaces decorating~$\Gamma_F$ may have positive genus, while those of~$\Gamma_L$ are discs. 
To address these discrepancies, we attach appropriate~$4$-manifolds to~$V_F$, first cancelling the extra edges of~$\Gamma_F$ (resulting in an intermediate $4$-manifold $Z_F$), and then surgering away the genera of the surfaces until they become discs (yielding the desired $4$-manifold~$W_F$).

\medskip

\begin{construction}
\label{cons:ZF}
To make the construction of~$Z_F$ precise, consider a nice~C-complex~$F$ for~$L$ endowed with a good basis of~$H_1(F)$, 
as defined in Terminology~\ref{ter:geo-symp-good}. In particular, it contains a family of ``cancelling curves''~$\{\delta_1, \dots, \delta_n\}$, where~$n=\frac{1}{2}(\#\{\text{clasps}\}-\vert\ell\vert)$.
As outlined above, our goal is to modify $V_F$ so as to do away with the excess edges in the plumbing graph of $P_F \subset \partial V_F$.

For each~$1\le k\le n$, let~$e_k$ and~$e_k'$ be the edges of~$\Gamma_F$ (oriented from~$F_1$ to~$F_2$) corresponding to the two clasps traversed by~$\delta_k$. By assumption, these carry opposite signs (say, positive for~$e_k$ and negative for~$e_k'$), and we wish to ``cancel'' them. 
To do so, we use the following construction from~\cite[Lemma~4.9]{ConwayNagelToffoli}. Set 
$$X_k := I \times I \times S^1 \times S^1$$
and consider the tori~$T_{k} = (-\partial D_{e_k}) \times S^1$ and~$T'_{k}= (-\partial D_{e'_k}) \times S^1$, with bicollar neighborhoods~$I \times T_k$ and~$I\times T'_k$ in~$P_F$. Glue~$X_k$ to~$P_F$ via the homeomorphism
\begin{equation*}
f_k\colon \partial I\times I\times S^1\times S^1\to I\times (T_k\sqcup T_k')
\end{equation*}
given by:
\begin{align*}
\{0\} \times I \times S^1 \times S^1 &\longrightarrow I \times (-\partial D_{e_k}) \times S^1 &
\{1\} \times I \times S^1 \times S^1 &\longrightarrow I \times (-\partial D_{e'_k}) \times S^1 \\
(0,t,x,y) &\longmapsto (t,x,y) &
(1,t,x,y) &\longmapsto (t,x^{-1},y)\,.
\end{align*}
Using the homeomorphism~$h\colon P_F\to\onu(\tilde{F}) \setminus \nu(L)=\overline{\partial V_F \setminus X_L}$ from Construction~\ref{cons:Homeoh}
(but suppressing the~$f_k$ from the notation),
we get a~$4$-manifold
\begin{equation}
\label{equ:def-Z}
Z_F : = V_F \cup_{h} \bigcup_{k=1}^{n} X_k
\end{equation}
which inherits the orientation of~$V_F$.
By construction, there is a homeomorphism
\begin{equation}
\label{equ:def-h'}
h' \colon P_{F'} \to \overline{\partial Z_F\setminus X_L}\,,
\end{equation}
where~$P_{F'}$ denotes the plumbed manifold associated to the plumbing graph obtained from~$\Gamma_F$ by removing the edges~$e_1, e_1', \ldots ,e_n, e_n'$ and by performing~$0$-surgeries on the surfaces~$F_1$ and~$F_2$
along the discs~$D_{e_k}, D_{e_k'} \subset F_1$ and~$D_{\overline{e}_k}, D_{\overline{e}_k'} \subset F_2$,  resulting in surfaces~$F_1'$ and~$F_2'$ (see Figure~\ref{fig:Xe}).
Thus, gluing the handles~$X_k$ cancels the $2n$ extra edges of~$\Gamma_F$ at the expense of increasing the genera of~$F_1$ and~$F_2$ (each by~$n$).
Moreover, an explicit such homeomorphism~$h'$ can be obtained from~$h\colon P_F\to\overline{\partial V_F \setminus X_L}$ as follows. First consider the restriction~$h''$ of~$h$ to the complement in~$P_F$ of~$h^{-1}(\bigcup_k\text{Im}(f_k))$: this space can be understood as~$P_{F''}$, with~$F''$ defined as~$F$ but without identifying
the tori corresponding to the edges~$\{e_k,e'_k\}_k$ (see Figure~\ref{fig:Xe}, left). Then, extend~$h''$ to~$h'\colon P_{F'} \to \overline{\partial Z_F\setminus X_L}$ via a parametrisation~$P_F\setminus P_{F''}\cong \bigsqcup_k I\times\partial I\times S^1\times S^1\subset\bigsqcup_k\partial X_k$ whose inverse coincides with the restriction of~$f_k$ to~$\partial I\times \partial I\times S^1\times S^1$.
\end{construction}

\begin{construction}
\label{cons:WF}
Continuing with the notation from Construction~\ref{cons:ZF},  we conclude the construction of $W_F$.
To do so, we address the reduction of genera by performing two sets of~$1$-surgeries along well chosen curves in~$F_1'$ and~$F_2'$. 
 \begin{itemize}
 \item We begin by surgering away the new unwanted genus created during Construction~\ref{cons:ZF}.
 For this,  note that for each cancelling curve~$\delta_k$, there are simple closed curves
$$\delta_k^1 \subset (F_1')^\circ \quad \text{and} \quad \delta_k^2 \subset (F_2')^\circ$$ such that surgery along all~$\delta_k^i$ transforms~$F_i'$ back to~$F_i$. 
There are many such curves but, for the moment, we make an arbitrary choice; a specific choice will be described in Lemma~\ref{lemma:Claim0} below.
For all~$i=1,2$ and~$1\le k\le n$, choose 
pairwise disjoint neighborhoods~$\nu(\delta_k^i)$ of~$\delta_k^i$ in~$(F_i')^\circ$, and set~$Y^i_k:= D^2 \times [-1,1]$.
Glue the ``round $2$-handle"~$Y^i_k\times S^1$ to~$P_{F'}$ via a homeomorphism
\begin{equation}
\label{eq:glue-Y}
\begin{tikzcd}
\partial Y^i_k\times S^1\supset \partial D^2\times [-1,1]\times S^1\arrow[rr,"h_k^i\times \id_{S^1}"]&&\delta_k^i\times S^1\subset F_i'\times S^1\,,
\end{tikzcd}
\end{equation}
where~$h_k^i\colon \partial D^2\times [-1,1]\xrightarrow{\cong}\overline{\nu}(\delta_k^i)$ is required to map the oriented curve~$\partial D^2\times\{0\}$ to~$\delta_k^i$.
 \item  Next, we surger away the genus of the Seifert surfaces $F_1$  and $F_2$.
Recall from Terminology~\ref{ter:geo-symp-good} that the good basis of~$H_1(F)$ contains pairwise disjoint oriented curves~$\{\alpha_1, \ldots, \alpha_{g_1}\}$ in~$F_1$ and~$\{\alpha_{g_1+1}, \ldots, \alpha_{g}\}$ in~$F_2$ disjoint from the clasps,
defining curves
$$\alpha_m \subset (F_1')^\circ\cup (F_2')^\circ.$$
For all~$1\le m\le g$, choose 
pairwise disjoint neighborhoods~$\nu(\alpha_m)$ of~$\alpha_m$ in~$(F_i')^\circ$, disjoint from~$\bigcup_{k,i}\overline{\nu}(\delta_k^i)$.
Set~$Y_m:= D^2 \times [-1,1]$ and glue to the round $2$-handle~$Y_m\times S^1$ to~$P_{F'}$ via
\[
\begin{tikzcd}
\partial Y_m\times S^1\supset \partial D^2\times [-1,1]\times S^1\arrow[rr,"h_m\times \id_{S^1}"]&&\overline{\nu}(\alpha_m)\times S^1\subset F_i'\times S^1\,,
\end{tikzcd}
\]
with~$h_m\colon \partial D^2\times [-1,1]\to\overline{\nu}(\alpha_m)$ a homeomorphism mapping~$\partial D^2\times\{0\}$ to~$\alpha_m$.
\end{itemize}
Combining these two steps and using the homeomorphism~$h'\colon P_{F'} \to \overline{\partial Z_F\setminus X_L}$ induced by~$h$, we get a~$4$-manifold which inherits the orientation of~$Z_F$:
\[
W_F:=Z_F \cup_{h'} 
 \bigcup_{k=1}^{n} \left( \left(Y_k^1 \times S^1\right) \cup \left(Y_k^2 \times S^1\right)\right)
\cup \bigcup_{m=1}^{g} \left(Y_m \times S^1 \right) \,,
\]
This construction is illustrated in Figure~\ref{fig:WF}.
\end{construction}

\begin{figure}
\centering
\begin{overpic}[width=12cm]{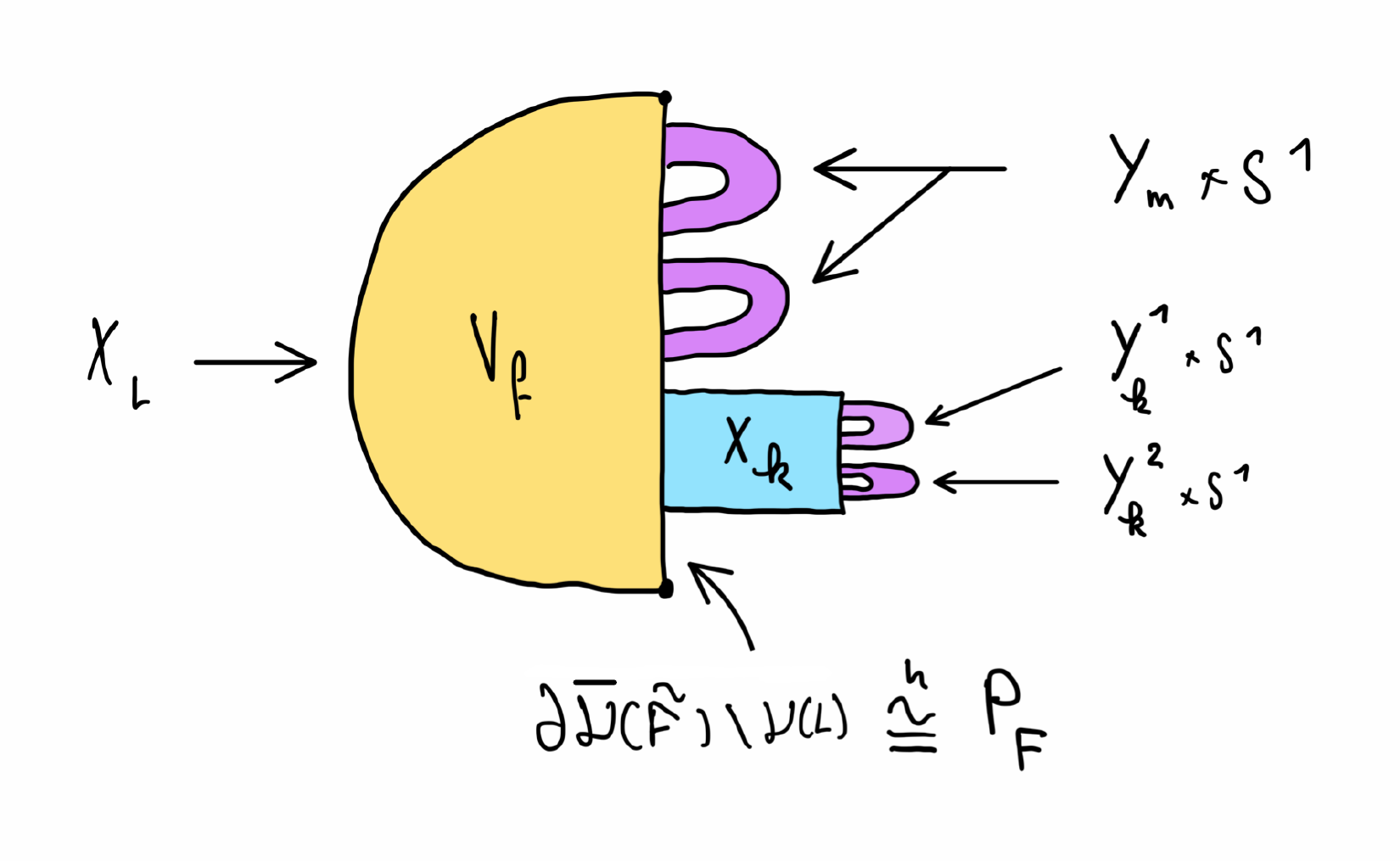}\end{overpic}
\caption{The manifold~$W_F$ is obtained by first attaching the round $2$-handles~$X_k$ to~$V_F$ to cancel the extra edges of $\Gamma_F$, and then attaching the handles~$Y_k^i \times S^1$ and~$Y_m \times S^1$ to cancel the genera.}
\label{fig:WF}
\end{figure}

\begin{remark}
We check that~$\partial W_F$ is homeomorphic to~$M_L$.
By construction, it decomposes as~$\partial W_F = X_L \cup_h P_L$,
where the manifolds~$X_L$ and~$P_L$ are glued along their boundaries via the restriction of the homeomorphism~$h$
to~$\partial P_L=\partial P_F=(\partial F_1\times S^1)\cup(\partial F_2\times S^1)$.
The restriction of~$h$ to~$\partial F_i \times S^1=\partial D_i\times S^1$ maps~$\partial D_i\times\{*\}$
(with~$*\in S^1$) to the Seifert longitude of~$K_i$ in~$\partial X_L$, and~$\{p_i\} \times S^1$ (with~$p_i\in\partial D_i$) to a meridian of $K_i$. Since this is the gluing of~$X_L$ and~$P_L$ used to define~$M_L$, we get an orientation preserving homeomorphism~$\partial W_F\cong M_L$.
\end{remark}

\medskip

\begin{remark}
\label{rem:KoComparison}
In the case~$\mu=1$,  Construction~\ref{cons:WF} boils down to~$W_F=V_F\cup\bigcup_{m=1}^{g_1}\left(Y_m\times S^1\right)$, whose boundary~$\partial W_F\cong X_L\cup P_L=M_L$ yields the~$0$-surgery along the knot~$L$.
This should be contrasted with Litherland and Ko's construction of a filling of $M_L$ which, for a Seifert surface $F$, takes the form~$V_F \cup (\mathcal{H}_g \times S^1)$ with~$\mathcal{H}_g$ a solid handlebody~\cite{LitherlandCobordism,Ko2}.
We also note that our construction depends on the choice of a good basis whereas theirs does not.
\end{remark}

\subsubsection*{Extension of the meridional homomorphism}
Recall that there is an isomorphism 
\[
\Phi_V\colon H_1(V_F) \stackrel{\cong}{\longrightarrow} \Z^\mu
\]
sending the meridian of~$\tilde F_i$ to~$t_i$ (see e.g.~\cite[Proposition 3.1]{CFT18}).
Furthermore, composing the homomorphism~$H_1(P_F)\to H_1(V_F)$ induced by~$P_F \stackrel{h}{\to}\partial \onu(\tilde{F}) \setminus \nu(L)\subset V_F$ with~$\Phi_V$ yields a homomorphism
\begin{equation}
\label{equ:phiF}
\varphi_F \colon H_1(P_F) \to \Z^\mu\,,
\end{equation}
which is meridional.
Lemmas~\ref{lemma:Claim0} and~\ref{lemma:Claim1} below ensure that we can perform the surgeries described above
on cycles that are in the kernel of~$\varphi_F$, thus allowing for the homomorphism~$\Phi_V$ to extend to~$H_1(W_F)$.

As a first step, consider the~$4$-manifold
\begin{equation}
\label{equ:U}
U:=V_F \cup \bigcup_{k=1}^n \Big(X_k \cup \left(Y_k^1 \times S^1\right)\cup \left(Y_k^2 \times S^1\right)\Big)\,,
\end{equation}
which depends on the choice of the attaching curves~$\delta_k^i$. In the case~$\mu=1$, we have~$U=V_F$, and
the conclusion of the following lemma holds trivially; we therefore assume~$\mu=2$.

\begin{lemma}
\label{lemma:Claim0}
For each cancelling curve~$\delta_k \subset F$, there exist simple closed curves~$\delta_k^1$ in~$(F_1')^\circ$ and~$\delta_k^2$ in~$(F_2')^\circ$ with the following properties:
\begin{itemize}
\item For each~$i$, surgery along all~$\{\delta_k^i\}_k$ transforms~$F_i'$ into~$F_i$.
\item The isomorphism~$\Phi_V\colon H_1(V_F)\to\Z^2$ extends to~$\Phi_U\colon H_1(U)\to\Z^2$, with~$U$
obtained as in~\eqref{equ:U} using the attaching curves~$\delta_k^i$.
\end{itemize}
\end{lemma}

\begin{proof}
We begin with the construction of the loops~$\delta_k^i \subset (F_i')^\circ$ for $i=1,2$.
Before diving into the details, we outline the idea.
Recall from Constructions~\ref{cons:ZF} and~\ref{cons:WF} that during the construction of $Z_F = V_F \cup_h \bigcup_{k=1}^n X_k$,  hollow tubes were added to the $F_i$ resulting in the new surfaces~$F_i'$.
In essence, the loops~$\delta_k^i$ are obtained by joining an arc running along the tube with an arc with in~$F_i^\circ$; see Figure~\ref{fig:Xe}.
In order to extend~$\Phi_V \colon H_1(V_F) \to \Z^2$ over~$H_1(U)$, we need show that it vanishes on (a specific choice of) these loops and, for this reason,  it is necessary to describe them in detail.

First, we may assume that the oriented simple closed cancelling curve~$\delta_k=:\delta\subset F$ behaves well near the clasps, as illustrated on the left of Figure~\ref{fig:TildeDelta}.
This makes it possible to define a pushoff~$\delta^{++}\subset\partial\onu(F)$ which lies in the surface
\[
(F_1\times \{\pm c\}\setminus \nu (F_2)) \cup (F_2 \times \{\pm c\} \setminus \nu(F_1))\,.
\]
As illustrated on the right hand side of Figure~\ref{fig:TildeDelta}, pushing the C-complex~$F$ into~$B^4$,  the cycle~$\delta^{++}$ is mapped to the cycle
\begin{equation}
\label{eq:wtdelta}
\tilde \delta\subset\partial \onu(\tilde F)\setminus \nu(L)
\end{equation}
obtained as the concatenation of the following four oriented paths:
\[
(\delta^{++} \cap  (F_1 \times \{c\})) \star\{1 + c\}\,,\quad \partial(\delta^{++} \cap (F_1 \times \{c\}))  \star [1+c, 2+c] \,,\quad (\delta^{++} \cap (\overline{F_2\setminus(F_1 \times J)} \times \{c\})) \star\{2 + c\}\,.
\]
Note that the middle formula defines two paths in Figure~\ref{fig:TildeDelta}.
\begin{figure}
\centering
\begin{overpic}[width=15cm]{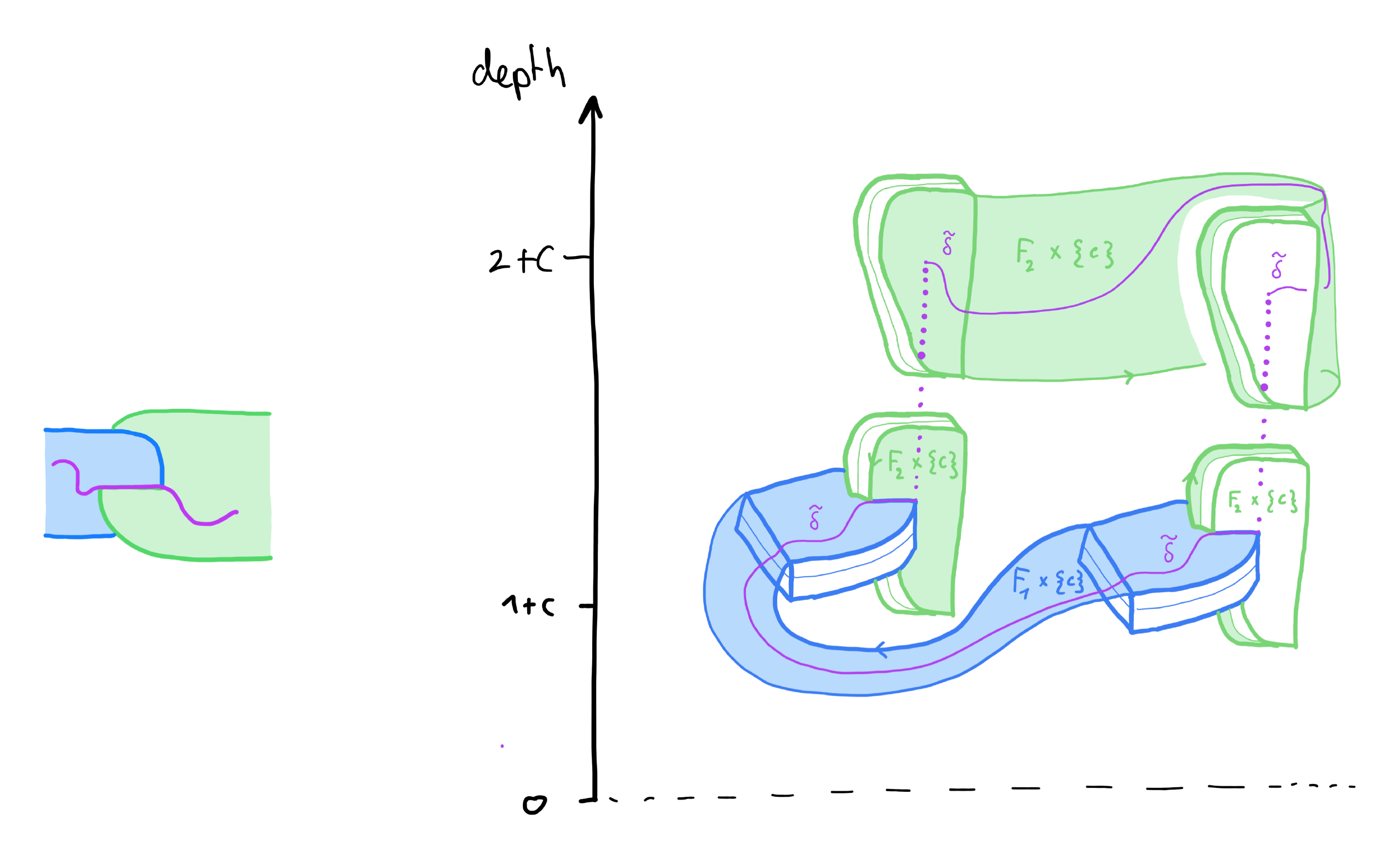}\end{overpic}
\caption{Left: a curve behaving well near a clasp. Right: a schematic picture of $\tilde \delta$.}
\label{fig:TildeDelta}
\end{figure}

Recall the homeomorphism~$h\colon P_F \to \partial \onu(\tilde{F}) \setminus \nu(L)$ from Lemma~\ref{lemma:HomeoPF},  and consider the cycle~$h^{-1}(\tilde \delta) \subset P_F$.
For~$i=1,2$, this cycle intersects~$F_i^\circ\times S^1$ along the fixed~$S^1$-coordinate~$y_i\in S^1$ corresponding to~$(c,i+c)\in\partial(J\star[i-c,i+c])$.
Up to reparametrization of~$S^1$, it may be assumed  that~$y_1=y_2=:y$.
By definition of a cancelling curve,
this cycle intersects~$T_k=\partial D_{e_k}\times S^1$ (resp.~$T'_{k}= \partial D_{e'_k} \times S^1$) once.
We write~$(x_i,y)$  (resp.~$(x_i',y)$) for the intersection of~$h^{-1}(\tilde\delta)$ with the boundary of the collar neighborhood of~$T_k$ (resp.~$T'_k$) in~$F_i^\circ\times S^1$, see Figure~\ref{fig:Xe}.
Up to reparametrization of~$\partial  D_{\overline{e}_k}\cong S^1$ and~$\partial  D_{\overline{e}'_k}\cong S^1$, it may be assumed  that~$x_1=x_2=:x$ and~$x'_1=x'_2=:x'$.

For~$i=1,2$, consider the simple oriented closed curve~$\delta^i$ in~$(F_i')^\circ$ defined by
\begin{equation}
\label{eq:deltai}
\delta^i\times\{y\}=\left(h^{-1}(\tilde \delta)\cap(F_i^\circ \times S^1)\right)\cup f_k(\gamma^i)\,,
\end{equation}
where~$\gamma^1$ is a simple oriented path in~$I\times\{0\}\times S^1\times\{y\}\subset \partial X_k$ joining~$(0,0,x,y)$ and $(1,0,(x')^{-1},y)$, the path~$-\gamma^2$ is the parallel copy of~$\gamma^1$ in~$I\times\{1\}\times S^1\times\{y\}\subset \partial X_k$, and the orientations are chosen to be compatible with the one of~$h^{-1}(\tilde \delta)$. This is illustrated in Figure~\ref{fig:Xe}.

\begin{figure}[tb]
\centering
\begin{overpic}[width=\textwidth]{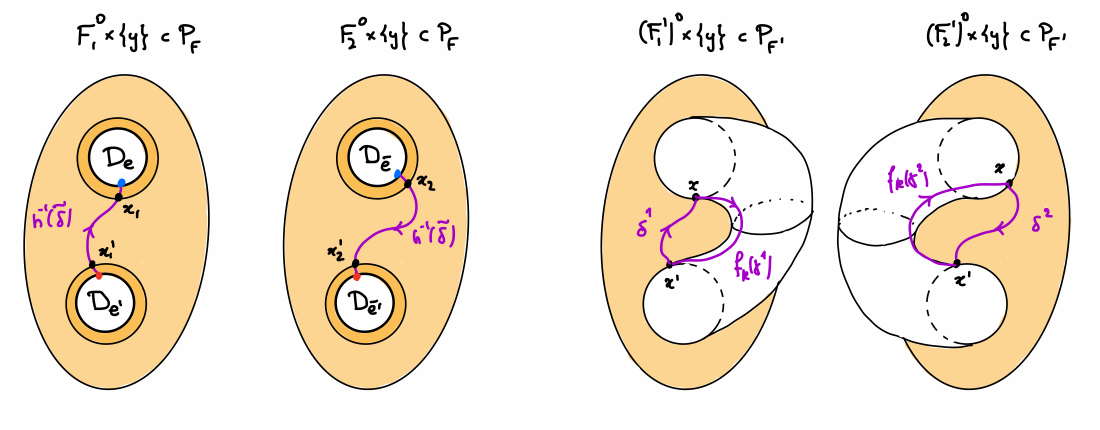}\end{overpic}
\caption{Left: the cycle~$h^{-1}(\tilde \delta)\subset P_F$ decomposed into two arcs. Right: the resulting cycles~$\delta^1$ and~$\delta^2$. }
\label{fig:Xe}
\end{figure}

This concludes the construction of the curves $\{ \delta_k^i\}$ for $i=1,2$.
By definition, the surface obtained from~$F_i'$ by surgery along all curves~$\delta^i_k$ is canonically homeomorphic to $F_i$,
proving the first item of the proposition. 

We now turn to the second item.
Recall the~$4$-manifold~$Z_F$ defined in~\eqref{equ:def-Z} and the homeomorphism~$h' \colon P_{F'} \to \overline{\partial Z_F\setminus X_L}$ from~\eqref{equ:def-h'}.
Since the paths~$\gamma^1,\gamma^2$ are antiparallel, the cycle
\[
\gamma^1+\gamma^2-\left(\{0\}\times I\times\{x\}\times\{y\}\right)+\left(\{0\}\times I\times\{(x')^{-1}\}\times\{y\}\right)
\]
bounds a rectangle in~$X_k$, as illustrated in Figure~\ref{fig:rectangle}.
Finally, observe that the two paths in the middle of the formula defining~$\tilde\delta$ are given by~$\{x\}\star[1+c,2+c]$
and~$\{(x')^{-1}\}\star[1+c,2+c]$; these are homotopic rel boundary in~$\partial V_F$ to the paths~$f_k(\{0\}\times I\times\{x\}\times\{y\})$ and~$f_k(\{0\}\times I\times\{(x')^{-1}\}\times\{y\})$, respectively.
This leads to the equality
\begin{equation}
\label{equ:deltak}
[\tilde \delta]=[h'(\delta^1\times\{y\})]+[h'(\delta^2\times\{y\})]\in H_1(Z_F)\,,
\end{equation}
as illustrated in Figure~\ref{fig:rectangle}.
We claim the isomorphism~$\Phi_V \colon H_1(V_F) \to \Z^2$ 
extends to a homomorphism~$\Phi_Z \colon H_1(Z_F) \to~\Z^2$
mapping~$[h'(\delta^i \times \{y\})]$ to~$1\in \Z^2$ for~$i=1,2$ and all~$y\in S^1$.
As explained in the proof of~\cite[Lemma~4.9]{ConwayNagelToffoli}, the unique homomorphism~$\Phi_X \colon H_1(X_{k}) \to \Z^2$ mapping~$[\{\ast\} \times \{t\} \times S^1 \times \{y\}]$ to~$t_1$ and~$[\{\ast\} \times \{t\} \times \{x\} \times S^1]$ to~$t_2$ is compatible with the meridional homomorphism~$\varphi_F\colon H_1(P_F) \to\Z^2$ defined in~\eqref{equ:phiF}, so the isomorphism~$\Phi_V\colon H_1(V_F) \to \Z^2$ extends to 
$$H_1(Z_F) \to \Z^2$$
This extension is not unique: the Mayer--Vietoris exact sequence for~$Z_F=V_F\cup\bigcup_k X_k$ shows it can take arbitrary values on some homology classes, which can be chosen to be those of~$h'(\delta^1 \times \{y\})$. Hence, we can choose the extension~$\Phi_Z \colon H_1(Z_F) \to \Z^2$ to satisfy
$$\Phi_Z([h'(\delta^1\times\{y\})]) = 1.$$
It remains to verify that~$\Phi_Z([h'(\delta^2\times\{y\})])=1$.
Consider the space~$(F_1\star[0,1])\cup(F_2\star[0,2])\subset B^4$, which should be thought of
as the ``trace'' of the pushed-in C-complex. 
Observe that its exterior
in $B^4$, is homeomorphic to a topological ball in~$V_F$ (see~\cite[Section~3.1]{CFT18}) that we denote
\[
B:=B^4\setminus((F_1 \times J \star[0,1+c])\cup(F_2 \times J \star [0,2+c]))\,.
\]
Since (the boundary of) this $4$-ball contains~$\tilde \delta$, we have~$[\tilde \delta]=0\in H_1(V_F)$
and~$\Phi_V([\tilde \delta]) = 1\in\Z^2$. 
Since~$\Phi_Z$ extends~$\Phi_V$, the equality~\eqref{equ:deltak} implies
\[
1 = \Phi_V([\tilde \delta]) = \Phi_Z([\tilde \delta]) = \Phi_Z([h'(\delta^1\times\{y\})]) \, \Phi_Z([h'(\delta^2\times\{y\})]) = \Phi_Z([h'(\delta^2\times\{y\})])\,,
\]
This concludes the proof of the claim.

\begin{figure}[tb]
\centering
\begin{overpic}[width=10cm]{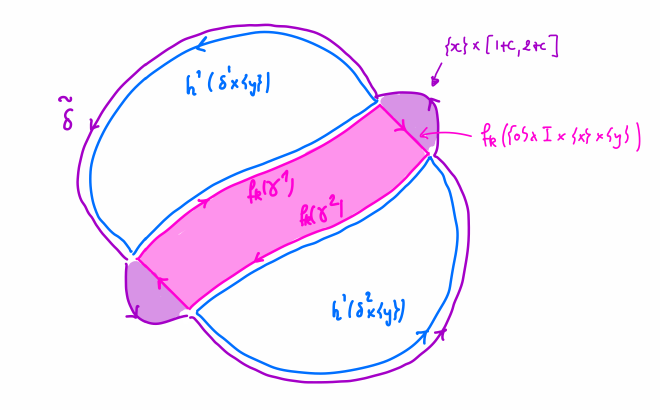}\end{overpic}
\caption{Schematic justification of $[\tilde \delta]=[h'(\delta^1\times\{y\})]+[h'(\delta^2\times\{y\})]$.}
\label{fig:rectangle}
\end{figure}

Since the extension~$\Phi_Z$ of~$\Phi_V$ constructed in the claim takes trivial values on the attaching curves~$h'(\delta^i\times\{y\})$, a Mayer-Vietoris argument shows that it extends further to~$\Phi_U\colon H_1(U)\to~\Z^2$.
\end{proof}

We are now ready to extend the coefficient system over $W_F$ which, recall, is defined as
\[
W_F=U\cup\bigcup_{m=1}^{g} \left(Y_m \times S^1\right)\,,
\]
where the round $2$-handle~$Y_m\times S^1=D^2\times [-1,1]\times S^1$ is attached via
\[
\begin{tikzcd}
\partial Y_m\times S^1\supset \partial D^2\times [-1,1]\times S^1\arrow[rr,"h_m\times \id_{S^1}"]&&\nu(\alpha_m)\times S^1\subset F_i'\times S^1
\end{tikzcd}
\]
with~$h_m$ a homeomorphism mapping~$\partial D^2\times\{0\}$ to the genus curves~$\alpha_m$.

\begin{lemma}
\label{lemma:Claim1}
The homomorphism~$\Phi_U\colon H_1(U)\to\Z^\mu$ extends to~$\Phi\colon H_1(W_F)\to\Z^\mu$.
\end{lemma}

\begin{proof}
We claim that given any oriented closed curve~$\alpha$ in~$F_i$ that is disjoint from the clasps, and any~$(s,t)\in\partial(J\star [i-c,i+c])$, the oriented closed curve~$\alpha \times \{s\} \star \{t\} \subset \partial \overline{\nu} (\tilde F)$ lies
in the kernel of~$\Phi_V\colon H_1(V_F)\to\Z^2$.
Let us first assume that~$t=i\pm c$ and~$s\ge 0$.
Since the cylinder~$\alpha\times[0,s]\star\{i\pm c\}$ is contained in~$\partial \overline{\nu}(\tilde F)\subset V_F$, the cycle~$\alpha \times \{s\} \star \{i\pm c\}$ is homologous in~$V_F$ to the cycle~$\alpha \times \{0\} \star \{i\pm c\}$, 
which is included in the exterior of the trace of the pushed in C-complex (recall the end of the proof of Lemma~\ref{lemma:Claim0}). 
Since this is a topological ball, we have
\[
\Phi_V([\alpha \times \{s\} \star \{i\pm c\}]) = \Phi_V([\alpha \times \{0\} \star \{i\pm c\}]) = 1\,.
\]
The other cases can be treated in a similar way.
This concludes the proof of the claim.

Since~$\Phi_U$ extends~$\Phi_V$, these classes also lie in the kernel of~$\Phi_U$. In particular, this holds true for
the attaching curves~$\alpha_m\times \{y\}\subset F'_i\times S^1$. The statement now follows from Mayer-Vietoris.
\end{proof}

\begin{remark}
When working with classical invariants of a surface~$F \subset X$ in a simply-connected~$4$-manifold with~$\pi_1(X\setminus \nu(F)) \cong \Z$,  it usually becomes necessary to choose a framing~$\overline{\nu}(F) \cong F\times D^2$ with the property that push offs of curves on $F$ are nullhomologous in $X$.
In the case of pushed-in Seifert surfaces,  explicitly parametrising~$\overline{\nu}(F)$ as in Construction~\ref{cons:parametrizationnu(F)} serves this purpose.
This is the underlying reason for which Lemma~\ref{lemma:Claim1} holds.
\end{remark}

\subsubsection*{Proof of the first item of Theorem~\ref{thm:ExistenceOfW}}
Let~$W_F$ be the~$4$-manifold described in Construction~\ref{cons:WF}, and let~$\Phi\colon\pi_1(W_F)\to\Z^\mu$ be the homomorphism
defined in Lemma~\ref{lemma:Claim1} (precomposed with the Hurewicz map). We now check that~$\Phi$ is an isomorphism which restricts to a meridional
homomorphism on~$M_L$.

The proof that~$\Phi$ is an isomorphism takes the most effort but the idea can be described succintly.
First,  since the $C$-complex~$F$ is connected,~$\pi_1(V_F)$ is freely generated by the meridians of the~$F_i$.
Next, when forming~$W_F$ as in Constructions~\ref{cons:ZF} and~\ref{cons:WF},  adding the~$X_k$ creates loops~$\delta_k^1,\delta_k^2$ that are then killed by adding the round $2$-handles~$Y_k \times S^1$; the~$S^1$-factors are identified with the aforementioned meridians.
Finally, adding the round $2$-handles~$Y_m \times S^1$ kills push offs of the genus curves~$\alpha_i$ (which, it turns out are already nullhomologous in~$V_F$) while again preserving the meridians.
Making this precise relies on the Seifert-Van Kampen theorem together with the following lemma, whose proof we leave to the reader.

\begin{lemma}
\label{lemma:Pushout}
Consider the following commutative diagrams of groups, where~$(D, \varphi_1, \varphi_2)$ is a pushout of~$C\stackrel{j}{\leftarrow}A\stackrel{i}{\rightarrow}B$, and~$\varphi\colon D\to G$ is given by the corresponding universal property:
\[
\begin{tikzcd}
A \arrow[r, "i"] \arrow[d, "j"'] & B \arrow[d, "\varphi_2"] \arrow[ddr, bend left=20, "g"] & \\
C \arrow[r, "\varphi_1"'] \arrow[drr, bend right=20, "f"'] & D \arrow[dr, "\varphi"] & \\
& & G\,.
\end{tikzcd}
\]
If~$i$ is surjective and~$f$ is an isomorphism, then~$\varphi \colon  D \to G$ is an isomorphism.\qed
\end{lemma}

We now proceed with the proof that~$\Phi\colon\pi_1(W_F)\to\Z^\mu$ is an isomorphism. 
For any non-negative integer~$0\le N\le n$, consider
the~$4$-manifold
\[
U_N:=V_F \cup \bigcup_{k=1}^N \Big(X_k \cup \left(Y_k^1 \times S^1\right)\cup \left(Y_k^2 \times S^1\right)\Big)\,.
\]
In order to make sense of this union,
recall that by~\eqref{eq:glue-Y}, the round $2$-handle~$Y^i_k\times S^1=D^2\times [-1,1]\times S^1$ is glued to~$X_k=I\times I\times S^1\times S^1$ via the restriction of~$h_k^i\times\id_{S^1}$ to the semi-circle in~$\partial D^2$ corresponding to~$\delta_k^i\cap X_k$, see Figure~\ref{fig:Xe}.

Note that~$U_0=V_F$ and~$U_n=U$. Note also that~$n=0$ in the case~$\mu=1$, so we assume~$\mu=2$.
Let us check that the restriction~$\Phi_N\colon \pi_1(U_N)\to\Z^2$ of~$\Phi$
is an isomorphism, inductively on~$N\ge 0$.
The base case follows from~\cite[Proposition~3.1]{CFT18}, so let us assume that~$\Phi_{N-1}$ is an isomorphism.
The induction step relies on the Seifert-Van Kampen theorem for the union
\[
U_N=U_{N-1}\cup \big( X_{N} \cup (Y_{N}^1 \times S^1) \cup (Y_{N}^2 \times S^1) \big)\,,
\]
which can be used since the intersection is connected:
\[
U_{N-1} \; \cap \; \big ( X_{N} \cup (Y_{N}^1 \times S^1) \cup (Y_N^2 \times S^1) \big) = (T_N \times I) \cup (T'_N \times I) \cup (h^{-1}(\tilde \delta_N) \times I \times S^1)=:A.
\]
By construction of~$\Phi_U$ (recall Lemma~\ref{lemma:Claim0}), the inclusion induced homomorphisms fit in the following commutative diagram where the square is a pushout of~$(\pi_1(A),i_\#,j_\#)$ by the Seifert-Van Kampen theorem:
\begin{equation}
\label{eq:Induction}
\begin{tikzcd}
\pi_1(A) \arrow{d}[swap]{j_\#} \arrow{r}{i_\#} & \pi_1\left(X_N \cup (Y_N^1 \times S^1) \cup (Y_{N}^2 \times S^1)\right) \arrow{d} \arrow[bend left =25]{ddr}{} & \\
\pi_1(U_{N-1}) \arrow{r} \arrow[bend right=10, swap]{drr}{\Phi_{N-1}}[swap]{\cong} & \pi_1\left(U_N \right) \arrow{dr}{\Phi_N} & \\
& & \Z^2\,,
\end{tikzcd}
\end{equation}
We now return to the diagram in~\eqref{eq:Induction}.
The map~$i_\#$ being clearly surjective and~$\Phi_{N-1}$ being an isomorphism by induction,
Lemma~\ref{lemma:Pushout} implies that~$\Phi_N$ is an isomorphism, completing the induction step.

Now, given any integer~$0\le M\le g:=\sum_{i=1}^\mu g_i$, consider the~$4$-manifold
\[
W_M:=U\cup\bigcup_{m=1}^M\left(Y_m\times S^1\right)\,.
\]
One readily checks that the restriction~$\Phi_M$ of~$\Phi$ to~$\pi_1(W_M)$
is an isomorphism, inductively on~$M\ge~0$: the base case was proved above (as~$W_0=U$), and the induction step follows from Seifert-Van Kampen together with Lemma~\ref{lemma:Pushout} applied to the pushout diagram
\begin{equation}
\label{eq:OtherPushout}
\begin{tikzcd}
\pi_1(\nu(\alpha_M) \times S^1) \arrow{d} \arrow{r} & \pi_1(Y_M \times S^1) \arrow{d} \arrow[bend left =15]{ddr} & \\
\pi_1(W_{M-1}) \arrow{r} \arrow[bend right=10, swap]{drr}{\Phi_{M-1}}[swap]{\cong} & \pi_1(W_M) \arrow{dr} {\Phi_M}& \\
& & \Z^\mu\,.
\end{tikzcd}
\end{equation}
Since~$W_{g}=W_F$ and~$\Phi_{g}=\Phi$, this concludes the proof that~$\Phi\colon\pi_1(W_F)\to\Z^\mu$
is an isomorphism.

It remains to check that the restriction of~$\Phi$ to~$\partial W_F=W_L$ yields a meridional homomorphism~$\varphi\colon H_1(M_L)\to\Z^\mu$, i.e.
that the restriction of~$\varphi$ to~$P_L$
maps the class of the loop~$\{p\}\times S^1$ with~$p\in D_i$ to~$t_i$. Via the commutative diagram
\begin{center}
\begin{tikzcd}
H_1(P_L) \arrow{r} \arrow{d}& H_1(W_F) \arrow{d}{\Phi} \\
H_1(M_L) \arrow{r}[swap]{\varphi} &  \Z^\mu\,,
\end{tikzcd}
\end{center}
we are left with the proof that the composition~$H_1(P_L)\to H_1(W_F) \stackrel{\Phi}{\to} \Z^{\mu}$ satisfies this property.
For this, recall from the construction of~$W_F$ that~$D_i$ is obtained from~$F_i$ performing~$0$-surgeries
followed by~$1$-surgeries along curves in the resulting surface~$F_i'$. Via this construction, the loop~$\{p\} \times S^1$ with~$p\in D_i$ is obtained from~$\{p'\} \times S^1$ with the corresponding~$p'\in F_i$. Since~$\Phi$ extends~$\Phi_V$ and~$h(\{p'\} \times S^1)$ is a meridian of~$\tilde F_i$ in~$V_F$, we get
\[
\Phi([\{p\} \times S^1]) = \Phi_V([h(\{p'\} \times S^1)]) = t_i\,,
\]
concluding the proof of Theorem~\ref{thm:ExistenceOfW}~(1).

\medskip

We end this section by recording an additional result for later use. Recall that~$Q$ denotes the field of fractions
of the ring~$\Lambda=\Z[\Z^\mu]$.

\begin{proposition}
\label{lem:V-WoverQ}
The inclusion $V_F \to W_F$ induces isomorphisms~$H_*(V_F; Q) \to H_*(W_F; Q)$.
\end{proposition}

The proof relies on the following well known lemma.

\begin{lemma}
\label{lemma:TwistedHomologyxS1} 
Let~$X$ be a connected CW-complex and let~$\varphi \colon H_1(X\times S^1) \to \Z^\mu$ be a homomorphism.
If the composition~$H_1(S^1)\to H_1(X\times S^1)\stackrel{\varphi}{\to}\Z^\mu$ maps a generator of~$H_1(S^1)$ to a non-trivial element of~$\Z^\mu$, then the~$\Lambda$-module~$H_\ast(X\times S^1; \Lambda)$ is torsion.
\end{lemma}

\begin{proof}[Proof of Lemma~\ref{lemma:TwistedHomologyxS1}]
Write~$z$ for the non-trivial element of~$\Z^\mu$ mentioned in the statement, and let~$\Lambda'$ denote
the localised ring~$\Z[\Z^\mu][(z-1)^{-1}]$.
Straightforward computations show that the chain complex~$C_*(X\times S^1;\Lambda')$ is acyclic,
see e.g.~\cite[Example~2.7]{Nicolaescu} and~\cite[Lemma~2.2]{CFT18}.
Since~$\Lambda'$ is flat, this yields~$0=H_*(X\times S^1;\Lambda')\cong H_*(X\times S^1;\Lambda)\otimes_\Lambda\Lambda'$, so~$H_\ast(X\times S^1; \Lambda)$ is~$(z-1)$-torsion.
\end{proof}

\begin{proof}[Proof of Proposition~\ref{lem:V-WoverQ}]
Recall that the manifold~$W_F$ is obtained from~$V_F$ via
\[
W_F=V_F \cup \bigcup_{k=1}^n \left(X_k \cup (Y_k^1 \times S^1)\cup (Y_k^2\times S^1)\right)\cup \bigcup_{m=1}^{g} \left(Y_m \times S^1\right)
\]
with~$X_k=I\times I\times S^1\times S^1$, and the homomorphism~$\Phi\colon H_1(W_F)\to\Z^\mu$ maps
each of the~$S^1$-factors appearing above to some~$t_i$.
By Lemma~\ref{lemma:TwistedHomologyxS1} and flatness of~$Q$ over~$\Lambda$,
each of these~$n+g$ spaces is~$Q$-acyclic.
Moreover, each of these spaces is glued along a connected subspace satisfying the hypothesis of
Lemma~\ref{lemma:TwistedHomologyxS1}, so these intersections are also~$Q$-acyclic.
The result now follows from an inductive use of the Mayer-Vietoris exact sequence for homology with~$Q$-coefficients.
\end{proof}

\subsection{The manifold~$W_{F,F'}$}
\label{sub:WFF'}

The goal of this section is to construct an explicit filling of~$M_{L,L'}$ over~$\Z^\mu$.
This is the manifold we will use to calculate~$\lambda(L,L'):=\lambda(M_{L,L'},\varphi)$,  in order to prove the fourth and the fifth items of Theorem~\ref{thm:ConcordanceInvarianceGeneralIntro}.

\medbreak

In order to construct this filling, we begin with the following lemma.

\begin{lemma}
\label{lem:HomeoPL}
Fix $\mu \leq 2$ and let~$L$ be a~$\mu$-component link.
For any meridional homomorphisms~$\varphi,\varphi'\colon H_1(P_L) \to \Z^\mu$, there exists a homeomorphism~$ f \colon P_L \to P_L$ such that~$\varphi' \circ  f_*=\varphi$.
\end{lemma}
\begin{proof}
First note that if~$\mu=1$, then~$P_L=D^2\times S^1$ admits a unique meridional homomorphism and the statement holds trivially. This is also the case when~$\mu=2$ and the linking number vanishes. Without loss of generality, we may therefore assume that~$L=K_1\cup K_2$ is a two-component link with~$\ell:=\vert\lk(K_1,K_2)\vert>0$.
We start by explaining the idea of the proof before diving into the details.
By Lemma~\ref{lemma:MeridionalPL}, two meridional homomorphisms~$\varphi, \varphi'\colon H_1(P_L)\to\Z^2$
agree on meridians, but might differ on big loops coming from the graph~$\Gamma_L$.
Note that~$\Gamma_L$ consists of two vertices and~$\ell$ edges between them, yielding~$\ell-1$ independent such big loops.
The idea is to define~$ f$ as
the appropriate composition of (Dehn twists $\times\id_{S^1}$) about curves around the punctures in~$D^\circ_1\sqcup D_2^\circ$ to obtain~$\varphi'\circ  f_*=\varphi$.
There is a sufficiently large number of such curves (namely~$2\ell$) compared to the number of~$\Z$-coordinates of~$\varphi$ on the big loops (namely~$2(\ell-1)$)
for this problem to have a solution, a favorable situation not met in general for~$\mu>2$ components.

We now work out the details.
Let~$\varphi, \varphi' \colon H_1(P_L) \to \Z^2$ be meridional homomorphisms.
Recall that~$\Gamma_L$ consists of two vertices decorated by discs~$D_1,D_2$ and~$\ell$ edges.
Let us denote by~$e_0,\dots,e_{\ell-1}$ these edges ordered arbitrarily and oriented from~$D_1$ to~$D_2$.
For~$0\le r\le \ell-1$, let~$\tau^{r,1}\colon  D^\circ_{1}\to D^\circ_{1}$ (resp.~$\tau^{r,2}\colon  D^\circ_{2}\to D^\circ_{2})$ denote the positive Dehn twist about a simple closed curve parallel to~$\partial D_{e_r}$ (resp.~$\partial D_{\overline{e}_r}$). For~$i=1,2$ and~$0\le r\le \ell-1$, write~$ f^{r,i}$ for the orientation-preserving self-homeomorphism of~$P_L$ obtained by extending~$\tau^{r,i}\times\id_{S^1}$ by the identity outside~$D^\circ_{i}\times S^1$.
These homeomorphisms~$ f^{r,i}$ will be the building blocks of~$ f$.
Fix an embedding~$\alpha\colon\Gamma_L\hookrightarrow P_L$ as defined above Lemma~\ref{lemma:MeridionalPL}.
By this lemma,
it induces an injective homomorphism~$\alpha_*\colon H_1(\Gamma_L)\to H_1(P_L)$.
Therefore, the classes~$\gamma_1,\dots,\gamma_{\ell-1}$ in~$H_1(P_L)$ given by~$\gamma_r=[\alpha(e_0)-\alpha(e_{r})]$ for~$1\le r\le\ell-1$ form a basis of~$\alpha_*(H_1(\Gamma_L))\subset H_1(P_L)$.
By construction, we have
\[
\varphi( f^{r,1}_*(\gamma_{s}))=
\begin{cases}
t^{\sgn(e_r)}_{2}\varphi(\gamma_{s}) & \text{if~$r=s$,}\\
\varphi(\gamma_{s}) & \, \text{else,}
\end{cases}
\quad\text{and}\quad
\varphi( f^{r,2}_*(\gamma_{s}))=
\begin{cases}
t^{-\sgn(e_r)}_{1}\varphi(\gamma_{s}) & \text{if~$r=s$,}\\
\varphi(\gamma_{s}) & \, \text{else,}
\end{cases}
\]
for all~$1\le r,s\le \ell-1$.
The case~$i=1$ is illustrated in Figure~\ref{fig:twist}.
This readily implies that the appropriate composition~$ f$ of the orientation-preserving homeomorphisms~$ f^{r,i}\colon P_L\to P_L$ and their inverses satisfies~$\varphi( f_*(\gamma_s))=\varphi'(\gamma_s)$ for all~$1\le s\le \ell-1$.
By Lemma~\ref{lemma:MeridionalPL}, this implies that~$\varphi\circ  f_*$ and~$\varphi'$ coincide on~$H_1(P_L)$.
\end{proof}

\begin{figure}[h]
\centering
\begin{overpic}[width=10cm]{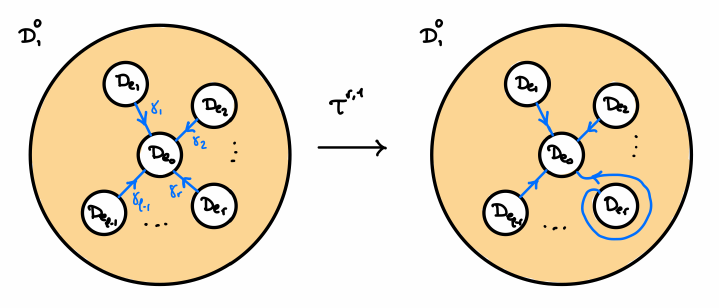}\end{overpic}
\caption{The Dehn twist~$\tau^{r,1}$ on~$D^\circ_1$ making up the homeomorphism~$f^{r,1}$ of~$P_L$.}
\label{fig:twist}
\end{figure}

\begin{construction}
\label{cons:WFF}
Fix $\mu \leq 2$, let~$F$ and~$F'$ be C-complexes for~$\mu$-component links~$L$ and~$L'$ with equal linking numbers,   and consider the~$4$-manifolds~$W_F$ and~$W_{F'}$ from Theorem~\ref{thm:ExistenceOfW}.
The isomorphisms~$\pi_1(W_F) \cong \Z^\mu \cong \pi_1(W_{F'})$ restrict to meridional homomorphisms on the subsets of~$\partial W_F,\partial W_{F'}$ that are homeomorphic to~$P_L$, i.e. on~$\partial W_F \setminus X_L$ and~$\partial W_{F'} \setminus X_{L'}$.
Use Lemma~\ref{lem:HomeoPL} to choose a homeomorphism~$f \colon \partial W_F \setminus X_L \to \partial W_{F'} \setminus X_{L'}$ that intertwines these coefficients systems.
Consider the~$4$-manifold with boundary given by
$$W_{F,F'}:=W_F \cup_h -W_{F'}.$$
Observe that by construction~$\partial W_{F,F'} \cong M_{L,L'}$ as manifolds over~$\Z^\mu$.
\end{construction}

The next proposition describes the~$Q$-intersection form of~$W_{F,F'}$ and, as a consequence,  leads to a calculation of the homology surgery invariant of~$M_{L,L'}$,  proving the fourth and fifth items of Theorem~\ref{thm:ConcordanceInvarianceGeneralIntro}.

\begin{theorem}
\label{thm:Calculation}
Fix~$\mu\le 2$. Let~$L$ and~$L'$ be~$\mu$-component links with $\Delta_L,\Delta_{L'} \neq 0$, let~$F$ and~$F'$ be nice connected~C-complexes for~$L$ and~$L'$ respectively, and let~$H_F$ and~$H_{F'}$ be a choice of matrices as in Construction~\ref{cons:MatrixH}.
If~$L$ and~$L'$ have equal linking numbers, then 
\begin{align*}
&\lambda(L,L')=[H_F]-[H_{F'}] \in L^4(Q) \\
&\widetilde{\lambda}(L,L')=[H_F]-[H_{F'}] \in L^4(Q)/\im(L^4(\Lambda) \to L^4(Q)).
\end{align*}
\end{theorem}
\begin{proof}
Recall from Proposition~\ref{prop:mu=2Indep} and Remark~\ref{rem:SimplifiedLinkInvariant} that~$\lambda(L,L'):=\lambda(M_{L,L'},\varphi)$ and~$\widetilde{\lambda}(L,L'):=\widetilde{\lambda}(M_{L,L'},\varphi)$
do not depend on the choice of a meridional homomorphism~$\varphi \colon H_1(M_{L,L'}) \to \Z^\mu$.
We will calculate these invariants using 
the~$4$-manifold~$W_{F,F'}$ from Construction~\ref{cons:WFF}, which can be defined since the linking numbers coincide.
Since~$M_{L,L'}$ bounds~$W_{F,F'}$ over~$\Z^\mu$, we obtain a meridional homomorphism~$\varphi$ that extends over~$W_{F,F'}$.
Since~$H_i(P_L;Q)=0$ (as can be seen by repeated applications of Lemma~\ref{lemma:TwistedHomologyxS1}),
a Mayer-Vietoris sequence shows that the inclusion induces an isometry~$\lambda_{W_{F,F'}}^Q \cong \lambda_{W_F}^Q \oplus -\lambda_{W_{F'}}^Q$.
Since~$\Delta_L,\Delta_{L'} \neq 0$, a Mayer-Vietoris argument shows that~$H_1(M_L;Q)=0=H_1(M_{L'};Q)$ and so these forms are nonsingular.
As a consequence, we deduce that~$\lambda_{W_{F,F'}}^Q=\lambda_{W_{F,F'}}^{Q,\textit{ns}}$ can be used to calculate
\begin{align*}
\lambda(L,L')=\lambda(M_{L,L'},\varphi)
&=[\lambda_{W_{F,F'}}^{Q,\textit{ns}}]-[\lambda_{W_{F,F'}}^{\Q,\textit{ns}}]=[\lambda_{W_F}^Q]-[\lambda_{W_{F'}}^Q]-[\lambda_{W_{F,F'}}^{\Q,\textit{ns}}] \\
\widetilde{\lambda}(L,L')=
\widetilde{\lambda}(M_{L,L'},\varphi)
&=[\lambda_{W_{F,F'}}^{Q,\textit{ns}}]
=[\lambda_{W_F}^Q]-[\lambda_{W_{F'}}^Q]\,.
\end{align*}
We assert that~$[\lambda_{W_{F,F'}}^{\Q,\textit{ns}}]=0$.
The second item of Theorem~\ref{thm:ExistenceOfW} ensures that~$[\lambda_{W_F}^{\Q,\textit{ns}}]=0 \in L^4(\Q)$ and similarly for~$W_{F'}$.
The assertion will therefore follow from an application of the Novikov-Wall theorem
(see~\cite{Wall}) to the decomposition~$W_{F,F'}=W_F\cup_{P_L} W_{F'}$.
The lagrangians that appear in the definition of the Maslov index are the kernel of the inclusion induced maps
\[
H_1(\overline{\nu}(L);\Q) \to H_1(X_L;\Q)\,,\quad H_1(\overline{\nu}(L);\Q) \to H_1(X_{L'};\Q)\,,\quad H_1(\overline{\nu}(L);\Q) \to H_1(P_L;\Q)\,.
\]
Since the linking numbers of~$L$ and~$L'$ agree, the first two Lagrangians are equal.
Hence, the symmetric form defining the Maslov Witt class is identically zero and the assertion follows.

The third item of Theorem~\ref{thm:ExistenceOfW} shows that~$\lambda_{W_F}^Q$ and~$\lambda_{W_{F'}}^Q$ are represented by~$H_F$ and~$H_{F'}$ respectively.
This concludes the proof of the theorem.
\end{proof}

\color{black}

\section{Relating the invariants}
\label{sec:Relating}

The goal of this section is to relate our two invariants when~$\mu\le 2$.
To do so, we begin by recalling the relation between the equivariant intersection form of a~$4$-manifold and the Blanchfield form of its boundary.
Despite this type of result being well known, we are not aware of a reference that quite matches our setting. 
As a consequence, we provide some details.
Once this intermediate proposition has been proved,  the main theorem follows promptly from considerations in~$L$-theory.

\medbreak

Throughout this section, we will write~$W$ for an oriented~$4$-manifold with connected boundary~$\partial W$, endowed with an isomorphism~$\Phi\colon\pi_1(W)\to\Z^\mu$ such that
the inclusion induced homomorphism~$i_\#\colon\pi_1(\partial W)\to\pi_1(W)$ is surjective.

We begin with three technical lemmas, the first of which concerns~$\pi_2(W) \cong H_2(W;\Lambda)$.

\begin{lemma}
\label{lem:H_2IsFree}
For~$W$ as above, the~$\Lambda$-module~$H_2(W;\Lambda)$ is free if and only if $\mu \leq 3$.
\end{lemma}
\begin{proof}
The projective dimension of~$\Z\cong\Lambda/(t_1-1,\ldots,t_\mu-1)$ as a module over
the group ring~$\Lambda=\Z[\Z^\mu]$ is equal to~$\mu$. In other words, the cohomological dimension of the group~$\Z^\mu$
is equal to~$\mu$. 
By Lemma~3.1 and Proposition~3.4 of~\cite{CK},  the~$\Lambda$-module~$\pi_2(W)$ is projective (hence free, see e.g.~\cite[Chapter V,Corollary 4.12]{Lam06}) if and only if $\mu \leq 3$.
\end{proof}

The next lemma follows promptly from the universal coefficient spectral sequence (UCSS), but we record it in order to avoid repeating the argument later on.
\begin{lemma}
\label{lem:H2Vanishes}
If $Y$ is a closed $3$-manifold and~$\varphi \colon \pi_1(Y) \to \Z^\mu$ is an epimorphism such that the Alexander module~$H_1(Y;\Lambda)$ is torsion, then Poincar\'e duality and $\varphi$ induce isomorphisms
$$ H_2(Y;\Lambda) \cong H^1(Y;\Lambda) \cong  H^1(\Z^\mu,\Lambda).$$
Furthermore, this group vanishes unless $\mu=1$, in which case it is isomorphic to $\Z$.
\end{lemma}
\begin{proof}
The UCSS yields the exact sequence 
$$ 
0 \to H^1(\Z^\mu;\Lambda) \to H^1(Y;\Lambda) \xrightarrow{\ev} \Hom(H_1(Y;\Lambda),\Lambda) \to H^2(\Z^\mu;\Lambda)\,.
$$
Since $H_1(Y;\Lambda)$ is torsion,  $\Hom(H_1(Y;\Lambda),\Lambda)$ vanishes and the first statement follows.
Finally, the isomorphisms~$H^1(\Z^\mu;\Lambda) \cong H^1(\mathbb{T}^\mu;\Lambda) \cong H_{\mu-1}(\mathbb{T}^\mu;\Lambda) \cong H_{\mu-1}(\mathbb{R}^\mu;\Z)$ yield the second statement.
\end{proof}

Our next technical lemma concerns the exact sequence of the pair~$(W,\partial W)$.
Recall that for a~$\Lambda$-module~$M$, we write~$\that M=TM/zM$ and~$\shat M=M/zM$, with~$zM$ the maximal pseudonull submodule of~$M$ (see Appendix~\ref{sub:PseudoNull}).

\begin{lemma}
\label{lemma:PresentationExactSequence}
Assume that $\mu \leq 2$.
If~$H_1(\partial W;\Lambda)$ is torsion, then the exact sequence of the pair~$(W,\partial W)$ induces a resolution
\begin{equation}
\label{eq:PresentationExactSequence}
H_2(W; \Lambda) \stackrel{\hat \jmath_*}{\to} \shat H_2(W,\partial W; \Lambda) \stackrel{\hat \partial}{\to} \that H_1(\partial W; \Lambda) \to 0
\end{equation}
of the~$\Lambda$-module~$\that H_1(\partial W; \Lambda)$, where the first two modules are free of the same rank.
When $\mu=1$,  it is additionally the case that $\that H_1(\partial W; \Lambda)=H_1(\partial W; \Lambda)$.
\end{lemma}
\begin{proof}
In the case~$\mu=1$, it is well known that the exact sequence of~$(W,\partial W)$ yields a square presentation $H_2(W; \Lambda) \to H_2(W,\partial W; \Lambda) \to H_1(\partial W; \Lambda) \to 0$ of~$H_1(\partial W; \Lambda)$, see e.g.~\cite[Lemma~3.2]{ConwayPowell}. 
In order to see that~$zH_1(\partial W;\Lambda)=0$,  first note that if $\Delta \doteq 1$, then $H_1(\partial W;\Lambda)=0$ and the conclusion is clear, whereas otherwise, since~$H_1(\partial W;\Lambda) \neq 0$ is torsion (so that~$\Delta\neq 0$),
the annihilator of any non-zero element of~$H_1(\partial W;\Lambda)$ is contained in the proper principal ideal of~$\Lambda$ generated by~$\Delta$,
the determinant of the corresponding presentation matrix.

We therefore focus on the case where $\mu=2$.
Since the twisted coefficient system is defined via the isomorphism~$\Phi\colon\pi_1(W)\to\Z^2$,
the corresponding cover is the universal cover~$\widetilde{W}$, and we have~$H_1(W;\Lambda)\cong H_1(\widetilde{W})=0$ since~$\widetilde{W}$ is simply-connected. Therefore, the exact sequence of the pair~$(W,\partial W)$
contains
\begin{equation}
\label{eq:ExactSequenceW}
H_2(W;\Lambda) \stackrel {j_*}{\to} H_2(W,\partial W; \Lambda) \stackrel{\partial}{\to} H_1(\partial W; \Lambda) \to 0\,,
\end{equation}
which by Lemma~\ref{lemma:PnMapInduction} induces a sequence
\[
\shat H_2(W;\Lambda) \stackrel {\hat \jmath_*}{\to} \shat H_2(W,\partial W; \Lambda) \stackrel{\hat\partial}{\to} \shat H_1(\partial W; \Lambda) \to 0\,.
\]
By Lemma~\ref{lem:H_2IsFree}, the module~$\shat H_2(W;\Lambda)=H_2(W;\Lambda)$ is free, while~$\shat H_1(\partial W;\Lambda)=\that H_1(\partial W;\Lambda)$ since~$H_1(\partial W;\Lambda)$ is torsion. 
Therefore,
it remains to check that~$\shat H_2(W,\partial W; \Lambda)$ is free of the same rank, and that~$\partial$
restricts to a surjective map on the maximal pseudonull submodules. Indeed,
 Lemma~\ref{lemma:ExactSequencePseudonull} then shows that the exactness of~\eqref{eq:ExactSequenceW}
implies the exactness of~\eqref{eq:PresentationExactSequence}.

To show the first point,  we first note that since~$H_1(W;\Lambda)=0$ and $\Ext^3_\Lambda(\Z,\Lambda)=0$,  the UCSS gives rise to the exact sequence
\begin{equation}
\label{eq:ExactSequenceEv}
0\longrightarrow
\overbrace{\Ext^2_\Lambda(\Z,\Lambda)}^{\cong \Z}\longrightarrow H^2(W;\Lambda)\stackrel{\ev}{\longrightarrow} \overline{\Hom_\Lambda(H_2(W; \Lambda),\Lambda)} \longrightarrow 0\,.
\end{equation}
We now show that~$\ker(\ev)\cong\Z$ coincides with the maximal pseudonull submodule~$zH^2(W;\Lambda)$.
Indeed, by Lemma~\ref{lem:H_2IsFree}, the target module of~$\ev$ is free; since~$zH^2(W_F; \Lambda)$ is torsion, it is included in~$\ker(\ev)$. Moreover, the submodule~$\ker(\ev)$ being isomorphic to~$\Z\cong\Lambda/(t_1-1,t_2-1)$, it is included in~$zH^2(W;\Lambda)$ by Example~\ref{ex:Z-in-zM}, showing the equality~$\ker(\ev)=zH^2(W;\Lambda)$.
As a consequence, the evaluation map induces an isomorphism
\begin{equation}
\label{eq:hatev}
\hat\ev\colon\shat H^2(W;\Lambda)\stackrel{\cong}{\longrightarrow}\overline{\Hom_\Lambda(H_2(W;\Lambda),\Lambda)}\,.
\end{equation}
By Lemma~\ref{lem:H_2IsFree}, it follows that~$\shat H_2(W,\partial W;\Lambda)\cong \shat H^2(W;\Lambda)$ is free,
of the same rank as~$H_2(W;\Lambda)$.

We now turn to the second point, i.e. the proof that~$\partial\colon H_2(W,\partial W; \Lambda)\to H_1(\partial W; \Lambda)$ restricts to a surjective map~$\partial|\colon zH_2(W,\partial W; \Lambda)\to zH_1(\partial W; \Lambda)$ (actually an isomorphism). By naturality of Poincar\'e duality, this is equivalent to checking that the inclusion~$\partial W\subset W$
induces a map~$j^*$ which restricts to an isomorphism~$j^*|\colon zH^2(W;\Lambda)\to zH^2(\partial W;\Lambda)$.
To do so, let us use the UCSS to compute~$H^2(\partial W;\Lambda)$.
First note that since~$H_1(\partial W;\Lambda)$ is torsion, Lemma~\ref{lem:H2Vanishes} implies that~$H_2(\partial W;\Lambda)=0$.
Since $\Hom(H_1(\partial W;\Lambda),\Lambda)=0$, an application of the UCSS gives rise to the exact sequence
\begin{equation}
\label{eq:ExactSequenceExt}
0 \to  \Ext_\Lambda^2(H_0(\partial W;\Lambda),\Lambda) \to H^2(\partial W;\Lambda) \to \Ext_\Lambda^1(H_1(\partial W;\Lambda),\Lambda)\to 0\,.
\end{equation}
The two exact sequences~\eqref{eq:ExactSequenceEv} and~\eqref{eq:ExactSequenceExt} fit into the following diagram,
where the vertical maps are induced by the inclusion~$\partial W\subset W$:
\[
\begin{tikzcd}
0\arrow[r]&\Ext_\Lambda^2(H_0(W;\Lambda),\Lambda)\arrow[r]\arrow[d,"(j_*)^*"]& H^2(W;\Lambda)\arrow[r,"\ev"]\arrow[d,"j^*"]& \overline{\Hom_\Lambda(H_2(W; \Lambda),\Lambda)}\arrow[r]&  0\\
0\arrow[r]&\Ext_\Lambda^2(H_0(\partial W;\Lambda),\Lambda)\arrow[r]& H^2(\partial W;\Lambda)\arrow[r,"\pi"]& \Ext_\Lambda^1(H_1(\partial W;\Lambda),\Lambda)\arrow[r]&  0\,.
\end{tikzcd}
\]
This diagram commutes by naturality of the UCSS.
Furthermore, since both~$W$ and~$\partial W$ are connected, the left vertical map is an isomorphism between modules isomorphic to~$\Z\cong\Lambda/(t_1-1,t_2-1)$.
Recall that, by the proof of the first point above, we have~$\ker(\ev)=zH^2(W;\Lambda)$. We now check that,
similarly, we have~$\ker(\pi)=zH^2(\partial W;\Lambda)$: this will imply that~$j^*$ restricts to an isomorphism~$j^*|\colon zH^2(W;\Lambda)\to zH^2(\partial W;\Lambda)$, and conclude the proof.
To check this claim, first recall that the inclusion~$\Z\cong \ker(\pi)\subset zH^2(\partial W;\Lambda)$ follows from Example~\ref{ex:Z-in-zM}.
Also, the short exact sequence of coefficients~$0 \to \Lambda \to Q \to Q/\Lambda \to 0$ induces an
isomorphism
\begin{equation}
\label{eq:Ext-Hom-iso}
\Ext_\Lambda^1(H_1(\partial W;\Lambda),\Lambda) \cong \Hom_\Lambda(H_1(\partial W;\Lambda),Q/\Lambda)\,.
\end{equation}
By Example~\ref{ex:NoPN}, we have~$z(Q/\Lambda)=0$, so the first point of Lemma~\ref{lemma:HomPN}
implies that the module on the right of~\eqref{eq:Ext-Hom-iso} has trivial maximal pseudonull submodule.
By the isomorphism~\eqref{eq:Ext-Hom-iso}, so does the module on the left.
This yields
\[
\pi(zH^2(\partial W;\Lambda))\subset z\Ext_\Lambda^1(H_1(\partial W;\Lambda),\Lambda)=0\,,
\]
which implies the inclusion~$zH^2(\partial W;\Lambda)\subset\ker(\pi)$ and concludes the proof.
\end{proof}

The next proposition relates the~$\Lambda$-intersection form of $W$ with the Blanchfield form of $\partial W$.
As mentioned above,  despite the result being well known for $\mu=1$ (see e.g.~\cite[Theorem 2.6]{BorodzikFriedl}) and holding over $\Lambda_S$ for all $\mu$~\cite{CFT18,ConwayBlanchfield} as well as for more general coefficients~\cite{StirlingThesis},  we are not aware of a version over $\Z[\Z^2]$ that involves $\that H_1(\partial W;\Z[\Z^2])$.

 \begin{proposition}
\label{prop:BlanchfieldRepresentation}
Assume that~$\mu \leq 2$ and that the Alexander module~$H_1(\partial W;\Lambda)$ is torsion.
If~$H$ is a matrix representing the intersection form on~$H_2(W; \Lambda)$, then
the Blanchfield pairing~$\hat\Bl^\varphi_{\partial W} \colon \that H_1(\partial W;\Lambda) \times \that H_1(\partial W;\Lambda) \to Q/\Lambda$ is isometric to the pairing
\[
\Lambda^n/H^{\T} \Lambda^n\times \Lambda^n/H^{\T}\Lambda^n \longrightarrow Q/\Lambda, \quad ([a],[b])\longmapsto [-a^{\T}H^{-1}\overline{b}]\,,
\]
where~$n$ the rank of~$H_2(W;\Lambda)$.
\end{proposition}
\begin{proof}
Since the result is well known for $\mu=1$, see e.g.~\cite[Theorem 2.6]{BorodzikFriedl}, we focus on the case~$\mu=2$.
Write~$\lambda$ for the equivariant intersection form on~$H_2(W;\Lambda)$.
Throughout the proof, we write~$\Bl$ instead of~$\Bl^\varphi_{\partial W}$ for simplicity.
Let us write~$\partial W\stackrel{i}{\to}W\stackrel{j}{\to}(W,\partial W)$ for the natural inclusions,
and~$\Lambda\stackrel{\iota}{\to}Q\stackrel{p}{\to}Q/\Lambda$ for the canonical short exact sequence of coefficients.
Recall from Lemma \ref{lemma:PnMapInduction} that any~$\Lambda$-linear map~$f\colon N\to N'$
functorially induces a map~$\hat f\colon \shat N\to \shat N'$. 
Since~$H_1(\partial W;\Lambda)$ is torsion, it follows that the short exact sequence~\eqref{eq:PresentationExactSequence} from Lemma~\ref{lemma:PresentationExactSequence} can be completed into the following diagram,
which is adapted from~\cite[(12)]{CFT18}:
\begin{center}
\begin{tikzcd}[sep=small]
H_2(W; \Lambda) \arrow{d}{\Ad_\lambda} \arrow{r}{\hat\jmath_\ast} &\shat H_2(W,\partial W;\Lambda) \arrow{r}{\hat \partial} \arrow{d}{\hat{\PD}^{-1}} & \shat H_1(\partial W;\Lambda) \arrow{r}\arrow{d}{\hat{\PD}^{-1}} \arrow[bend left=15]{dddr}{\Ad_{\Bl}} & 0\\
\overline{\Hom_\Lambda(H_2(W;\Lambda),\Lambda)} \arrow{d}{\iota^\ast} & \shat H^2(W;\Lambda) \arrow[swap]{l}{\hat{\ev}} \arrow{d}{\hat\iota^\ast} \arrow{r}{\hat{i^\ast}} & \shat H^2(\partial W;\Lambda) \arrow{d}{\hat{\BS}^{-1}}\\
\overline{\Hom_\Lambda(H_2(W;\Lambda),Q)} & H^2(W;Q) \arrow{d}{(j^\ast)^{-1}}\arrow[swap]{l}{\ev} & \shat H^1(\partial W;Q/\Lambda) \arrow[bend left=5]{dr}{\hat{\ev}} \arrow{d}{\hat \delta} & \\
\overline{\Hom_\Lambda(H_2(W,\partial W;\Lambda),Q)} \arrow[swap]{u}{(j_*)^\ast} \arrow{d}{\cong} & H^2(W,\partial W;Q) \arrow[swap]{l}{\ev}  \arrow{r}{p^\ast} \arrow[swap]{dl}{\hat\ev} & \shat H^2(W,\partial W;Q/\Lambda) \arrow{dr}{\hat{\ev}} &\overline{\Hom_\Lambda(\shat H_1(\partial W;\Lambda),Q/\Lambda)} \arrow{d}{\hat\partial^\ast} \\
\overline{\Hom_\Lambda(\shat H_2(W,\partial W;\Lambda),Q)}\arrow{rrr}{p^\ast}& & & \overline{\Hom_\Lambda(\shat H_2(W,\partial W;\Lambda),Q/\Lambda)}\,.
\end{tikzcd}
\end{center}
All the maps appearing in this diagram are well-defined in a straightforward way, with the following
possible exceptions. The maps labeled~$\hat\ev$ are well-defined via the second and third points of Lemma~\ref{lemma:HomPN} together with the fact that~$z\Lambda=z(Q/\Lambda)=0$ (Example~\ref{ex:NoPN}).
The bottom left vertical map is the isomorphism given in Lemma~\ref{lemma:HomPN}(iii).
Also, the map~$\hat\iota^*\colon\shat H^2(W;\Lambda)\to H^2(W;Q)$ is well-defined as~$H^2(W;Q)=\shat H^2(W;Q)$.
Finally, the map~$j^*\colon H^2(W,\partial W;Q)\to H^2(W;Q)$ is an isomorphism by Poincar\'e duality, the exact sequence of~$(W,\partial W)$, and the fact that~$H_\ast(\partial W;\Lambda)$ is torsion (apply Poincar\'e duality and the UCSS).

We now discuss the commutativity of this diagram. The (four) squares involving two evaluation maps commute by naturality of the UCSS. The square involving~$\Ad_\lambda$, the triangle involving~$\Ad_{\Bl}$, and the bottom-left triangle commute by definition. Also, the square involving two maps labeled~$\hat\PD$ commutes by naturality of Poincar\'e duality. 
Finally, by e.g.~\cite[Appendix A]{ConwayBlanchfield} (see also~\cite[Lemma 5.4]{CFT18} and~\cite[Lemma~4.4]{LambdaSpheres}),
the central rectangle {\em anti}-commutes.

Adapting~\cite{CFT18}, let us define an adjoint~$\Ad_\theta$ as the composition
\[
\shat H_2(W,\partial W;\Lambda) \stackrel{\hat{\PD}^{-1}}{\longrightarrow} \shat H^2(W;\Lambda) \stackrel{\hat{\iota^\ast}}{\longrightarrow} H^2(W;Q) \stackrel{(j^\ast)^{-1}}{\longrightarrow} H^2(W,\partial W;Q)\stackrel{\hat\ev}{\longrightarrow}\overline{\Hom_\Lambda(\shat H_2(W,\partial W;\Lambda),Q)}\,.
\]
This defines a pairing~$\theta \colon \shat H_2(W,\partial W; \Lambda) \times \shat H_2(W,\partial W;\Lambda) \to Q$ via~$\theta (x,y) = \Ad_\theta (y)(x)$. As one can check, the (anti-)commutativity of the big diagram implies the commutativity of
\begin{equation}
\begin{tikzcd}
\label{diagram:Adjoints}
H_2(W;\Lambda) \times H_2(W;\Lambda) \arrow{d}{\hat \jmath_\ast \times \hat \jmath_\ast} \arrow{r}{-\lambda} & \Lambda \arrow{d}{\iota}\\
\shat H_2(W,\partial W;\Lambda) \times \shat H_2(W,\partial W;\Lambda) \arrow{d}{\hat \partial \times \hat \partial} \arrow{r}{-\theta} & Q \arrow{d}{p} \\
\shat H_1(\partial W;\Lambda) \times \shat H_1(\partial W;\Lambda) \arrow{r}{\Bl} & Q/\Lambda\,,\\
\end{tikzcd}
\end{equation}
with~$\shat H_1(\partial W;\Lambda)=\that H_1(\partial W;\Lambda)$ since~$H_1(\partial W;\Lambda)$
is torsion.

We now fix a~$\Lambda$-basis~$\{ x_1, \ldots, x_n\}$ of~$H_2(W;\Lambda)$ and write~$\{x_1^\ast, \ldots, x_n^\ast\}$ for the dual basis of~$\overline{\Hom_\Lambda(H_2(W;\Lambda),\Lambda)}$.
Set~$y_k=(\hat{PD}\circ\hat{\ev}^{-1})(x^*_k)$ for~$1\le k\le n$,
where~$\hat{\ev}\colon \shat H^2(W; \Lambda)\to\overline{\Hom_\Lambda(H_2(W;\Lambda),\Lambda)}$ is an isomorphism by~\eqref{eq:hatev}. 
This yields a basis~$\{y_1,\dots,y_n\}$ of~$\shat H_2(W,\partial W; \Lambda)$.
One verifies that~$H^{\T}$ is the matrix of~$\hat\jmath_*$ in
the bases~$\{x_1,\ldots,x_n\}$ of~$H_2(W;\Lambda)$
and~$\{y_1,\dots,y_n\}$ of~$\shat H_2(W,\partial W; \Lambda)$.
The remainder of the proof proceeds as in~\cite[Proof of Theorem~1.2]{CFT18}.
%
\end{proof}

\begin{remark}
\label{rem:BlanchfieldEven}
Proposition~\ref{prop:BlanchfieldRepresentation} and Remark~\ref{rem:AlternativeDefinitionBoundaryLinking} show that~$\hat \Bl_{\partial W}$ is the boundary of the hermitian form~$\lambda^\Lambda_W$. Since the boundary of a hermitian form is necessarily an even linking form,
Proposition~\ref{prop:BlanchfieldRepresentation} actually shows that~$\hat \Bl_{\partial W}$ is even when~$\mu \leq 2$ and~$H_1(\partial W;\Lambda)$ is torsion. More generally,
note that a linking form~$(T,b)$ over~$\Lambda$ is even if~$T$ is annihilated by a polynomial~$\Delta \in \Lambda$ with~$\Delta=\overline{\Delta}$.
This implies that the Blanchfield form of a closed $3$-manifold $Y$ is necessarily even if $\that H_1(Y;\Lambda)$ admits a square presentation matrix (so that $\Delta_{\hat{t}H_1(Y;\Lambda)}$ annihilates $\that H_1(Y;\Lambda)$):
indeed, since~$\hat{\Bl}_Y$ is nondegenerate on both sides (recall Section~\ref{sub:BlanchfieldDef}),  it is known that~$\Delta_{\hat{t}H_1(Y;\Lambda)}=\overline{\Delta_{\hat{t}H_1(Y;\Lambda)}}$; see e.g.~\cite[Lemma 3.26]{Hillman}.
\end{remark}

Proposition~\ref{prop:BlanchfieldRepresentation} (and Remark~\ref{rem:BlanchfieldEven}) ensures that when $\mu \leq 2$ and~$H_1(\partial W;\Lambda)$ is torsion, the Blanchfield form $\hat \Bl_{\partial W}$ defines an element in the Witt group $L^4(\Lambda,\Lambda \setminus \{0\})$ of linking forms (or more precisely the Witt group of even nonsingular linking forms defined on torsion modules of projective dimension $1$).
We are now ready to prove Theorem~\ref{thm:RelatingInvariantsIntro} from the introduction.
Recall the exact sequence~$L^4(\Lambda) \to L^4(Q)\to L^4(\Lambda,\Lambda \setminus \{0\})$ from Appendix~\ref{sub:Ltheory}.

\begin{theorem}
\label{thm:RelatingInvariants}
Fix $\mu \leq 2$ and let $L$ and $L'$ be two $\mu$-component links with equal linking number and~$\Delta_L,\Delta_{L'} \neq 0$.
\begin{enumerate}[(i)]
\item The map~$L^4(Q) \to L^4(\Lambda,\Lambda \setminus \{0\})$ takes the homology surgery invariant~$\lambda(L,L')$ to the Blanchfield invariant~$\Bl(L,L')$.
In particular, if~$\lambda(L,L')$ vanishes, then so does~$\Bl(L,L') \in  L^4(\Lambda,\Lambda \setminus \{0\})$.
\item The injective map~$L^4(Q)/\im(L^4(\Lambda) \to L^4(Q)) \to L^4(\Lambda,\Lambda \setminus \{0\})$ takes the simplified homology surgery invariant~$\widetilde{\lambda}(L,L')$ to the Blanchfield invariant~$\Bl(L,L')$.
In particular,~$\widetilde{\lambda}(L,L')$ vanishes if and only if~$\Bl(L,L') \in  L^4(\Lambda,\Lambda \setminus \{0\})$ does.
\end{enumerate}
\end{theorem}
\begin{proof}
By the first and second items of Theorem~\ref{thm:ExistenceOfW}, the Witt class~$\lambda(L,L') \in L^4(Q)$ is represented by the equivariant intersection form~$\lambda_{W_{F,F'}}^{Q,\textit{ns}}$.
Since~$\Delta_L,\Delta_{L'} \neq 0$ (so that a short Mayer-Vietoris argument shows that~$H_1(M_L;\Lambda)$ and $H_1(M_{L'};\Lambda)$ are $\Lambda$-torsion) and~$\mu \le 2$, Lemma~\ref{lem:H2Vanishes} yields~$H_2(\partial W_{F,F'};Q)=0$ and thus the form is nonsingular, i.e.~$\lambda_{W_{F,F'}}^{Q,\textit{ns}}=\lambda_{W_{F,F'}}^Q$.
Since~$H_2(W_{F,F'};\Lambda)$ is free (recall Lemma~\ref{lem:H_2IsFree}), we can represent this latter form by a matrix~$H$ with coefficients in~$\Lambda$.
By Remark~\ref{rem:AlternativeDefinitionBoundaryLinking},  the map~$L^4(Q) \to L^4(\Lambda,\Lambda \setminus \{0\})$ sends~$\lambda(L,L')$ to the Witt class of the linking form~$([x],[y])\mapsto x^TH^{-1}\overline{y}$.
By Proposition~\ref{prop:BlanchfieldRepresentation}, this latter form is isometric to the Blanchfield form on~$\partial W_{F,F'}=M_{L,L'}$ for some meridional homomorphism~$\varphi$, showing the first point.
The proof of the second point is analogous.
\end{proof}

We conclude by noting another consequence of Proposition~\ref{prop:BlanchfieldRepresentation},
 which appeared in the introduction as the sixth item of Theorem~\ref{thm:BlClassProperties}.

\begin{corollary}
\label{cor:CalculationBlmu=2}
If~$L$ is a $2$-component link with linking number one and~$L'=\mathcal{H}$ is the Hopf link, then the linking form~$\Bl_{M_{L,\mathcal{H}}}^\varphi$ is represented by~$H_F$ with~$F$ a nice C-complex for~$L$.
\end{corollary}
\begin{proof}
Since $L$ has linking number one,  Example~\ref{ex:P} implies~$P_L$ agrees with $M_{\mathcal{H}}$, the generalised Seifert surgery of the Hopf link.
It follows that~$M_{L,\mathcal{H}}=M_L$. 
The statement now follows by combining Proposition~\ref{prop:BlanchfieldRepresentation} with the third item of Theorem~\ref{thm:ExistenceOfW}.
\end{proof}

\section{Projective dimension and nonsingularity}
\label{sec:ProjDim}

At this stage we have proved Theorem~\ref{thm:ConcordanceInvarianceGeneralIntro} (assuming a calculation that will be performed in Section~\ref{sec:IntersectionForm}), Theorem~\ref{thm:RelatingInvariantsIntro}, as well as all of  Theorem~\ref{thm:BlClassProperties} besides its third and fourth items.
This section proves these missing items.
The organisation is as follows.
Section~\ref{sub:Technical} proves a technical lemma that we will be needed in both of the later sections.
Section~\ref{sub:ProjDim} analyses the projective dimension of Alexander modules of closed $3$-manifolds.
Section~\ref{sub:NonSing} focuses on the nonsingularity of the Blanchfield form and completes the proof of the third and fourth items of Theorem~\ref{thm:BlClassProperties}.

\subsection{A technical lemma}
\label{sub:Technical}

We begin with a technical lemma that will be used both to study the projective dimension of the Alexander module and the nonsingularity of the Blanchfield form.
For this, it will be helpful to note that a homomorphism~$\varphi \colon H_1(Y) \to \Z^3$ is an element of
\begin{equation}
\label{eq:phi-fY}
\Hom_\Z(H_1(Y),\Z^3)\cong H^1(Y;\Z^3)\cong [Y,B\Z^3]\cong [Y,\mathbb{T}^3]\,.
\end{equation}
We then define the \emph{degree of $\varphi$}, denoted $\deg(\varphi)\in \Z$, as the degree of the associated homotopy class of a map~$f_Y^\varphi \colon Y\to\mathbb{T}^3$,  i.e. as the unique integer such that~$(f_Y^\varphi)_*([Y])=\deg(\varphi)[\bbT^3]$.

\begin{lemma}
\label{lem:Ext-N}
Let $Y$ be a closed $3$-manifold and let $\varphi \colon H_1(Y) \to \Z^3$ be an epimorphism such that the Alexander module~$H_1(Y;\Lambda)$ is torsion.
For any~$\Lambda$-module~$N$, we have an isomorphism
\[
\Ext^2_\Lambda(H_1(Y;\Lambda),N)\cong N_\pi/\deg(\varphi) N_\pi\,,
\]
where~$N_\pi$ stands for the quotient of~$N$ by the submodule~$\left\{\sum_{i=1}^3(t_i-1)n_i\mid n_i\in N\right\}$.
\end{lemma}
\begin{proof}
The idea is to compute the twisted homology~$\Lambda$-module~$H^3(Y;N)$ using the UCSS.
To do so, first note that since~$\varphi$ is surjective, the induced covering space~$\widetilde{Y}$ is a noncompact
connected~$3$-manifold. 
Thus, the corresponding homology module~$H_3(Y;\Lambda)\cong H_3(\widetilde{Y};\Z)$ vanishes.
The definition of group cohomology yields~$\Ext^k_\Lambda(H_0(Y;\Lambda),N)\cong H_{3-k}(\mathbb{T}^3;N)=0$
for~$k\ge 4$.
Since we also have~$H_2(Y;\Lambda)=0$ by Lemma~\ref{lem:H2Vanishes}, the~$E_2$-page of the UCSS for~$H^*(Y;N)$ has the following form:
\begin{center}
\begin{tikzpicture}
  \matrix (m)[matrix of math nodes, nodes in empty cells, column sep=1ex,row sep=1ex]{ 
	   && & & &&\\
	4 && 0 & \ast & \ast & \ast &\\
	3 && \Ext_\Lambda^3(H_0(Y;\Lambda),N) & \ast & \ast & \ast &\\
	2 && \ast & \Ext_\Lambda^2(H_1(Y;\Lambda),N) & \ast & \ast &\\
	1 && \ast & \Ext_\Lambda^1(H_1(Y;\Lambda),N) & 0 & \ast &\\
	0 && \ast & \ast & 0&0&\\
    \strut &&   0  &  1  &  2  & 3 &\strut \\};
\draw[thick] (m-1-2.east) -- (m-7-2.east) ;
\draw[thick] (m-7-1.north) -- (m-7-7.north) ;
\end{tikzpicture}
\end{center}
Therefore, both~$E_\infty^{3,0}$ and~$E_\infty^{2,1}$ vanish, while~$E_\infty^{1,2}\cong\Ext_\Lambda^2(H_1(Y;\Lambda),N)$ and
\[
E_\infty^{0,3}\cong\coker\left(\Ext^1_\Lambda(H_1(Y;\Lambda),N)\stackrel{d}{\longrightarrow}\Ext^3_\Lambda(H_0(Y;\Lambda),N)\right)\,.
\]
By the UCSS, we get an exact sequence
\[
\Ext^1_\Lambda(H_1(Y;\Lambda),N)\longrightarrow\Ext^3_\Lambda(H_0(Y;\Lambda),N)\longrightarrow H^3(Y;N)\longrightarrow\Ext^2_\Lambda(H_1(Y;\Lambda),N)\longrightarrow0\,.
\]

As explained in~\eqref{eq:phi-fY}, the homomorphism~$\varphi$ corresponds to the homotopy class of a map~$f_{Y}^\varphi=:f\colon Y\to \bbT^3$.
Let~$\psi\colon H_1(\bbT^3)\to\Z^3$ be the unique meridional isomorphism such that~$\psi\circ f_*=\varphi$.
We are now in presence of a map~$f\colon (Y,\varphi)\to(\bbT^3,\psi)$ between three-dimensional~$\Z^3$-manifolds.
Using the fact that~$H_1(\bbT^3;\Lambda)$ vanishes, the exact sequence displayed above  yields the diagram
\[
\begin{tikzcd}
0\arrow[r]& \Ext^3_\Lambda(H_0(\bbT^3;\Lambda),N) \arrow[r,"\cong"] \arrow[d,"(f_*)^*","\cong"'] & H^{3}(\bbT^3;N)\arrow[d,"f^*"] \arrow[r]& 0 &\\
&\Ext^3_\Lambda(H_0(Y;\Lambda),N)\arrow[r,"g"] & H^3(Y;N)\arrow[r]&\Ext^2_\Lambda(H_1(Y;\Lambda),N)\arrow[r]&0\,,
\end{tikzcd}
\]
which is commutative by naturality of the UCSS.
Note that the upper middle map is an isomorphism by exactness of the first line, while the left vertical map~$(f_*)^*$
is an isomorphism by naturality.
As a consequence, exactness of the second line yields an isomorphism
\[
\Ext^2_\Lambda(H_1(Y;\Lambda),N)\cong\coker(g)=\coker(f^*)\,.
\]

To compute this module, consider the following diagram, where the left horizontal maps are Poincar\'e duality
isomorphisms, and the right horizontal maps are isomorphisms coming from the definition of twisted homology (i.e. induced by the augmentation maps):
\[
\begin{tikzcd}
H^3(\bbT^3;N) \arrow[r,"\PD","\cong"'] \arrow[d,"f^*"] & H_0(\bbT^3;N) \arrow[r,"\cong"'] & N_\pi\arrow[d,equal]\\
H^3(Y;N) \arrow[r,"\PD","\cong"'] & H_0(Y;N) \arrow[u,"f_*","\cong"']\arrow[r,"\cong"'] & N_\pi\,.
\end{tikzcd}
\]
Since~$f_*$ is also an isomorphism, one might be led to believe that~$f^*$ is always onto.
The issue is that, while the right square commutes, the left square
does not commute in general. Indeed, by definition of Poincar\'e duality and by functoriality of the cap product, we have
\[
(f_*\circ\PD\circ f^*)(\alpha)=f_*(f^*(\alpha)\frown [Y])=\alpha\frown f_*([Y])=\alpha\frown\deg(f)[\bbT^3]=\deg(f)\PD(\alpha)
\]
for all~$\alpha\in H^3(\bbT^3;N)$. The three equations and diagrams displayed above yield
\[
\Ext^2_\Lambda(H_1(Y;\Lambda),N)\cong\coker(f^*)\cong N_\pi/\deg(f) N_\pi\,.
\]
The result now follows from the definition of~$\deg(\varphi)$.
\end{proof}

\color{black}

\subsection{The projective dimension of the Alexander module}
\label{sub:ProjDim}
The goal of this section is to prove the following result.

\begin{proposition}
\label{prop:NoSQP4}
Let~$Y$ be a closed~$3$-manifold and let~$\varphi \colon H_1(Y) \to \Z^\mu$ be an epimorphism such that the Alexander module~$H_1(Y;\Lambda)$  is torsion.
The~$\Lambda$-module~$H_1(Y;\Lambda)$ has projective dimension~$1$ if and only if~$\mu=1$ or~$\mu=3$ and~$\deg(\varphi)=\pm 1$.
\end{proposition}
\begin{proof}
We begin with the if direction.
Assume that~$\mu=1$.
Since $\Omega_3(\Z)=0$,  the pair $(Y,\varphi)$ bounds a~$4$-manifold~$(W,\psi)$.
Surgering $W$ if necessary, we can assume that $\pi_1(W) \cong \Z$.
The exact sequence of the pair~$(W,Y)$ then provides the required free resolution; recall Lemma~\ref{lemma:PresentationExactSequence}.
Let us now assume~$\mu=3$.
By Lemma~\ref{lem:Ext-N},  we have~$\deg(\varphi)=\pm 1$ if and only if~$\Ext^2_\Lambda(H_1(Y;\Lambda),N)=0$ for any~$\Lambda$-module~$N$. 
By~\cite[Lemma~4.1.6]{Weibel}, this is equivalent to the~$\Lambda$-module~$H_1(Y;\Lambda)$ having projective dimension at most~$1$.

For the only if direction, it remains to assume that~$1<\mu\ne 3$,
and prove that~$H_1(Y;\Lambda)$ has projective dimension
at least~$2$.
We will show that the module~$\Ext_\Lambda^{2}(H_1(Y;\Lambda),\Lambda)$ is nontrivial, which implies the statement.
To check this claim, let us focus on the~$E_2$-page of the UCSS applied to the chain complex~$C_*(\widehat{Y})$. 
Since $H_2(Y;\Lambda) = 0$ by Lemma~\ref{lem:H2Vanishes} (recall that~$\mu>1$), the~$2$-column only contains zeros. 
Since the modules~$H_\ast(Y;\Lambda)$ are torsion,  all the modules in the~$0$-row vanish.
Finally, since~$H_0(Y;\Lambda)\cong\Z$ and~$\mu\neq 3$,  we deduce that~$E_2^{0,3}=\Ext^3_\Lambda(\Z,\Lambda) =0$. Hence, the~$E_2$-page has the following form:
\begin{center}
\begin{tikzpicture}
  \matrix (m)[matrix of math nodes, nodes in empty cells, column sep=1ex,row sep=1ex]{ 
	   && & & &&\\
	4 && \Ext^4_\Lambda(\Z,\Lambda) & \ast & \ast & \ast&\\
	3 && 0 & \ast & \ast & \ast &\\
	2 && \ast & \Ext_\Lambda^2(H_1(Y;\Lambda),\Lambda) & \ast &\ast &\\
	1 && \ast & \ast & 0 &\ast&\\
	0 && \ast & \ast & 0&0&\\
    \strut &&   0  &  1  &  2  & 3 &\strut \\};
\draw[thick] (m-1-2.east) -- (m-7-2.east) ;
\draw[thick] (m-7-1.north) -- (m-7-7.north) ;
\end{tikzpicture}
\end{center}
This implies that the modules~$E^{3,0}_\infty$,~$E^{2,1}_\infty$ and~$E^{0,3}_\infty$
all vanish. As a consequence,
computing~$H^3(Y;\Lambda)\cong H_0(Y;\Lambda)\cong\Z$ with the help of this spectral sequence
yields the isomorphisms
\[
\Z\cong H^3(Y;\Lambda)\cong E_\infty^{1,2} \cong \ker(\Ext_\Lambda^2(H_1(Y;\Lambda),\Lambda) \stackrel{d_2}{\longrightarrow} \Ext^4_\Lambda(\Z,\Lambda))\,.
\]
In particular, the module~$\Ext_\Lambda^2(H_1(Y;\Lambda),\Lambda)$ is nontrivial, concluding the proof.
\end{proof}

\begin{remark}
\label{rem:proj-dim-1}
One can check that a finitely presented torsion~$\Lambda$-module admits a square presentation matrix if and only if its projective dimension is at most~$1$.
In particular Proposition~\ref{prop:NoSQP4} should be compared to the fact that the Alexander module of a link with nonzero Alexander polynomial admits a square presentation matrix if and only if $\mu <3$~\cite{CrowellStrauss}.
\end{remark}

\subsection{Nonsingularity}
\label{sub:NonSing}

The goal of this section is to prove the following result,  and to complete the proof of the third and fourth items of Theorem~\ref{thm:BlClassProperties}.

\begin{proposition}
\label{prop:NonSingTorsionY}
Let $Y$ be a closed $3$-manifold and let $\varphi \colon H_1(Y) \to \Z^\mu$ be an epimorphism.
If the Alexander module~$H_1(Y;\Lambda)$ is torsion, then 
 $$\hat{t} H_1(Y;\Lambda)
\cong
\begin{cases}
 H_1(Y;\Lambda) &\quad \text{if } \mu \neq 2, \\
H_1(Y;\Lambda)/\Z &\quad \text{if } \mu=2.
\end{cases}
$$
Additionally,  the following assertions are equivalent:
\begin{itemize}
\item the Blanchfield pairing~$\hat{\Bl}^{\varphi}_Y$ is nonsingular,
\item either $\mu \neq 3$ or $\mu=3$ and $\deg(\varphi) \neq 0$.
\end{itemize}
\end{proposition}
\begin{proof}
This follows from an analysis of the evaluation~$\ev \colon H^1(Y; Q/\Lambda)\to \overline{\Hom_\Lambda(H_1(Y; \Lambda),Q/\Lambda)}$ as we now explain.
Since~$H_1(Y;\Lambda)$ is torsion, the Bockstein map $\BS \colon H^1(Y;Q/\Lambda) \to H^2(Y;\Lambda)$ from the definition of the Blanchfield form is readily seen to be an isomorphism.
It follows that 
$$\hat{t} H_1(Y;\Lambda)=TH_1(Y;\Lambda)/zH_1(Y;\Lambda)=H_1(Y;\Lambda)/zH_1(Y;\Lambda)=H_1(Y;\Lambda)/\ker(\ev)\,,$$
where the equality~$zH_1(Y;\Lambda)=\ker(\ev)$ is a consequence of Blanchfield's theorem (recall Section~\ref{sec:Blanchfield}).
We deduce that $\hat{t} H_1(Y;\Lambda)=H_1(Y;\Lambda)$ if and only if $\ev$ is injective and, since $\hat{\Bl}^{\varphi}_Y$ is nondegenerate,  that $\hat{\Bl}^{\varphi}_Y$ is nonsingular if and only if $\ev$ is surjective.

Thus the proposition will follow once we show that $\ev$ is injective for $\mu \neq 2$ and analyze its cokernel.
The UCSS leads to the exact sequence
\begin{equation}
\label{eq:UCSSPatch}
\begin{split}
0\longrightarrow \Ext^1_\Lambda(\Z, Q/\Lambda) \longrightarrow &H^1(Y; Q/\Lambda) 	\stackrel{\ev}{\longrightarrow} \overline{\Hom_\Lambda(H_1(Y;\Lambda),Q/\Lambda)} \\
	&\stackrel{d_2}{\longrightarrow} \Ext^2_\Lambda(\Z,Q/\Lambda)\longrightarrow H^2(Y;Q/\Lambda) \longrightarrow \Ext_\Lambda^1(H_1(Y;\Lambda),Q/\Lambda) \longrightarrow 0.
\end{split}
\end{equation}
It therefore remains to understand $\Ext^k_\Lambda(\Z, Q/\Lambda)$ for $k=2,3$ as well as the image of~$d_2$.

We begin with the former.
Using the Bockstein isomorphism associated to the short exact sequence $0 \to \Lambda \to Q \to Q/\Lambda \to 0$ and the definition of group cohomology, we obtain
\begin{equation*}
\Ext^{k-1}_\Lambda(\Z, Q/\Lambda)
\cong \Ext^{k}_{\Z[\Z^\mu]}(\Z,\Lambda) \cong
H^k(\Z^\mu;\Lambda).
\end{equation*}
We deduce that $\ker(\ev)=0$ for $\mu \neq 2$ and $\ker(\ev) \cong \Z$ for $\mu=2$.
Similarly,  $\coker(\ev)=0$ if~$\mu \neq 3$.

It only remains to consider the case where $\mu=3$ and determine $\coker(\ev)$.
Since~$H_*(Y;\Lambda)$ is torsion,  the Bockstein
homomorphism and Poincar\'e duality yield
$
H^2(Y;Q/\Lambda)\cong H^3(Y;\Lambda)\cong H_0(Y;\Lambda)\cong\Z.
$
Thus for $\mu=3$, the exact sequence from~\eqref{eq:UCSSPatch} takes the form
\[
0\to H^1(Y;Q/\Lambda) \stackrel{\ev}{\longrightarrow} \overline{\Hom_\Lambda(H_1(Y;\Lambda);Q/\Lambda)} \stackrel{d_2}{\longrightarrow} \Z \to \Z \to \Ext_\Lambda^2(H_1(Y;\Lambda),\Lambda) \to 0\,,
\]
Applying Lemma~\ref{lem:Ext-N} with~$N=\Lambda$ shows that~$\Ext^2_\Lambda(H_1(Y;\Lambda),\Lambda)\cong \Z/\deg(\varphi) \Z$.

We conclude.
We have seen that $\hat{\Bl}_Y^\varphi$ 
\color{black}
being nonsingular is equivalent to~$d_2$ being the zero map,  to the map~$\Z\to\Z$ being multiplication by some~$n\neq 0$, and to~$\Ext_\Lambda^2(H_1(M_L;\Lambda),\Lambda)\cong\Z/n\Z$ with~$n\neq 0$.
As mentioned above,  this is equivalent to~$\deg(\varphi)\neq 0$.
The proposition follows.
\end{proof}

As mentioned in Remark~\ref{rem:triple-deg}, the integer~$\deg(\varphi)$ sometimes admits an explicit description.

\begin{example}\label{ex:triple-deg}
Let~$L$ and~$L'$ be~$3$-component links with vanishing pairwise linking numbers, and let~$M_{L,L'}=X_L\cup - X_{L'}$ be the corresponding closed~$3$-manifold (recall Construction~\ref{cons:MLL'}). We claim that for any meridional
homomorphism~$\varphi\colon H_1(M_{L,L'})\to\Z^3$, the following equality holds:
\begin{equation}
\label{eq:triple-deg}
\deg(\varphi)=\overline{\mu}_{L'}(123)-\overline{\mu}_{L}(123)\,.
\end{equation}
Indeed, observe that any such~$\varphi$ restricts to the meridional isomorphism~$\psi\colon H_1(X_L)\to\Z^3$ and to the opposite of the meridional isomorphism~$\psi'\colon H_1(X_{L'})\to\Z^3$. By~\eqref{eq:phi-fY}, the corresponding map~$f^\varphi_{M_{L,L'}}\colon M_{L,L'}\to\mathbb{T}^3$ can be obtained by gluing~$f^\psi_{X_L}\colon X_L\to\mathbb{T}^3$ and~$-f^{\psi'}_{X_{L'}}\colon X_{L'}\to\mathbb{T}^3$ along~$\partial X_L\cong\partial X_{L'}$. Since the pairwise linking numbers of~$L$ vanish, the inclusion of~$X_L$ into the~$0$-surgery~$M_L=X_L\cup -P_L$ is seen to induce an isomorphism~$H_1(X_L)\cong H_1(M_L)$. Using~\eqref{eq:phi-fY} once again, the map~$f^\psi_{X_L}\colon X_L\to\mathbb{T}^3$ extends uniquely up to homotopy to~$f^\psi_{M_L}\colon M_L\to\mathbb{T}^3$, and similarly for~$L'$. 
Assuming that these extensions coincide on~$P_L=P_{L'}$,
we now have the equality
\begin{align*}
f^\varphi_{M_{L,L'}}(M_{L,L'})&=f^\psi_{X_{L}}(X_L)-f^{\psi'}_{X_{L'}}(X_{L'})
=(f^\psi_{M_{L}}(X_L)-f^\psi_{M_L}(P_L))-(f^{\psi'}_{M_{L'}}(X_{L'})-f^{\psi'}_{M_L'}(P_{L'}))\\
&=f^\psi_{M_{L}}(M_L)-f^{\psi'}_{M_{L'}}(M_{L'})
\end{align*}
in~$C_3(\mathbb{T}^3)$, which implies~$\deg(\varphi)=\deg(\psi)-\deg(\psi')$.
By~\cite[Section~4]{PP22} (with a sign correction, see~\cite{M-M}),  we have~$\deg(\psi)=-\overline{\mu}_L(123)$ and~$\deg(\psi')=-\overline{\mu}_{L'}(123)$, concluding the proof of~\eqref{eq:triple-deg}.
\end{example}

This result of~\cite{PP22} also allows us to give an explicit example of a singular Blanchfield pairing.

\begin{example}
\label{ex:SingBl}
Consider a~$3$-component link~$L$ with vanishing pairwise linking numbers and triple linking number,
but non-vanishing Alexander polynomial (e.g. an element of the infinite family described in~\cite[Figure~1]{CFP}). Let~$Y=M_L$ denote the~$0$-surgery on~$L$, and let~$\varphi$ be the unique meridional homomorphism on~$L$.
By~\cite[Section~4]{PP22}, we have~$\deg(\varphi)=-\overline{\mu}_L(123)$ which vanishes by hypothesis.
By Proposition~\ref{prop:NonSingTorsionY}, the pairing~$\hat\Bl_Y^\varphi$ is singular.
\end{example}

\begin{proof}[Proof of Theorem~\ref{thm:BlClassProperties} (3)-(4)]
Let~$L,L'$ be~$\mu$-component links with~$\Delta_L,\Delta_{L'}\neq 0$, an assumption which implies that~$H_1(M_{L,L'};\Lambda)$ is torsion.
For~$\mu\le 2$, Proposition~\ref{prop:BlanchfieldRepresentation} can hence be applied and yields a square presentation matrix for~$\Bl(L,L')$, and therefore for~$\that H_1(M_{L,L'};\Lambda)$.
If~$\mu>3$ and if~$\mu=3$ and~$\deg(\varphi)\neq \pm 1$, Proposition~\ref{prop:NoSQP4} and Remark~\ref{rem:proj-dim-1} imply that~$H_1(M_{L,L'};\Lambda)$ does not admit a square presentation matrix.
By Proposition~\ref{prop:NonSingTorsionY}, this module coincides with~$\that H_1(M_{L,L'};\Lambda)$ for~$\mu\neq 2$, thus completing the proof of the third item.
As for the fourth item, it follows readily from Proposition~\ref{prop:NonSingTorsionY} when~$\Delta_L,\Delta_{L'}\neq 0$, a condition automatically satisfied for~$\mu=1$. 
\end{proof}

\color{black}

\section{Intersection form calculations}
\label{sec:IntersectionForm}

The aim of this section is to prove the second and third items of Theorem~\ref{thm:ExistenceOfW}.
More precisely, we start in Section~\ref{sec:ExplicitBasis} by determining an explicit basis of~$H_2(W_F)$ and computing the integral intersection form with respect to this basis, leading to a proof of Theorem~\ref{thm:ExistenceOfW}~(2).
We then turn to the~$\Z[\Z^\mu]$-intesection form, starting in Section~\ref{sub:ConstructionSpheres} by constructing a family of immersed~$2$-spheres in~$W_F$ yielding a collection of elements in~$\pi_2(W_F) \cong H_2(W_F; \Lambda)$. 
Section~\ref{sub:CalculInt} is devoted to the computation of the equivariant intersection form on this family. 
In Section~\ref{sub:EquivIntForm}, we use an algebraic observation (that we learnt from conversations with Feller and Lewark~\cite{FL26}; see also~\cite[Lemma 2.4]{JuhaszPowell}) to show that this family forms a basis, thereby completing the proof of Theorem~\ref{thm:ExistenceOfW}~(3).

We would like to emphasise
that part of our set-up and calculations were inspired by arguments that we learnt from Feller and Lewark~\cite{FL26} for~$\mu=1$ (but taking place in a different $4$-manifold).

\subsection{The integral intersection form}
\label{sec:ExplicitBasis}

Given a nice C-complex~$F$ for a~$\mu$-component link~$L$ with~$\mu\le 2$, we have constructed a~$4$-manifold~$W_F$ such that~$\partial W_F=M_L$ (recall Section~\ref{sub:WF}).
The aim of the present section is to compute the integral intersection form~$Q_{W_F}$ on~$H_2(W_F)$.

\medskip

We start by exhibiting a basis of~$H_2(W_F)$, assuming the notation of Section~\ref{sub:WF}.
In particular, we use the notation
\[
\{\alpha_1,\ldots,\alpha_{g}, \beta_1,\ldots,\beta_{g}, \gamma_1,\ldots,\gamma_l, \delta_1, \ldots, \delta_n\}
\]
for a good basis of~$H_1(F)$ (recall Terminology~\ref{ter:geo-symp-good}, and the fact that~$l=n=0$ if~$\mu=1$), and we write
\[
W_F = V_F \; \cup \; \bigcup_{k=1}^n \left(X_k\cup (Y_k^1\times S^1) \cup (Y_k^2 \times S^1) \right) \; \cup \;  \bigcup_{m = 1}^{g} \left(Y_m \times S^1 \right)
\]
for the construction of~$W_F$ (recall the second part of Section~\ref{sub:WF}).

\medskip

We now describe immersed spheres and tori in~$W_F$ that will later be shown to generate~$H_2(W_F)$.
\begin{itemize}
\item Each oriented simple closed curve~$\beta\in\{\beta_1,\dots,\beta_g\}$ with~$\beta\subset F_i$ determines an oriented torus~$\beta \times S^1 \subset P_F$.
(Indeed the discs $D_i$ that enter the definition of $P_F$ can be assumed to avoid the genus curves).
 Via the homeomorphism~$h\colon P_F \to\partial \onu(\tilde{F}) \setminus \nu(L)$ from Construction~\ref{cons:Homeoh} it is mapped to the oriented embedded torus
\[
T_{\beta} := h(\beta \times S^1) \subset \partial \onu(\tilde{F}) \setminus \nu(L)\,.
\]
\item Each oriented simple closed curve~$\alpha\in\{\alpha_1,\dots,\alpha_g\}$ with~$\alpha\subset F_i$
determines a curve $\alpha \times \{0\} \star \{i+c\}\subset\partial \onu(\tilde F)$ contained in the exterior~$B$ of the trace of the
pushed-in C-complex. Since~$B$ is a topological ball, this curve bounds a properly immersed oriented disc~$D_{\alpha}\subset B\subset V_F$. Moreover, writing~$\alpha=\alpha_m$, the space~$Y_m \times S^1=D^2\times [-1,1]\times S^1$ is attached to~$V_F$ with~$\partial D^2\times\{0\}\times \{*\}$ glued along this same curve, which therefore bounds a disc~$D_m\subset Y_m \times S^1 \subset W_F$.
We then consider the immersed sphere
\[
T_{\alpha} := D_{\alpha} \cup D_m
\]
that we orient by extending the orientation of~$D_{\alpha}$.

\item
Each clasp~$c$ in the C-complex~$F$ yields an edge~$e$ in the graph~$\Gamma_F$, and an oriented torus~$T_e=-\partial D_e \times S^1 \subset P_F$. 
We then consider the oriented embedded torus
\[
T_c:=h(T_e)\subset \partial\onu(\tilde{F}) \setminus \nu(L)
\]
obtained via the homeomorphism~$h\colon P_F \to\partial \onu(\tilde{F}) \setminus \nu(L)$ from Construction~\ref{cons:Homeoh}.
\item
Recall from the proof of Lemma~\ref{lemma:Claim0} (e.g.~\eqref{eq:wtdelta}) that each cancelling curve~$\delta\in\{\delta_1,\dots,\delta_n\}$ determines an oriented
loop~$\tilde \delta\subset\partial B$.
In more detail, the curve~$\tilde \delta$ is obtained from the push off $\delta^{++}$ by pushing the C-complex into $B^4$.
By construction, the curve~$\tilde \delta$
\color{black}
 bounds an oriented immersed disc~$D_{\delta}\subset B\subset V_F$.
Moreover, writing~$\delta=\delta_k$, the manifold~$W_F$ is obtained by attaching~$X_{k}$ along~$\tilde\delta$
and further attaching~$(Y^1_{k} \times S^1)\cup(Y^2_{k} \times S^1)$ to~$(\delta^1 \times S^1)\cup(\delta^2 \times S^1)$, where~$\delta^i$ are simple closed curves obtained by splitting~$\tilde \delta$ into two curves; recall~\eqref{eq:deltai} and Figure~\ref{fig:rectangle}.
Since~$\delta^i \times \{y\}$ bounds a disc in~$(Y^i_{k}\times S^1)\subset W_F$,
this yields an embedded disc~$D_k\subset W_F$ with boundary~$\partial D_k=\tilde \delta$.
We then consider the immersed sphere
\[
R_{\delta} := D_{\delta} \cup D_{k}\,,
\]
and orient it by extending the orientation on~$D_{\delta}$ such that~$\partial D_{\delta}=\tilde \delta$.
\end{itemize}

This family of immersed surfaces in~$W_F$ represents a family of elements in~$H_2(W_F)$, which turns out to be slightly
too large for a basis.
To describe the correct subset, consider the tori~$T_1,\dots,T_{\vert\ell\vert}$ corresponding to the clasps that are not traversed by
any of the cancelling curves~$\delta_k$ (recall Terminology~\ref{ter:geo-symp-good}, and the fact that~$\ell=\lk(K_1,K_2)$).
There remains~$2n$ clasps, paired by the cancelling curves~$\delta_k$: we choose a representative of each pair, and denote by~$T_{c_1},\dots,T_{c_n}$ the corresponding tori. 

\begin{lemma}
\label{lemma:Basis}
The family
\[
\mathcal {D} : = \{T_{\beta_1}, \ldots, T_{\beta_{g}}, T_{\alpha_1}, \ldots, T_{\alpha_{g}}, T_1, \ldots, T_{\vert\ell\vert}, T_{c_1}, \ldots, T_{c_n}, R_{\delta_1}, \ldots, R_{\delta_n}\}
\]
represents a basis of~$H_2(W_F)$.
\end{lemma}

\begin{proof}
We first claim that the abelian group~$H_2(V_F)$ is freely generated by the classes of the images under~$h\colon P_F\to\partial \onu(\tilde{F}) \setminus \nu(L)\subset\partial V_F$ of the following tori in~$P_F$:
\[
\{\beta_i \times S^1,\alpha_i\times S^1\}_{i=1}^{g}\;\cup\; \{T_e\}_e\,,
\]
where~$e$ ranges over all the edges of the graph~$\Gamma_F$ (i.e. all the clasps in~$F$).
The long exact sequences of the pairs~$(V_F, B^4)$ and~$(h(P_F), \nu(\tilde F))$ together with excision yield isomorphisms
\[
H_2(V_F) \cong H_3(B^4, V_F) \cong H_3(\overline{\nu}(\tilde F), h(P_F)) \cong H_2(h(P_F))\,.
\]
By naturality, the resulting isomorphism~$H_2(h(P_F)) \cong H_2(V_F)$ is induced by the inclusion,
and we are left with the computation of~$H_2(P_F)$.
If~$\mu=1$, then~$P_F=F_1\times S^1$ and the claim follows from the K\"unneth formula; we therefore assume~$\mu=2$.
The Mayer-Vietoris exact sequence for the decomposition~$P_F = (F_1^\circ \times S^1) \cup (F_2^\circ \times S^1)$ yields
\[
\bigoplus_{e} H_2(T_e) \to H_2(F_1^\circ \times S^1) \oplus H_2(F_2^\circ \times S^1) \to H_2(P_F) \to \bigoplus_{e} H_1(T_e) \to H_1(F_1^\circ \times S^1) \oplus H_1(F_2^\circ \times S^1),
\]
with the last map injective.
The claim now follows from the K\"unneth formula applied to~$F_i^\circ\times S^1$.

As a second step, we study the gluing of~$X_k=I\times I\times S^1\times S^1$ to~$V_F$
along two edges~$e,e'$ of~$\Gamma_F$ (recall the second part of Section~\ref{sub:WF}).
A careful use of the Mayer-Vietoris exact sequence shows that the effect on the second homology is the following:
the two classes~$[T_e]$ and~$[T_{e'}]$ get identified, while two new independent classes~$[\delta_k^1\times S^1]$ and~$[\delta_k^2\times S^1]$ appear.
Similarly, a Mayer-Vietoris argument shows that attaching the thickened solid tori~$Y_k^1\times S^1$ and~$Y_k^2\times S^1$
cancels the classes~$[\delta_k^1\times S^1]$ and~$[\delta_k^2\times S^1]$ (recall that~$Y_k^i=D^2\times [-1,1]$ is glued so that~$\partial D^2=\delta_k^i$), while creating the new class~$[R_{\delta_k}]$.
Summarizing the results obtained so far, the second homology of
\[
 V_F\cup  \bigcup_{k=1}^n \left(X_k \cup (Y_k^1 \times S^1) \cup (Y_k^2 \times S^1)\right)
\]
is freely generated by the classes of the tori
\[
\{T_{\beta_1}, \ldots, T_{\beta_{g}}, h(\alpha_1\times S^1), \ldots, h(\alpha_{g}\times S^1), T_1, \ldots, T_{\vert\ell\vert}, T_{c_1}, \ldots, T_{c_n}, R_{\delta_1}, \ldots, R_{\delta_n}\}\,,
\]
with~$\ell=n=0$ if~$\mu=1$.
Finally, one last Mayer-Vietoris argument shows that the gluing of~$Y_m\times S^1$ has the effect of replacing
the generator~$h(\alpha_m\times S^1)$ with~$T_{\alpha_m}$. This concludes the proof.
\end{proof}

We now compute the intersection form~$Q_{W_F}$ on the elements of the basis~$\mathcal D$.

\begin{lemma}
\label{lemma:Matrix}
The intersection form~$Q_{W_F}$ on~$H_2(W_F)$ with respect to the basis~$\mathcal D$ is given by a block matrix
of size~$g+g+\vert\ell\vert+n+n$ of the form
\[
\begin{pmatrix}
0 & Q_F(\beta,\alpha) & 0 & 0 & \ast \\
Q_F(\alpha, \beta)^{\T} & \ast & 0 & 0 & \ast \\
0 & 0 & 0 & 0 & 0 \\
0 & 0 & 0 & 0 & \widetilde\id \\
\ast & \ast & 0 &\widetilde \id & \ast
\end{pmatrix}\,,
\]
where~$Q_F(\rho, \rho')_{m,m'}$ denotes the algebraic intersection of~$\rho_m$ and~$\rho'_{m'}$ in~$F$
and $\widetilde \id$ is a matrix with~$\pm 1$ on the diagonal and zeros elsewhere.
\end{lemma}

\begin{proof}
We first claim that the intersections with the~$T_\beta$ tori are given by
\[
Q_{W_F}(T_{\beta},T_{\beta'}) = Q_{W_F}(T_{\beta},T_l) = Q_{W_F}(T_{\beta},T_{c}) = 0\quad\text{and}\quad
Q_{W_F}(T_{\beta},T_{\alpha}) = Q_F(\beta, \alpha)\,, 
\]
where~$Q_F$ denotes the algebraic intersection of curves on~$F$.
Indeed, recall that the tori~$T_{\beta}$,~$T_l$ and~$T_{c}$ are all of the form~$h(\rho \times S^1)$ for a curve~$\rho$ in~$F$ which is either~$\beta$ or the boundary of a neighborhood of a point in~$K_2 \cap F_1$. Since these curves are pairwise disjoint, the corresponding tori are also pairwise disjoint and therefore have trivial intersection.
Furthermore, since each torus is of the form~$\rho \times S^1$ with~$\rho$ a simple closed curve on an oriented surface, the self intersection of any of these tori is also trivial and the first triple equality is established.
For the remaining equality, observe that
\[
T_{\beta}=\beta\times\partial(J\star[i-c,i+c])=\beta\times\left((J\star\{i-c,i+c\})\cup(\partial J\star[i-c,i+c])\right)
\]
is contained in~$\partial \overline{\nu}(\tilde F)$, and that~$T_{\alpha} \cap \partial \overline{\nu}(\tilde F)$ is equal to~$\alpha \times \{0\}\star \{i+c\}$ for $\alpha\subset F_{i}$. 
Thus, we have the equality
\[
T_{\beta} \cap T_{\alpha} = (\beta \cap \alpha) \times \{0\} \star \{i+c\}\,,
\]
so intersection points in~$T_{\beta} \cap T_{\alpha}$ correspond to intersection points in~$\beta \cap \alpha$.
It can be checked that the chosen conventions for the orientations are such that the signs coincide,
but we do not provide the detailed proof of this fact as it does not affect the validity of Theorem~\ref{thm:ExistenceOfW}~(2).
This establishes the first line and column of the block matrix in the statement.

We now claim that the intersection with the~$T_\alpha$ spheres satisfy
\[
Q_{W_F}(T_{\alpha},T_l) = Q_{W_F}(T_{\alpha},T_{c}) = 0\,,
\]
thus establishing the second line and column of the block matrix.
Indeed, since the tori~$T_l$ and~$T_{c}$ are contained in~$\partial \overline{\nu}(\tilde F)$, their intersection with~$T_{\alpha}$ is contained in~$\alpha \times \{0\} \star \{i+c\}$. As~$\alpha$ is disjoint from a neighborhood of the clasps, there is no such intersection, hence the claimed equalities.

We now check that the intersections with the~$T_l$ tori are given by 
\[
Q_{W_F}(T_l,T_{l'}) =Q_{W_F}(T_l,T_{c}) = Q_{W_F}(T_l,R_{\delta}) = 0\,.
\]
Since~$T_l$ and~$T_{c}$ are tori around clasps, they are pairwise disjoint and therefore have trivial intersection.
The diagonal coefficients vanish as well, since~$T_l$ is the product of a boundary component of~$F_i^\circ$ with a circle and it  can be moved to a close parallel copy.
Finally, the remaining intersection~$T_l \cap R_{\delta}$ is also trivial, as~$T_l$ is included in~$\partial \overline{\nu}(\tilde F)$ and~$R_{\delta} \cap \partial \overline{\nu}(\tilde F) = \tilde \delta$ crosses precisely two ``cancelled'' clasps, while~$T_l$ is a torus around a clasp which is not cancelled.

We now conclude the proof with the claim that the intersections with the~$T_c$ tori are given by 
\[
Q_{W_F}(T_{c},T_{c'}) = 0\,,\quad Q_{W_F}(T_{c_k},R_{\delta_{k'}}) = \begin{cases} \pm 1 & \text{if } k = k'\\ 0 & \text{otherwise.}\end{cases}
\]
The first equality follows from the same arguments as above. Furthermore, since~$T_{c_k}$ is contained in~$\partial \overline{\nu}(\tilde F)$, its only intersections with~$R_{\delta_{k'}}$ are the intersections with~$\tilde \delta_{k'}$. By construction, the clasp corresponding to~$T_{c_k}$ is traversed exactly once by~$\delta_k$ and not crossed by~$\delta_{k'}$ for~$k\neq k'$, hence the claimed formula.
\end{proof}

\begin{remark}
\label{rem:Q-L}
It is not difficult to give a more precise description of the matrix for~$Q_{W_F}$: it is actually given by
\[
\begin{pmatrix}
0 & Q_F(\beta,\alpha) & 0 & 0 & Q_F(\beta, \delta) \\
Q_F(\alpha, \beta)^{\T} & L(\alpha, \alpha) & 0 & 0 & L(\alpha, \delta) \\
0 & 0 & 0 & 0 & 0 \\
0 & 0 & 0 & 0 & \widetilde\id \\
Q_F(\beta, \delta)^{\T} & L(\alpha, \delta)^{\T} & 0 &\widetilde \id & L(\delta,\delta)
\end{pmatrix}\,,
\]
where~$L(\rho, \rho')_{m,m'}$ stands for~$\lk(\rho_m,(\rho'_{m'})^{++})$, see~\cite{GaetanPhD}. 
However, this additional level of detail is not needed for our purposes.
\end{remark}

Equipped with Lemma~\ref{lemma:Matrix}, we now prove the second item of Theorem~\ref{thm:ExistenceOfW}.

\begin{proof}[Proof of Theorem~\ref{thm:ExistenceOfW}~(2)]
If~$\mu=1$, Lemma~\ref{lemma:Matrix} shows that the intersection form~$Q_{W_F}$ is given by a block matrix of size~$g+g$ of the form
\[
\begin{pmatrix}
0 & Q_F(\beta,\alpha) \\
Q_F(\alpha, \beta)^{\T} & \ast
\end{pmatrix}\,,
\]
which is nonsingular (since~$Q_F$ is) and metabolic.
If~$\mu=2$, this same Lemma~\ref{lemma:Matrix} shows that the radical of~$Q_{W_F}$ contains the subgroup~$R\subset H_2(W_F)$
freely generated by the classes of the tori~$T_1,\dots,T_{\vert\ell\vert}$.
Moreover, the form induced on the quotient~$H_2(W_F)/R$ is represented by the matrix
\[
\begin{pmatrix}
0 & Q_F(\beta,\alpha) &  0 & \ast \\
Q_F(\alpha, \beta)^{\T} & \ast &  0 & \ast \\
0 & 0 &  0 & \widetilde \id \\
\ast & \ast & \widetilde \id & \ast
\end{pmatrix}\,,
\]
which is nonsingular since~$\widetilde \id$ and~$Q_F$ are (once restricted to the subgroup generated by the genus curves).
This implies that~$R$ coincides with the radical of~$Q_{W_F}$.
Furthermore, this latter matrix is congruent to the direct sum
\[
\begin{pmatrix}
0 & Q_F(\beta,\alpha) \\
Q_F(\alpha, \beta)^{\T} & \ast
\end{pmatrix}
\oplus
\begin{pmatrix}
0 & \id \\
\id & \ast
\end{pmatrix}
\]
of (nonsingular) metabolic matrices, hence is itself metabolic.
\end{proof}

\subsection{Equivariant intersection form: construction of a basis}
\label{sub:ConstructionSpheres}

The aim of this section is to construct a family of immersed spheres in~$\widetilde{W}_F $, which will be shown in Section~\ref{sub:EquivIntForm} to form a basis of~$H_2(W_F;\Lambda)$.

Throughout this section, we assume the notation of Section~\ref{sub:WF} for our construction of the~$4$-manifold~$W_F$ from a nice~C-complex~$F$ for a~$2$-component link~$L$. In particular, we write~$B$
for the topological~$4$-ball given by the exterior of the trace of the pushed-in~C-complex,
and we use the notation
\[
\{\alpha_1,\ldots,\alpha_{g}, \beta_1,\ldots,\beta_{g}, \gamma_1,\ldots,\gamma_l, \delta_1, \ldots, \delta_n\}
\]
for a good basis of~$H_1(F)$.
From such a good basis, we now construct immersed spheres in~$W_F$ that will later be shown to determine a basis of~$H_2(W_F;\Lambda)$.

\begin{construction}
We construct a family 
\[
\{T_{\alpha_1}, \ldots, T_{\alpha_{g}}, S_{\beta_1}, \ldots, S_{\beta_{g}}, U_{\gamma_1}, \ldots, U_{\gamma_l}, R_{\delta_1}, \ldots, R_{\delta_n}\}
\]
of immersed spheres in~$W_F$.
\begin{itemize}
\item
As described in Section~\ref{sec:ExplicitBasis}, each oriented simple closed curve~$\alpha\in\{\alpha_1,\dots,\alpha_g\}$ yields an oriented immersed sphere~$T_{\alpha} = D_{\alpha} \cup D_m\subset W_F$ where $D_\alpha \subset B$ and~$D_m \subset Y_m$ have boundary the pushed-in curve~$\alpha \times \{0\} \star \{i+c\}$.
\item
For each~$\beta\in\{\beta_1,\dots,\beta_g\}$, the pushoffs $\beta^\pm$ are contained in~$S^3\setminus\nu(F)=\partial B \cap (S^3 \star \{0\})$ and therefore bound immersed discs~$D_{\beta^\pm}\subset B$. We define the immersed sphere
\[
S_{\beta} = D_{\beta^-} \cup (\beta \times J) \cup D_{\beta^+}\subset W_F\,,
\]
and orient it by extending the orientation of~$D_{\beta^-}$ such that~$\partial D_{\beta^-}=\beta^{-}$.
\item For each linking curve~$\gamma\in\{\gamma_1,\dots,\gamma_l\}$, consider the four push-offs of~$\gamma$, namely the  simple closed curves~$\{\gamma^\varepsilon\mid \varepsilon\in\{\pm 1\}^2\}$ which we can assume to intersect transversally.
Since they lie in~$S^3\setminus\nu(F)=\partial B \cap (S^3 \star \{0\})$, each push-off~$\gamma^\varepsilon$ bounds an immersed disc~$D_{\gamma^\varepsilon}\subset B$.
We now show that these~$4$ curves also bound a~$4$-punctured sphere in~$X_L\subset V_F$. Capping it off with the
aforementioned four discs in~$B\subset V_F$ will produce the desired sphere~$U_{\gamma}$ inside~$V_F\subset W_F$.
The construction has its roots in~\cite[Section 4.1.1]{CFT18}.

Recall that~$\gamma$ crosses exactly two clasps.
Away from these clasps, the four push-offs are connected by four bands embedded in~$X_L$: a band~$B_1^+$ (resp.~$B_1^-$) connecting~$\gamma^{++}$ and~$\gamma^{-+}$ (resp.~$\gamma^{+-}$ and~$\gamma^{--})$ near~$F_1$, and a band~$B_2^+$ (resp.~$B_2^-$) connecting~$\gamma^{++}$ and~$\gamma^{+-}$ (resp.~$\gamma^{-+}$ and~$\gamma^{--}$) near~$F_2$, see Figure~\ref{fig:saddle}.
Moreover, near each of these two clasps, the four arcs~$\{\gamma^\varepsilon\}_{\varepsilon}$ are connected by a saddle-shaped disc embedded in~$X_L$, also illustrated in Figure~\ref{fig:saddle}. We write~$S_{\gamma}^1,S_{\gamma}^2$ for these two discs corresponding to the two clasps crossed by~$\gamma$.
Putting all these pieces together, we define the sphere
\[
U_{\gamma} := \bigcup_{\varepsilon \in \{\pm 1\}^2} D_{\gamma^\varepsilon} \cup \bigcup_{i=1,2} \left(B_i^+ \cup B_i^-\right) \cup S_{\gamma}^1 \cup S_{\gamma}^2
\]
immersed in~$V_F\subset W_F$, and we orient it by extending the orientation on~$D_{\gamma^{--}}$ such that~$\partial D_{\gamma^{--}}=\gamma^{--}$.

\begin{figure}
\centering
\begin{overpic}[width=10cm]{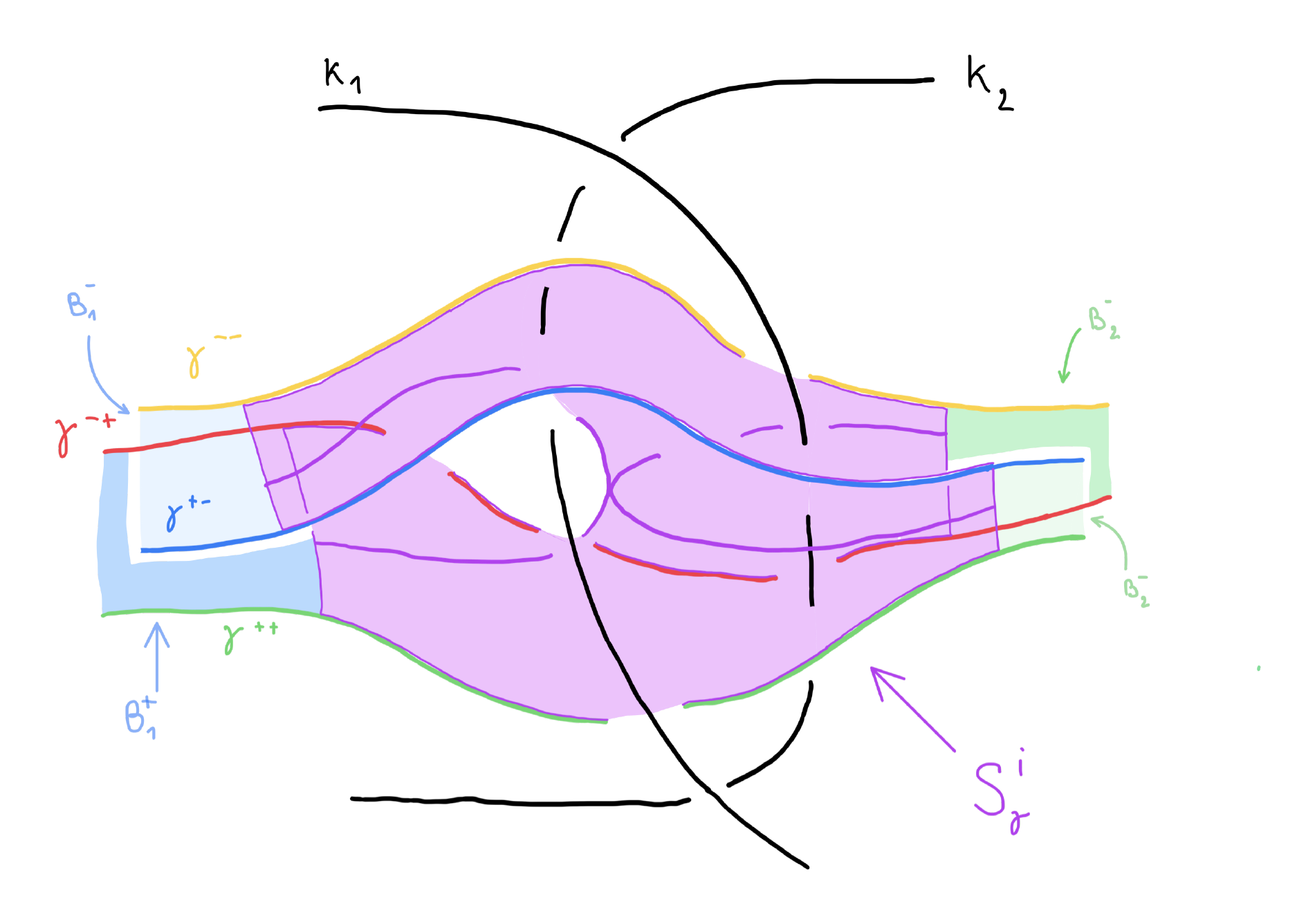}\end{overpic}
\caption{The four bands~$B_i^\pm$ and the saddle~$S_\gamma$ connecting the four pushoffs near a clasp.}
\label{fig:saddle}
\end{figure}

\item As described in Section~\ref{sec:ExplicitBasis}, each oriented cancelling curve~$\delta\in\{\delta_1,\dots,\delta_n\}$ determines an immersed oriented sphere~$R_{\delta}= D_{\delta} \cup D_k\subset W_F$.
Here,  recall that~$D_\delta \subset B$ and~$D_k$ have as common boundary the curve~$\widetilde{\delta} \subset \partial B$ obtained by pushing the push off~$\delta^{++}$ into the~$4$-ball.
The disc~$D_k$ is obtained by attaching to the space illustrated in Figure~\ref{fig:rectangle} one disc in~$Y_k^1\times \{y\}$
along~$\delta^1 \times \{y\}$ and one disc in~$Y_k^2\times\{y\}$ along~$\delta^2 \times \{y\}$.
\end{itemize}
\end{construction}

\medskip

We now wish to pick lifts of these spheres in the universal cover~$\widetilde W_F$.
We do so using {\em based immersions} (see e.g.~\cite[Section 7]{RanickiAlgebraicAndGeometric},~\cite[Section 11.3]{DET} or~\cite[Section 2.2]{KPRT24}), as this will allow us to compute 
the equivariant intersection form via Proposition~\ref{prop:FL1} below.

Recall that if~$(\Sigma,x)$ and~$(W,z)$ are pointed manifolds, a \emph{based immersion} consists of an
immersion~$f \colon \Sigma \to W$ together with a path~$b$ in~$W$ from~$f(x)$ to~$z$, called a \emph{whisker}. 
Let~$\widetilde{W}$ denote the universal cover of~$W$, and let us fix a basepoint~$\tilde z\in\widetilde{W}$ over~$z\in W$.
If the induced homomorphism~$f_\# \colon \pi_1(\Sigma) \to \pi_1(W)$ is trivial, we denote by~$\tilde f \colon \Sigma \to \widetilde W$ the unique lift of~$f$ satisfying~$\tilde f(x) = \tilde b(0)$, where~$\tilde b$ is the unique lift of~$b$ such that~$\tilde b(1) = \tilde z$.  
In particular, if~$\Sigma$ is a closed oriented sphere, then any based immersion~$f\colon \Sigma\to W$ determines a class~$[\tilde{f}(\Sigma)]\in H_2(W; \mathbb{Z}[\pi_1(W)])$.

\begin{convention}
Coming back to our immersed spheres in~$W_F$, we pick our basepoints and whiskers as follows:
a basepoint for~$T_{\alpha}$ in~$D_{\alpha}$, for~$S_{\beta}$ in~$D_{\beta^-}$, for~$U_{\gamma}$ in~$D_{\gamma^{--}}$ and for~$R_{\delta}$ in~$D_{\delta}$. We also fix an arbitrary lift~$\tilde{z}\in\widetilde{W}_F$ of a basepoint~$z\in B\subset W_F$, and we choose all whiskers to lie inside~$B$.
These data uniquely determine lifts of the spheres constructed above, yielding a family 
\[
\mathcal C = \{\tilde T_{\alpha_1}, \ldots, \tilde T_{\alpha_{g}}, \tilde S_{\beta_1}, \ldots, \tilde S_{\beta_{g}}, \tilde U_{\gamma_1}, \ldots,\tilde U_{\gamma_l},\tilde R_{\delta_1}, \ldots,\tilde R_{\delta_n}\}
\]
of classes in~$H_2(W_F;\Z[\pi_1(W_F)])=H_2(W_F;\Lambda)$.
\end{convention}

\subsection{Computation of the equivariant intersection form}
\label{sub:CalculInt}

In this section, we compute the pairwise equivariant intersections of the family~$\mathcal C$.
These computations will be used in Section~\ref{sub:EquivIntForm} to show that~$\mathcal C$ represents a basis of~$H_2(W_F;\Lambda)$, and to prove Theorem~\ref{thm:ExistenceOfW}~(3).
\medskip

To do so, we rely on the following classical result; see e.g.~\cite[Proposition 7.22]{RanickiAlgebraicAndGeometric} or the discussions in~\cite[Section 11.3]{DET} or~\cite[Remark 2.19]{KPRT24}.

\begin{proposition}
\label{prop:FL1}
Let~$(W,z)$ be a pointed four-manifold, and let~$\lambda_W$ denote the equivariant intersection form on~$H_2(W;\Z[\pi_1(W)])$. For~$i = 1,2$, let~$f_i \colon \Sigma_i \to W$ be a based immersion of an oriented, pointed sphere
with whisker~$b_i$, such that~$f_1(\Sigma_1)$ and~$f_2(\Sigma_2)$ intersect transversely. For each intersection point~$p$, choose paths~$\xi_i \subset f_i(\Sigma_i)$ from~$f_i(x_i)$ to~$p$ such that~$f_i^{-1}(\xi_i)$ is a path in~$\Sigma_i$ from~$x_i$ to~$f_i^{-1}(p)$. The concatenated path
\[
\xi_p := b_1^{-1} \cdot \xi_1 \cdot \xi_2^{-1} \cdot b_2
\]
is a loop based at~$z$ and thus determines a class~$[\xi_p]\in\pi_1(W,z)$. We have the formula
\[
\lambda_W([\tilde f_1(\Sigma_1)], [\tilde f_2(\Sigma_2)]) = \sum_{p \in f_1(\Sigma_1) \cap f_2(\Sigma_2)} [\xi_p] \cdot Q_p(f_1(\Sigma_1), f_2(\Sigma_2)) \in \mathbb{Z}[\pi_1(W,z)]\,,
\]
where~$Q_p$ denotes the sign of the intersection point~$p$.\qed
\end{proposition}

We will also make systematic use of the following well-known fact.

\begin{lemma}
\label{lemma:AlgebraicIntersectionNumber}
The algebraic intersection number of two properly immersed discs in a~$4$-ball~$B$ is equal the linking number of their oriented boundaries in~$\partial B$.
\qed
\end{lemma}

Since the family~$\mathcal C$ consists of four different types of elements, there are ten distinct types of intersections to evaluate. None of these computations is conceptually difficult, but they can become rather technical and somewhat tedious.  
For clarity, we organise the computations in order of increasing complexity, starting with the most straightforward cases. Each intersection result is first presented as a linking number, and then checked to coincide with the corresponding
coefficient of the matrix~$H_F$ from Construction~\ref{cons:MatrixH}.

\begin{lemma}
\label{lemma:IntTR}
For each~$\alpha\in\{\alpha_1,\dots,\alpha_g\}$ and~$\delta\in\{\delta_1,\dots,\delta_n\}$, we have the equalities
\[
\lambda_{W_F}(\tilde T_{\alpha},\tilde R_{\delta}) = \lk(\alpha, \delta) =  H_F(\alpha, \delta)\,.
\]
\end{lemma}

\begin{proof}
By construction, we have 
\[
T_{\alpha} \cap R_{\delta} = (D_{\alpha} \cup D_m ) \cap (D_{\delta} \cup D_{k})=D_{\alpha} \cap D_{\delta}\,.
\]
Indeed, the discs~$D_m \subset Y_m$ and~$D_k$ are disjoint since the curves~$\alpha$ and~$\delta$ are disjoint (by definition of a good basis), while~$D_{\alpha} \cap D_k$ (resp.~$D_{\delta}\cap D_m$) is empty as~$D_{\alpha}$ (resp.~$D_{\delta}$) lies in~$B$.
For each intersection point~$p \in D_{\alpha} \cap D_{\delta}$, the path~$\xi_p$ appearing in Proposition~\ref{prop:FL1} is constructed as follows: first, it follows the whisker of~$T_{\alpha}$ from~$z\in W_F$ to the basepoint of~$T_{\alpha}$ in~$D_{\alpha}$, a path which by assumption lies within~$B$; from there, $\xi_p$ follows an arbitrary path within~$D_{\alpha}\subset B$ to the intersection point~$p$, and then an arbitrary path within~$D_{\delta}\subset B$ to the basepoint of~$R_{\delta}$; finally, it returns to~$z$ along the whisker of~$R_{\delta}$, which we assumed to lies within~$B$. In conclusion, the full path~$\xi_p$ is contained in the~$4$-ball~$B$, and therefore represents the trivial element in~$\pi_1(W_F,z)$.
By Proposition~\ref{prop:FL1} and Lemma~\ref{lemma:AlgebraicIntersectionNumber}, we have the equality
\[
\lambda_{W_F}(\tilde T_{\alpha},\tilde R_{\delta}) = \sum_{p \in D_{\alpha} \cap D_{\delta}} Q_p(D_{\alpha},D_{\delta}) = \lk(\partial D_{\alpha}, \partial D_{\delta})\,.
\]
To compute this linking number (of loops that both lie in $\partial B \cong S^3$), observe that for~$\alpha\subset F_i$, the cycles~$\partial D_{\alpha}=\alpha\times\{0\}\star\{i+c\}$ and~$\alpha^{-}=\alpha\times\{-c\}\star \{0\}$ are homologous in~$\partial B$ via the cylinder 
\[
(\alpha \times [-c,0] \star \{i+c\}) \cup (\alpha \times \{-c\} \star [0,i+c])\,.
\]
Moreover, the cycles~$\partial D_{\delta}=\tilde \delta$ and~$\delta^{++}$ are homologous in~$\partial B$ because $\tilde \delta$  was obtained by carefully pushing $\delta^{++}$ into the the $4$-ball.
Explicitly, the homology is given by the cylinder
\begin{equation}
\label{eq:cylinder}
C_{\delta}:=((\delta^{++} \cap (F_1 \times \{c\})) \star [0,1+c]) \;\cup\; ((\delta^{++} \cap (F_2 \times \{c\})) \star [0, 2+c])\,,
\end{equation}
recall the proof of Lemma~\ref{lemma:Claim0} and specifically~\eqref{eq:wtdelta}.
Finally, these two cylinders are disjoint since~$\alpha$ and~$\delta$ are. 
This implies that
\[
\lk(\partial D_{\alpha}, \partial D_{\delta}) = \lk(\alpha^{-}, \delta^{++})=\lk(\alpha,\delta)\,,
\]
using once again the fact that~$\alpha$ and~$\delta$ are disjoint. 
Together with the definition of $H_F$ from Construction~\ref{cons:MatrixH},
this same fact also yields
\[
H_F(\alpha, \delta) = \frac{t_1t_2 \lk(\alpha, \delta^{--}) - t_1 \lk(\alpha, \delta^{-+}) - t_2 \lk(\alpha, \delta^{+-}) + \lk(\alpha, \delta^{++})}{(t_1-1)(t_2-1)} = \lk(\alpha, \delta)\,,
\]
concluding the proof.
\end{proof}

\begin{lemma}
\label{lemma:IntRR}
The equalities
\[
\lambda_{W_F}(\tilde R_{\delta},\tilde R_{\delta'}) = \lk(\delta, \delta'^{++}) =  H_F(\delta, \delta')
\]
hold  for all~$\delta,\delta'\in\{\delta_1,\dots,\delta_n\}$.
\end{lemma}

\begin{proof}
The case~$\delta\neq\delta'$ can be checked by following the proof of Lemma~\ref{lemma:IntTR} almost verbatim:
since the curves~$\delta$ and~$\delta'$ are disjoint,
we have~$R_{\delta} \cap R_{\delta'} = D_{\delta} \cap D_{\delta'}$ and for each intersection point~$p$,
the closed loop~$\xi_p$ lies inside~$B$ and therefore represents the trivial class in~$\pi_1(W_F,z)$.
This yields
\[
\lambda_{W_F}(\tilde R_{\delta},\tilde R_{\delta'}) = \lk(\partial D_{\delta},\partial D_{\delta'})=\lk(\delta^{++}, \delta'^{++})
\]
since~$\partial D_{\delta}=\tilde\delta$ is homologous to~$\delta^{++}$ via the cylinder~$C_{\delta}$ from~\eqref{eq:cylinder}, with~$C_{\delta}\cap C_{\delta'}=\emptyset$ for~$\delta\neq\delta'$.
Using twice more the fact that~$\delta$ and~$\delta'$ are disjoint together with the definition of $H_F$ from Construction~\ref{cons:MatrixH}, we get
\[
\lk(\delta^{++}, \delta'^{++})=\lk(\delta, \delta')=H_F(\delta, \delta')\,,
\]
settling the case~$\delta\neq\delta'$.

In the case~$\delta=\delta'$, one can isotope the sphere~$R_{\delta}=D_{\delta}\cup D_k$ to a nearby copy~$R'_{\delta}=D'_{\delta}\cup D'_k$ such that~$D_{\delta}$ and~$D'_{\delta}$ are disjoint: this can be done by translating~$\delta^i\times\{y\}$ to a parallel copy~$\delta^i\times\{y'\}$ and constructing~$D_k'$ as~$D_k$ using~$y'$ instead of~$y$.
One can then proceed with the proof as in the case~$\delta\neq\delta'$, noting that since~$\delta$ crosses two clasps of opposite signs, the linking number~$\lk(\delta,\delta^\varepsilon)$ does not depend on~$\varepsilon$.
\end{proof}

\begin{lemma}
\label{lemma:IntTS}
For each~$\alpha\in\{\alpha_1,\dots,\alpha_g\}$ and~$\beta\in\{\beta_1,\dots,\beta_g\}$ with~$\beta\subset F_{i}$, we have
\[
\lambda_{W_F}(\tilde T_{\alpha},\tilde S_{\beta}) = \lk(\alpha,\beta^-) - t^{-1}_{i}\lk(\alpha, \beta^+) = H_F(\alpha, \beta)\,.
\]
\end{lemma}

\begin{proof}
First note that
\[
T_{\alpha}\cap S_{\beta} = (D_{\alpha} \cup D_m) \cap (D_{\beta^-} \cup (\beta \times J) \cup D_{\beta^+})=(D_{\alpha} \cap D_{\beta^+})\cup(D_{\alpha} \cap D_{\beta^-})\,,
\]
as~$D_m\subset Y_m \times S^1$ is disjoint from~$S_{\beta} \subset V_F$ and~$\beta \times J\subset F_{i} \times J$ is disjoint from~$D_{\alpha} \subset B$.
If a point~$p$ lies in~$D_{\alpha} \cap D_{\beta^-}$, then by our conventions for basepoints and whiskers, the path~$\xi_p$ is entirely contained in~$B$ and hence represents the trivial element in~$\pi_1(W_F,z)$. On the other hand, for~$p \in D_{\alpha} \cap D_{\beta^+}$, the path~$\xi_p$ is constructed as follows: starting at the basepoint~$z$, it follows the whisker of~$T_\alpha$ to~$D_{\alpha}$, proceeds within~$D_{\alpha}$ to~$p$, then traverses the annulus~$\beta \times J$ to reach the basepoint of~$S_{\beta}$ in~$D_{\beta^-}$, and finally returns to~$z$ along the whisker of~$S_{\beta}$.
Hence, this loop always lies within~$B$ except when it crosses the annulus~$\beta \times J$ from~$D_{\beta^+}$ to $D_{\beta^-}$. Since the orientation
of a meridian of~$K_i$ corresponds to the direction from~$F_{i} \times \{-c\} \star \{0\}$ to~$F_{i} \times \{c\} \star \{0\}$ in~$F_{i} \times J \subset S^3 \star \{0\}$, we get~$[\xi_p]=t_i^{-1}$. This is illustrated in Figure~\ref{fig:class-path}.
In summary, we have
\[
\lambda_{W_F}(\tilde T_{\alpha},\tilde S_{\beta}) = \sum_{p\in D_{\alpha} \cap D_{\beta^-}} Q_p(D_{\alpha}, D_{\beta^-}) 
+ t_{i}^{-1}\sum_{p \in D_{\alpha} \cap D_{\beta^+}}  Q_p(D_{\alpha}, D_{\beta^+}).
\]
For~$\alpha\subset F_j$, Lemma~\ref{lemma:AlgebraicIntersectionNumber} together with our convention for the orientation of~$S_\beta$ yield
\[
\sum_{p \in D_{\alpha} \cap D_{\beta^\pm}} Q_p(D_{\alpha}, D_{\beta^\pm}) =\mp \lk(\partial D_{\alpha},\partial D_{\beta^{\pm}})=\mp \lk(\alpha \times \{0\} \star \{j+c\}, \beta^{\pm})\,.
\]
Furthermore, the cylinder~$\beta \times \{\pm c\} \star [0, j+c]$
ensures that the cycles~$\beta^\pm$ and~$\beta \times \{\pm c \} \star \{j+c\}$ are homologous in~$\partial B \setminus (\alpha \times \{0\} \star \{j+c\})$, so
\[
\lk(\alpha \times \{0\} \star \{j+c\}, \beta^{\pm})=\lk(\alpha \times \{0\} \star \{j+c\}, \beta \times \{\pm c\} \star \{j+c\}) = \lk(\alpha, \beta^\pm)\,.
\]
The three equations displayed above imply the first equality in the statement.

The second equality follows from the definition of~$H_F$ if~$\mu=1$,
together with the fact that~$\lk(\alpha,\beta^\varepsilon)$ does not depend on~$\varepsilon_1$ (resp.~$\varepsilon_2$)
for~$\beta\subset F_2$ (resp.~$\beta\subset F_1$) if~$\mu=2$.
\end{proof}

\begin{figure}
\centering
\begin{overpic}[width=12cm]{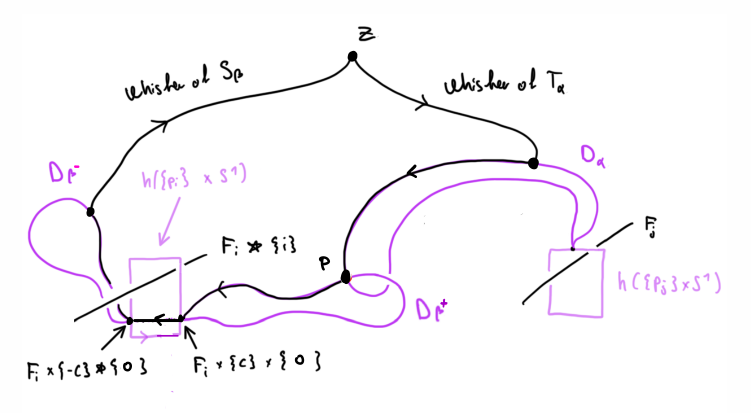}\end{overpic}
\caption{A schematic picture of the path~$\xi_p$ for~$p \in D_{\alpha} \cap D_{\beta^+}$.}
\label{fig:class-path}
\end{figure}

\begin{lemma}
\label{lemma:IntTT}
For each~$\alpha,\alpha'\in\{\alpha_1,\dots,\alpha_g\}$, we have
\[
\lambda_{W_F}(\tilde T_{\alpha},\tilde T_{\alpha'}) = \lk(\alpha, \alpha'^+) = H_F(\alpha,\alpha')\,.
\]
\end{lemma}

\begin{proof}
For~$\alpha\neq \alpha'$, the cycles~$\alpha$ and~$\alpha'$ are disjoint, and one shows~$\lambda_{W_F}(\tilde T_{\alpha},\tilde T_{\alpha'})=\lk(\partial D_{\alpha}, \partial D_{\alpha'})$ as in the proof of Lemma~\ref{lemma:IntTR}.
For~$\alpha=\alpha'$, since~$\alpha$ embeds in an orientable surface,
one can perturb~$T_{\alpha}=D_{\alpha}\cup D_m$ to~$T'_{\alpha}=D'_{\alpha}\cup D'_m$
such that~$\partial D_{\alpha}$ and~$\partial D'_{\alpha}$ are disjoint,
and we get the equality~$\lambda_{W_F}(\tilde T_{\alpha},\tilde T_{\alpha})=\lk(\partial D_{\alpha}, \partial D'_{\alpha})$.

If~$\alpha\neq \alpha'$ but~$\alpha$ and~$\alpha'$ lie on the same surface~$F_i$, then~$\partial D_{\alpha} = \alpha \times \{0\} \star \{i+c\}$ and~$\partial D_{\alpha'} = \alpha' \times \{0\} \star \{i+c\}$ both lie in~$S^3 \star \{i+c\}$, yielding
\[
\lk(\partial D_{\alpha}, \partial D_{\alpha'})= \lk(\alpha, \alpha')= \lk(\alpha, \alpha'^\varepsilon)
\]
for any~$\varepsilon$ since~$\alpha$ and~$\alpha'$ are disjoint. The case~$\alpha=\alpha'$ yields~$\lk(\partial D_{\alpha},\partial D'_{\alpha})= \lk(\alpha, \alpha^\varepsilon)$ for all~$\varepsilon$.
We are left with the case of~$\alpha$ and~$\alpha'$ on different surfaces, say~$\alpha \subset F_1$ and~$\alpha'\subset F_2$. We want to compute
\[
\lk(\partial D_{\alpha}, \partial D_{\alpha'})=\lk(\alpha \times \{0\} \star \{1+c\}, \alpha' \times \{0\} \star \{2+c\})\,.
\]
Since~$\alpha$ and~$\alpha'$ are disjoint, the cylinder
\[
(\alpha' \times [0,c] \star \{2+c\}) \;\cup\; (\alpha' \times \{c\} \star [1+c, 2+c])\;\subset\; \partial B \setminus (\alpha \times \{0\} \star \{1+c\})
\]
ensures that the cycles~$\alpha' \times \{0\} \star \{2+c\}$ and~$\alpha' \times \{c\} \star \{1+c\}$ are homologous in~$\partial B \setminus (\alpha \times \{0\} \star \{1+c\})$. This yields
\[
\lk(\partial D_{\alpha}, \partial D_{\alpha'}) = \lk(\alpha \times \{0\} \star \{1+c\}, \alpha' \times \{c\} \star \{1+c\})=\lk(\alpha, \alpha'^+)\,,
\]
completing the proof of the first equality.

The second equality follows from the definition of~$H_F$ together with the fact that~$\lk(\alpha,\alpha'^\varepsilon)$ does not depend on~$\varepsilon$.
\end{proof}

\begin{lemma}
\label{lemma:IntSS}
For any~$\beta,\beta'\in\{\beta_1,\dots,\beta_g\}$ with~$\beta\subset F_i$ and~$\beta'\subset F_j$, we have
\[
\lambda_{W_F}(\tilde S_{\beta},\tilde S_{\beta'}) = (t_{i}-1) (t_{j}^{-1} \lk(\beta, \beta'^+) - \lk(\beta, \beta'^-)) = H_F(\beta,\beta')\,.
\]
\end{lemma}

\begin{proof}
Let us first assume~$\beta\neq\beta'$. Since~$\beta$ and~$\beta'$ are disjoint, the usual arguments imply
\[
S_{\beta} \cap S_{\beta'} = (D_{\beta^-} \cup (\beta \times J) \cup D_{\beta^+}) \cap (D_{\beta'^-} \cup (\beta' \times J) \cup D_{\beta'^+})=\bigcup_{\varepsilon,\varepsilon\in\{\pm\}}(D_{\beta^\varepsilon}\cap D_{\beta'^{\varepsilon'}})\,,
\]
so Proposition~\ref{prop:FL1} yields
\begin{equation}
\label{eq:SS1}
\lambda_{W_F}(\tilde S_{\beta},\tilde S_{\beta'}) = \sum_{\varepsilon, \varepsilon'\in\{\pm\}} \sum_{p\in D_{\beta^\varepsilon} \cap D_{\beta'^{\varepsilon'}}} [\xi_p]\cdot Q_p(D_{\beta^\varepsilon}, D_{\beta'^{\varepsilon'}})\,.
\end{equation}
For all~$\varepsilon,\varepsilon'\in\{\pm\}$, Lemma~\ref{lemma:AlgebraicIntersectionNumber} implies
\begin{equation}
\label{eq:SS2}
\sum_{p\in D_{\beta^\varepsilon}\cap D_{\beta'^{\varepsilon'}}} Q_p(D_{\beta^\varepsilon}, D_{\beta'^{\varepsilon'}})=\lk(\partial D_{\beta^\varepsilon},\partial D_{\beta'^{\varepsilon'}}) = \varepsilon\varepsilon'\, \lk(\beta^\varepsilon, \beta'^{\varepsilon'})= \varepsilon\varepsilon'\, \lk(\beta, \beta'^{\varepsilon'})\,,
\end{equation}
with the sign~$\varepsilon\varepsilon'$ coming from our orientation conventions.
If~$\beta=\beta'$, we can perturb~$\beta$ to~$\beta'$ disjoint from~$\beta$ and follow the same argument,
so equations~\eqref{eq:SS1} and~\eqref{eq:SS2} still hold.

We now compute the class of~$\xi_p$ for~$p\in D_{\beta^\varepsilon}\cap D_{\beta'^{\varepsilon'}}$, which turns out only to depend on~$\varepsilon,\varepsilon'$. Indeed, if~$\varepsilon=\varepsilon'=-1$, then~$\xi_p$ lies inside~$B$
and we have~$[\xi_p]=1$. By~\eqref{eq:SS1} and~\eqref{eq:SS2}, the corresponding contribution to the equivariant intersection form is~$\lk(\beta,\beta'^-)$.
For~$(\varepsilon,\varepsilon')=(-1,+1)$, the path~$\xi_p$ winds negatively around~$F_{j}$,
yielding a class~$[\xi_p]=t^{-1}_j$ and a total contribution of~$-t_{j}^{-1} \lk(\beta, \beta'^+)$.
Similarly, the contribution of~$(\varepsilon,\varepsilon')=(+1,-1)$ is equal to~$-t_{i} \lk(\beta, \beta'^-)$, while
the contribution of~$(\varepsilon,\varepsilon')=(+1,+1)$ is~$t_{i} t_{j}^{-1} \lk(\beta, \beta'^+)$.
Summing all four contributions shows the first equality in the statement.

The second equality follows from the definition of~$H_F$ together with the fact that~$\lk(\beta,\beta'^{\varepsilon'})$ does not depend on~$\varepsilon'$.
\end{proof}

Our computations of intersections of~$U_\gamma$-spheres will make use of the following lemma.

\begin{lemma}
\label{lemma:Pi1Term}
For a linking curve~$\gamma\in\{\gamma_1,\dots,\gamma_l\}$ and~$p \in D_{\gamma^{\varepsilon}}$, consider a path~$\eta_p$
of the following form: it starts at the basepoint~$z\in B$, follows any path in~$B$ to reach~$p$, proceeds to the basepoint of~$U_\gamma$ along a path contained in~$U_\gamma$, and returns to~$z$ inside~$B$.
Then, its class is given~by
\[
[\eta_p] = t_1^{-\frac{\varepsilon_1+1}{2}} t_2^{-\frac{\varepsilon_2+1}{2}}\in\pi_1(W_F,z)\,.
\]
\end{lemma}

\begin{proof}
The initial and third portions of~$\eta_p$ lie in~$B$, so we only need to understand the contribution of its second part.
This middle portion connects~$p \in D_{\gamma^{\varepsilon}}$ to the basepoint of~$U_\gamma$, which by convention lies in~$D_{\gamma^{--}}$. Therefore, for each~$i$ such that~$\varepsilon_i = +1$, the path~$\eta_p$ crosses a band~$B^{\pm}_i$ from~$F_i\times\{c\}$ to~$F_i\times\{-c\}$. This corresponds to winding once around a meridian of~$K_i$ in the negative direction, yielding a factor~$t_i^{-1}=t_i^{-\frac{\varepsilon_i+1}{2}}$. 
For indices~$i$ with~$\varepsilon_i = -1$, the path~$\eta_p$ does not wind around~$F_i$, contributing a factor~$1 = t_i^{-\frac{\varepsilon_i+1}{2}}$.
\end{proof}

\begin{lemma}
\label{lemma:IntTU}
For each~$\alpha\in\{\alpha_1,\dots,\alpha_g\}$ and~$\gamma\in\{\gamma_1,\dots,\gamma_l\}$, we have
\[
\lambda_{W_F}(\tilde T_{\alpha},\tilde U_{\gamma}) = \sum_{\varepsilon} \varepsilon_1\varepsilon_2\, t_1^{-\frac{\varepsilon_1+1}{2}} t_2^{-\frac{\varepsilon_2+1}{2}} \, \lk(\alpha, \gamma^\varepsilon) = H_F(\alpha, \gamma)\,.
\]
\end{lemma}

\begin{proof}
By construction, we have
\[
T_{\alpha} \cap U_{\gamma} = (D_{\alpha} \cup D_m) \cap \Big( \bigcup_{\varepsilon \in \{\pm1\}^2} D_{\gamma^\varepsilon} \cup \bigcup _{i=1,2} \left(B_i^+ \cup B_i^-\right) \cup S_{\gamma}^1 \cup S_{\gamma}^2 \Big)=D_{\alpha} \cap \Big( \bigcup_{\varepsilon \in \{\pm1\}^2} D_{\gamma^\varepsilon} \Big)\,,
\]
since~$D_m\subset Y_m$ is disjoint from~$U_{\gamma}$ while~$D_{\alpha} \subset B$ meets neither~$B_i^\pm$ nor~$S_{\gamma}^1\cup S_{\gamma}^2$. 
For each intersection point~$p \in D_{\alpha} \cap D_{\gamma^\varepsilon}$, the path~$\xi_p=b_1^{-1}\cdot\xi_1\cdot\xi^{-1}_2\cdot b_2$ from Proposition~\ref{prop:FL1} is such that~$b_1^{-1}\cdot\xi_1\subset B$ and~$b_2\subset B$. By this proposition together with Lemmas~\ref{lemma:Pi1Term} and~\ref{lemma:AlgebraicIntersectionNumber}, the contribution to~$\lambda_{W_F}(\tilde T_{\alpha},\tilde U_{\gamma})$ of each~$\varepsilon \in \{\pm1\}^2$ is equal to 
\[ 
t_1^{-\frac{\varepsilon_1+1}{2}} t_2^{-\frac{\varepsilon_2+1}{2}} \sum_{p \in D_{\alpha} \cap D_{\gamma^\varepsilon}} Q_p(D_{\alpha}, D_{\gamma^\varepsilon})=\varepsilon_1\varepsilon_2\,t_1^{-\frac{\varepsilon_1+1}{2}} t_2^{-\frac{\varepsilon_2+1}{2}}\lk(\partial D_{\alpha},\gamma^\varepsilon)\,,
\]
with the sign coming from our convention for the orientation of~$U_{\gamma}$.
In order to calculate these linking number, recall that  for~$\alpha\subset F_i$, we have~$\partial D_{\alpha}=\alpha \times \{0\} \star \{i+c\}$.

In the case that~$\varepsilon_i =+1$, the cylinders
\[
(\alpha \times [-c,0] \star \{i+c\}) \cup (\alpha \times \{-c\} \star [0,i+c])\subset \partial B\setminus\gamma^\varepsilon
\]
and~$\alpha \times [-c,0] \subset S^3\setminus\gamma^\varepsilon$ show that
\[
\lk(\alpha \times \{0\} \star \{i+c\}, \gamma^\varepsilon) = \lk(\alpha^-, \gamma^\varepsilon)= \lk(\alpha, \gamma^\varepsilon)\,. 
\]
The case~$\varepsilon_i=-1$ is similar and yields~$\lk(\partial D_{\alpha},\gamma^\varepsilon)=\lk(\alpha,\gamma^\varepsilon)$.
In conclusion, we obtain
\[
\lambda_{W_F}(\tilde T_{\alpha},\tilde U_{\gamma}) 
= \sum_{\varepsilon} \varepsilon_1\varepsilon_2\, t_1^{-\frac{\varepsilon_1+1}{2}} t_2^{-\frac{\varepsilon_2+1}{2}} \, \lk(\alpha, \gamma^\varepsilon)
=H_F(\alpha, \gamma)\,.\qedhere
\]
\end{proof}

\begin{lemma}
\label{lemma:IntSU}
For any~$\beta\in\{\beta_1,\dots,\beta_g\}$ and~$\gamma\in\{\gamma_1,\dots,\gamma_l\}$ with~$\beta\subset F_i$,
we have
\[
\lambda_{W_F}(\tilde S_{\beta},\tilde U_{\gamma}) = (1-t_{i}) \sum_{\varepsilon \in \{\pm1\}^2} \varepsilon_1\varepsilon_2 \, t_1^{-\frac{\varepsilon_1 + 1}{2}} t_2^{-\frac{\varepsilon_2 + 1}{2}}  \lk(\beta, \gamma^\varepsilon) = H_F(\beta, \gamma)\,.
\] 
\end{lemma}

\begin{proof}
Since~$\beta$ and~$\gamma$ are disjoint, the usual arguments lead to
\begin{align*}
S_\beta \cap U_{\gamma} &= (D_{\beta^-} \cup (\beta \times J) \cup D_{\beta^+}) \cap \Big( \bigcup_{\varepsilon \in \{\pm\}^2} D_{\gamma^{\varepsilon}} \cup \bigcup _{i=1,2} \left(B_i^+ \cup B_i^-\right) \cup S_{\gamma}^1 \cup S_{\gamma}^2 \Big)\\
	&=\bigcup_{\varepsilon'\in\{\pm\}}D_{\beta^{\varepsilon'}}\cap\bigcup_{\varepsilon\in \{\pm\}^2} D_{\gamma^{\varepsilon}}\,,
\end{align*}
so Proposition~\ref{prop:FL1} yields
\[
\lambda_{W_F}(\tilde S_{\beta},\tilde U_{\gamma} ) = \sum_{\varepsilon',\varepsilon} \sum_{p\in D_{\beta^{\varepsilon'}} \cap D_{\gamma^\varepsilon}} [\xi_p]\cdot Q_p(D_{\beta^{\varepsilon'}}, D_{\gamma^\varepsilon})\,.
\]
For each~$p \in D_{\beta^-} \cap D_{\gamma^\varepsilon}$, the path~$\xi_p=b_1^{-1}\cdot\xi_1\cdot\xi^{-1}_2\cdot b_2$ from Proposition~\ref{prop:FL1} satisfies~$b_1^{-1}\cdot\xi_1\subset B$ and~$b_2\subset B$. 
By Lemma~\ref{lemma:Pi1Term}, its class is equal to~$[\xi_p] = t_1^{-\frac{\varepsilon_1 + 1}{2}} t_2^{-\frac{\varepsilon_2 + 1}{2}}$. 
For~$p \in D_{\beta^+} \cap D_{\gamma^\varepsilon}$, since the path~$\xi_1$ crosses~$ \beta \times J$ in the positive direction, 
a similar calculation whose details we omit gives~$[\xi_p]=t_i \cdot t_1^{-\frac{\varepsilon_1 + 1}{2}} t_2^{-\frac{\varepsilon_2 + 1}{2}}$.
This leads to
\[
\lambda_{W_F}(\tilde S_{\beta},\tilde U_{\gamma} ) = \sum_{\varepsilon',\varepsilon} \,t_i^{\frac{\varepsilon'+1}{2}} t_1^{-\frac{\varepsilon_1 + 1}{2}} t_2^{-\frac{\varepsilon_2 + 1}{2}}\sum_{p\in D_{\beta^{\varepsilon'}} \cap D_{\gamma^\varepsilon}} Q_p(D_{\beta^{\varepsilon'}}, D_{\gamma^\varepsilon})\,.
\]
Furthermore, Lemma~\ref{lemma:AlgebraicIntersectionNumber} together with our orientation conventions imply
\[
\sum_{p\in D_{\beta^{\varepsilon'}} \cap D_{\gamma^\varepsilon}} Q_p(D_{\beta^{\varepsilon'}}, D_{\gamma^\varepsilon})
=\lk(\partial D_{\beta^{\varepsilon'}},\partial D_{\gamma^\varepsilon})
= -\varepsilon'\varepsilon_1\varepsilon_2\lk(\beta^{\varepsilon'}, \gamma^\varepsilon)
= -\varepsilon'\varepsilon_1\varepsilon_2\lk(\beta, \gamma^\varepsilon)
\]
for all~$\varepsilon'\in\{\pm\}$ and~$\varepsilon\in\{\pm\}^2$,
and we get
\[
\lambda_{W_F}(\tilde S_{\beta},\tilde U_{\gamma}) = (1-t_{i}) \sum_{\varepsilon \in \{\pm1\}^2} \varepsilon_1\varepsilon_2 \, t_1^{-\frac{\varepsilon_1 + 1}{2}} t_2^{-\frac{\varepsilon_2 + 1}{2}}  \lk(\beta, \gamma^\varepsilon)\,.
\]
The final equality follows from the definition of~$H_F$ together with the fact that~$\lk(\beta, \gamma^\varepsilon)$ does not depend on~$\varepsilon$.
\end{proof}

\begin{lemma}
\label{lemma:IntSR}
For each~$\beta\in\{\beta_1,\dots,\beta_g\}$ and~$\delta\in\{\delta_1,\dots,\delta_n\}$ with~$\beta\subset F_{i}$, we have
\[
\lambda_{W_F}(\tilde S_{\beta},\tilde R_{\delta}) = \lk(\beta^-, \delta)- t_{i} \lk(\beta^+, \delta) = H_F(\beta, \delta)\,.
\]
\end{lemma}

\begin{proof}
Since~$S_{\beta}\subset V_F$ is disjoint from~$D_{k}$ while~$(\beta \times J) \subset F_{i} \times J$ is disjoint from~$D_{\delta} \subset B$, we have
\[
S_{\beta} \cap R_{\delta} = (D_{\beta^-} \cup (\beta \times J) \cup D_{\beta^+}) \cap (D_{\delta} \cup D_k)=
(D_{\beta^-} \cap D_{\delta}) \; \cup \; (D_{\beta^+} \cap D_{\delta})\,.
\]
For~$p \in D_{\beta^-} \cap D_{\delta}$, the path~$\xi_p$ lies inside~$B$, yielding~$[\xi_p]=1\in\pi_1(W_F,z)$.
For~$p \in D_{\beta^+} \cap D_{\delta}$, the path $\xi_p$ crosses once the annulus~$\beta \times J$ from~$\beta^-$ to~$\beta^+$, yielding~$[\xi_p]=t_i\in \pi_1(W_F,z)$.
Proposition~\ref{prop:FL1} and Lemma~\ref{lemma:AlgebraicIntersectionNumber} then imply the equality
\[
\lambda_{W_F}(\tilde S_{\beta},\tilde R_{\delta}) = \lk(\beta^-,\tilde \delta)-t_{i} \lk(\beta^+,\tilde \delta)\,,
\]
where the sign comes from our convention for the orientation of~$S_{\beta}$. Hence, it remains to check the equality~$\lk(\beta^\pm,\tilde \delta)=\lk(\beta^\pm,\delta)$, where the first linking number is in~$\partial B$ and the second one in~$S^3\star\{0\}$. (Recall that~$\beta^\pm$ belongs to the intersection~$\partial B\cap(S^3\star\{0\})=S^3\setminus\nu(F)$.)

\begin{figure}
\centering
\begin{overpic}[width=15cm]{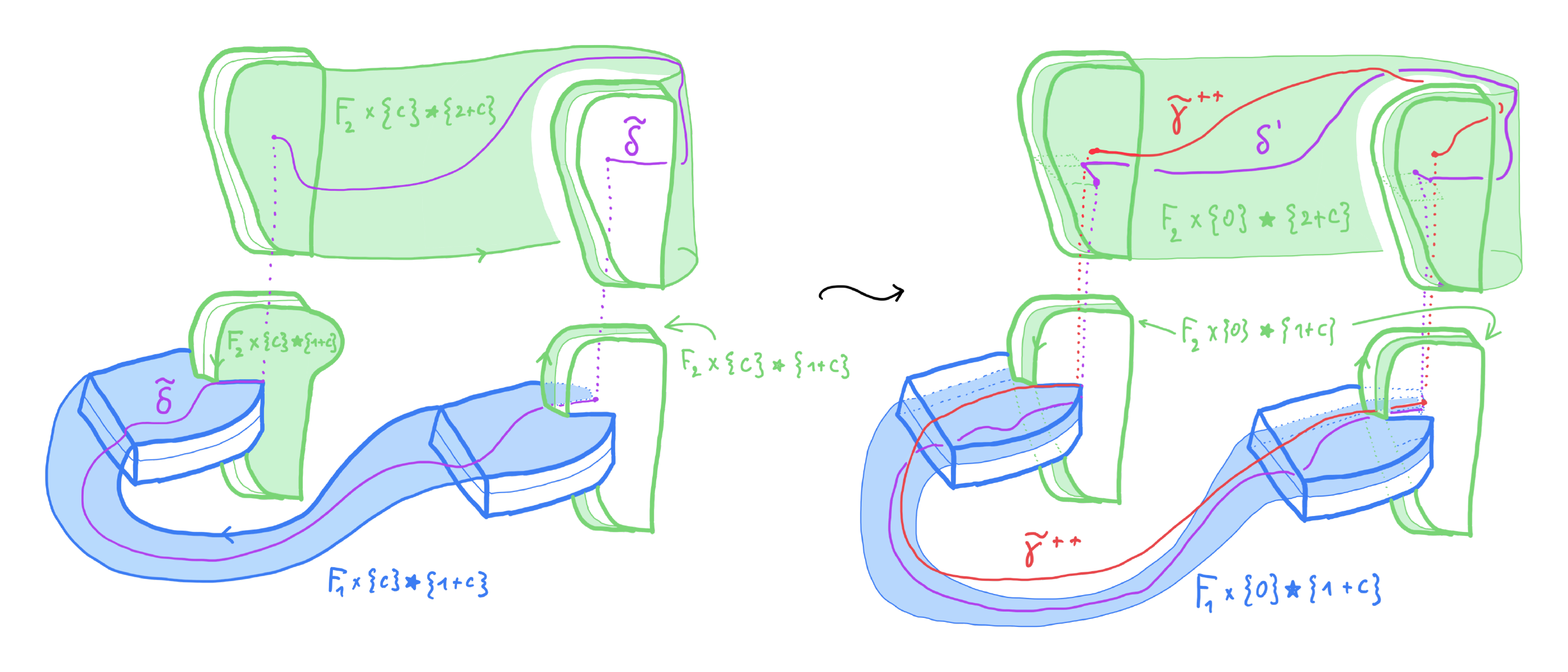}\end{overpic}
\caption{The cycle~$\delta'$ obtained by sliding the arcs contained in~$F_j \times \{c\} \star \{j+c\}$ to~$F_j \times \{0\} \star \{j+c\}$.}
\label{fig:SlidingDeltaUR}
\end{figure}

To check this fact, we slide the arcs~$(\delta^{++} \cap (F_j \times \{c\}))  \star \{j+c\}$ of~$\tilde \delta$ to~$(\delta^{++} \cap (F_j \times \{0\}) ) \star \{j+c\}$ in $\partial B$ for~$j = 1,2$, as illustrated by the purple curve in Figure~\ref{fig:SlidingDeltaUR};
in order to still have a loop, we also move the arcs~$\partial (\delta^{++} \cap (F_1 \times \{c\})) \star [1+c, 2+c]$ to parallel copies in~$(F_1 \times \{0\}) \star [1+c, 2+c]$,
and denote the resulting curve by~$\delta'\subset \partial B$. 
Note that this deformation is performed at depth range~$[1+c,2+c]$, and therefore does not cross the cycle~$\beta^\pm$ which lies at depth~$0$.
We then bring the cycle~$\beta^\pm$ to depth~$i + c$ via the cylinder~$\beta \times \{\pm c\} \star [0,i +c]\subset F_i\times\{\pm c\}\star [0,i+c]\subset\partial B\setminus \delta'$, yielding
\[
\lk(\beta^\pm,\tilde \delta)=\lk(\beta \times \{\pm c\} \star \{i + c\},\delta')\,.
\]
Finally, the triple~$(\partial B,\beta \times \{\pm c\} \star \{i + c\},\delta')$ is homeomorphic to~$(S^3, \beta^\pm,\delta)$, as illustrated in Figure~\ref{fig:HomeoBSR} (left), and the claim follows.

The second equality follows from the definition of~$H_F$ together with the fact that~$\lk(\beta^\varepsilon,\delta)$ only depends on~$\varepsilon_{i}$ for~$\beta\subset F_i$.
\end{proof}

\begin{figure}
\centering
\begin{overpic}[width=0.49\textwidth]{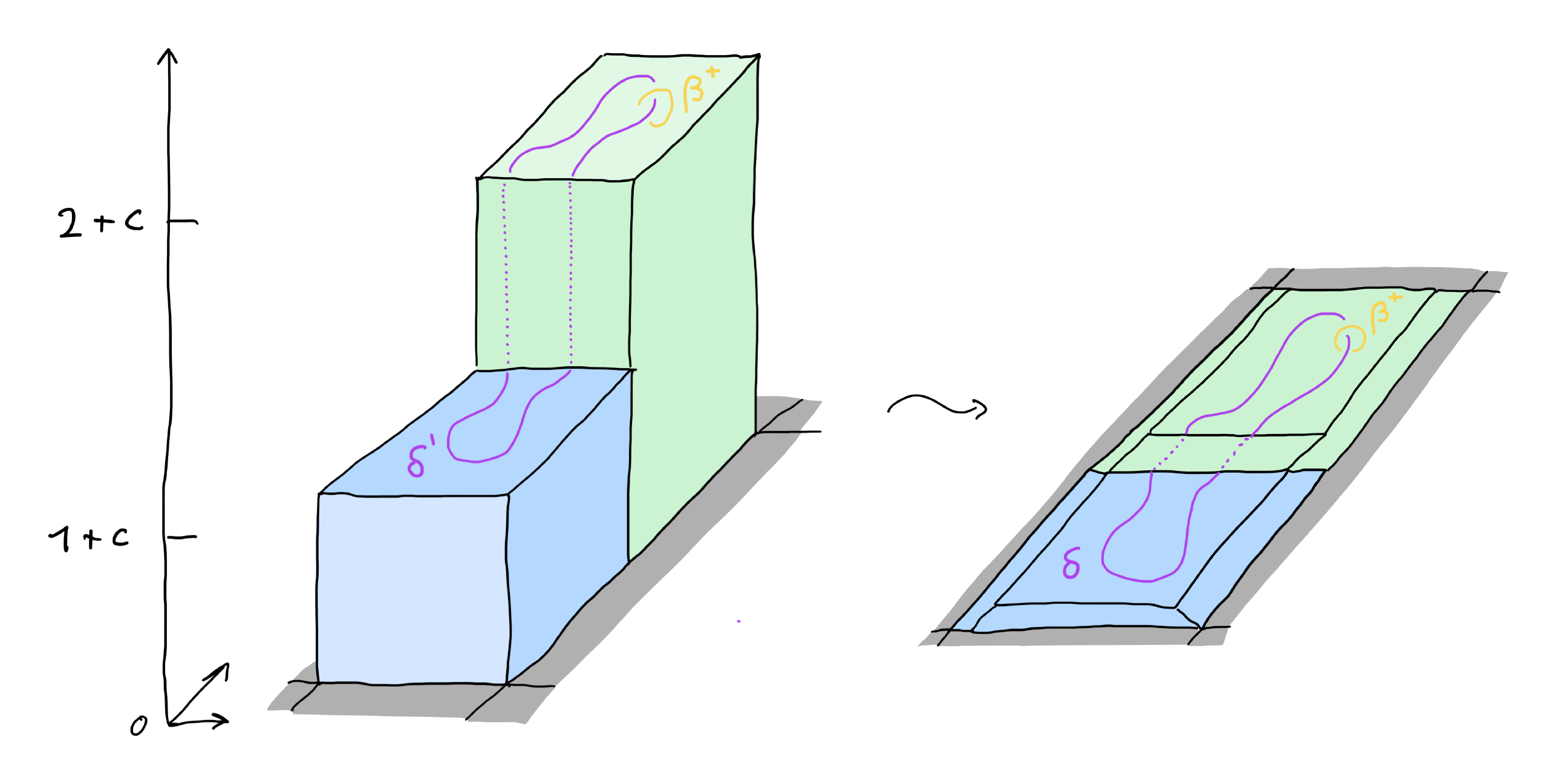}\end{overpic}
\begin{overpic}[width=0.49\textwidth]{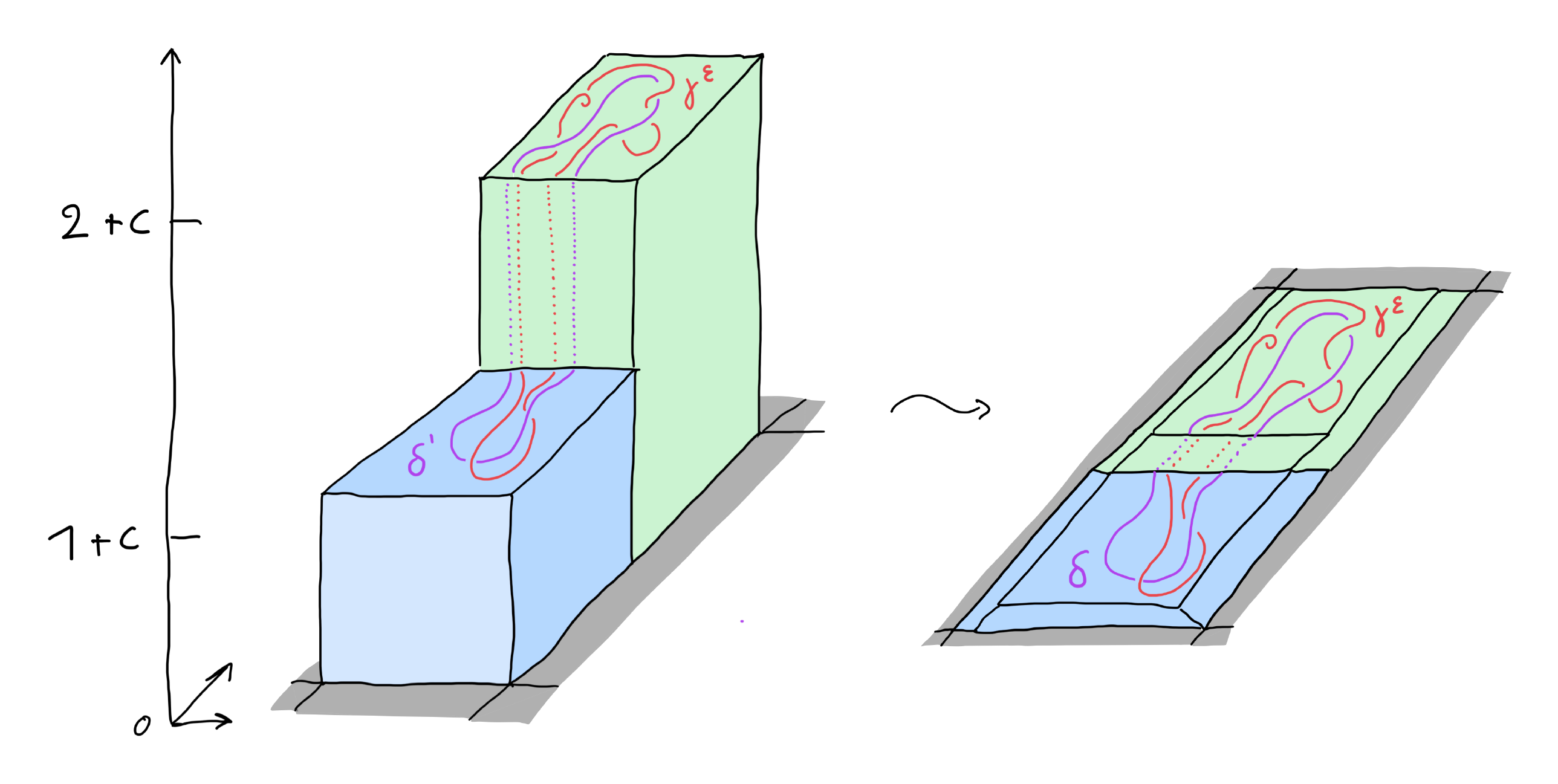}\end{overpic}
\caption{The triples~$(\partial B,\beta \times \{\pm c\} \star \{i + c\},\delta')$ and~$(S^3, \beta^\pm, \delta)$ are homeomorphic, as well as the triples~$(\partial B, \tilde \gamma^\varepsilon, \delta')$ and~$(S^3, \gamma^\varepsilon, \delta)$.}
\label{fig:HomeoBSR}
\end{figure}

\begin{lemma}
\label{lemma:IntUR}
For all~$\gamma\in\{\gamma_1,\dots,\gamma_l\}$ and~$\delta\in\{\delta_1,\dots,\delta_n\}$, we have
\[
\lambda_{W_F}(\tilde U_{\gamma},\tilde R_{\delta}) = \sum_{\varepsilon \in \{\pm1\}^2} \varepsilon_1\varepsilon_2 \, t_1^{\frac{\varepsilon_1 +1}{2}} t_2^{\frac{\varepsilon_2+1}{2}}\, \lk(\gamma^\varepsilon,\delta)= H_F(\gamma, \delta)\,.
\]
\end{lemma}

\begin{proof}
By the usual arguments, all the intersection points lie inside
\[
U_{\gamma}\cap R_\delta=\Big( \bigcup_{\varepsilon \in \{\pm1\}^2} D_{\gamma^\varepsilon} \Big)\cap D_\delta\,.
\]
Since~$D_{\delta}$ is contained in~$B$, Lemma~\ref{lemma:Pi1Term} applied to~$\eta_p=\xi_p^{-1}$ shows that the loop~$\xi_p$ associated to~$p\in D_{\gamma^\varepsilon} \cap D_{\delta}$ has class~$[\xi_p] = t_1^{\frac{\varepsilon_1 +1}{2}} t_2^{\frac{\varepsilon_2+1}{2}}\in\pi_1(W,z)$. Thus, Proposition~\ref{prop:FL1} yields
\[
\lambda_{W_F}(\tilde U_{\gamma},\tilde R_{\delta}) = \sum_{\varepsilon \in \{\pm1\}^2} t_1^{\frac{\varepsilon_1 +1}{2}} t_2^{\frac{\varepsilon_2+1}{2}} \,\lk(\partial D_{\gamma^\varepsilon}, \partial D_{\delta}) = \sum_{\varepsilon \in \{\pm1\}^2} \varepsilon_1\varepsilon_2 \, t_1^{\frac{\varepsilon_1 +1}{2}} t_2^{\frac{\varepsilon_2+1}{2}}\, \lk(\gamma^\varepsilon, \tilde \delta)\,,
\]
and we are left with checking the equality~$\lk(\gamma^\varepsilon, \tilde \delta) = \lk(\gamma^\varepsilon, \delta)$.

We do so in three steps. First, as in the proof of Lemma~\ref{lemma:IntSR}, we move~$\tilde\delta$ to~$\delta'$ inside~$\partial B$ (recall Figure~\ref{fig:SlidingDeltaUR}) without crossing~$\gamma^\varepsilon$ which lies at depth~$0$.
Second, we bring the curve~$\gamma^\varepsilon$ to a curve~$\tilde\gamma^\varepsilon$ at the depth of~$\delta'$ via the cylinder 
\[
\big( (\gamma^\varepsilon \cap (F_1 \times \{\varepsilon_1 c\})) \star [0,1+c] \big) \cup\big( \partial (\gamma^\varepsilon \cap (F_1 \times \{\varepsilon_1 c\})) \star [0,2+c] \big) \cup \big( (\gamma^\varepsilon \cap \overline{(F_2 \times \{\varepsilon_2 c\}) \setminus (F_1 \times J)}) \star [0,2+c]
\]
contained in~$\partial B \setminus \delta'$, see the right of Figure~\ref{fig:SlidingDeltaUR}.
Note that the curves~$\tilde \gamma^\varepsilon$ and~$\delta'$ both decompose into three parts : an arc in~$F_1 \times J \star \{1+c\}$, two vertical arcs in~$(F_1 \cap F_2) \times J^2 \star [1+c, 2+c]$, and an arc in~$F_2 \times J \star \{2+c\}$.
Finally, as in the proof of Lemma~\ref{lemma:IntSR}, the triple~$(\partial B, \tilde \gamma^\varepsilon, \delta')$ is homeomorphic to~$(S^3, \gamma^\varepsilon, \delta)$, see Figure \ref{fig:HomeoBSR} (right). The claim follows.
\end{proof}

The last intersection, between two~$U_\gamma$ spheres, raises a new technical issue.
Indeed, two curves~$\gamma,\gamma' \in \{\gamma_1, \ldots, \gamma_l\}$ might intersect non trivially
away from the clasps, and each such intersection~$x$ yields two intersections at the level of pushoffs.
As these pushoffs are the boundaries of the discs making up the corresponding spheres, these discs do not have disjoint boundaries, making Lemma~\ref{lemma:AlgebraicIntersectionNumber} inapplicable.  (Note that even though curves~$\gamma$ and~$\delta$ can intersect, the relevant discs are bounding
disjoint curves~$\gamma^\varepsilon$ and~$\tilde\delta$, so this issue does not occur in Lemma~\ref{lemma:IntUR}.)

To circumvent this issue, we isotope the spheres as follows: for every~$x\in\gamma\cap\gamma'$ as above, we first shrink~$\gamma\times J$ along the arc~$x \times J$, and then push a small neighborhood of the interior of this arc into~$B^4$ (see Figure~\ref{fig:Shrinking}).
This deformation slightly perturbs the discs~$D_{\gamma^\varepsilon}$, whose boundary is now disjoint from~$\partial D_{\gamma'^\varepsilon}$. 
Via this move, the intersection~$x \times J$ between the bands~$\gamma\times J$ and~$\gamma'\times J$ is replaced by two intersection points between the boundary of the shrinked band and the interior of the other one. We will refer to this process as \emph{shrinking}.

\begin{figure}[h]
\centering
\begin{overpic}[width=13cm]{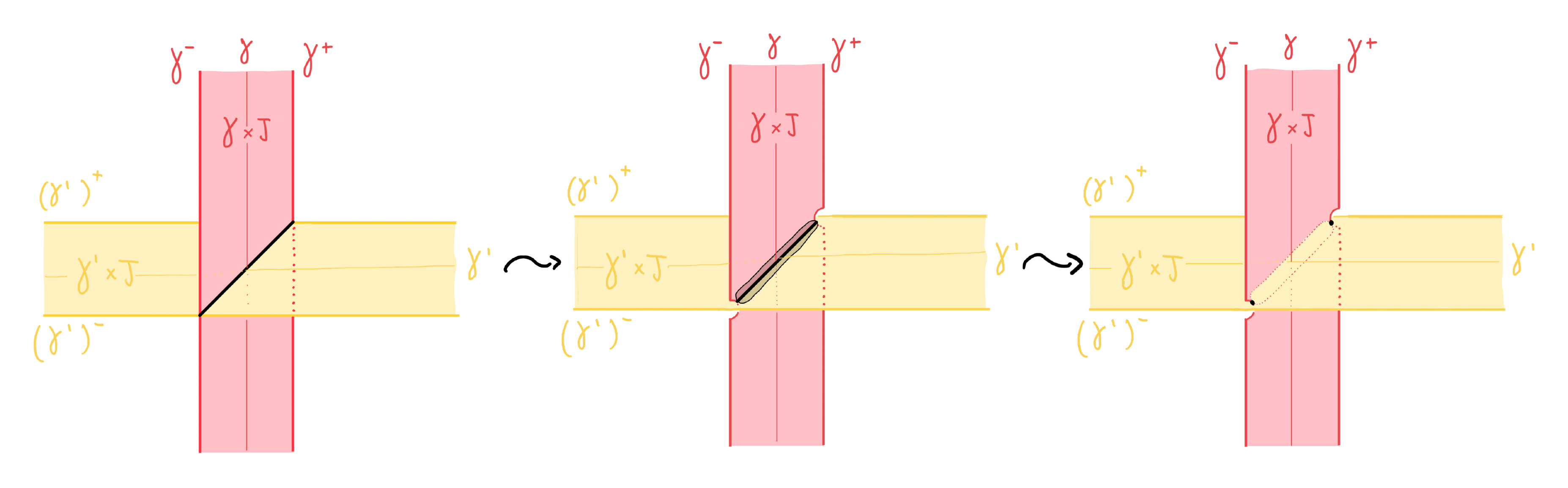}\end{overpic}
\caption{The process of \emph{shrinking}: an intersection arc between two bands (left) is first shrunk (middle),
and a neighborhood of the arc is pushed inside~$B^4$ resulting in two intersection points (right).}
\label{fig:Shrinking}
\end{figure}

\begin{lemma}
\label{lemma:IntUU}
For any~$\gamma,\gamma'\in\{\gamma_1,\dots,\gamma_l\}$, we have
\[
\lambda_{W_F}(\tilde U_{\gamma},\tilde U_{\gamma'})= \sum_{\varepsilon \in \{\pm1\}^2}{(t_1^{-\varepsilon_1}-1)(t_2^{-\varepsilon_2}-1) \lk(\gamma, \gamma'^\varepsilon)} = H_F(\gamma, \gamma')\,.
\]
\end{lemma}

\begin{proof}
By construction, the spheres~$U_{\gamma}$ and~$U_{\gamma'}$ are made of saddles, bands, and discs,
and one readily checks that intersections can only occur between pairs of elements of the same type, and at their boundary~$\gamma^\varepsilon,\gamma'^{\varepsilon'}$.

\begin{claim*}
The only non-trivial contributions come from intersections of disc interiors.
\end{claim*}
To prove this claim,  first note that since the two saddles (say, of~$\gamma$) are oriented, they can be slightly pushed in the direction of their
positive normal vector in~$S^3 \star \{0\}$, as shown in Figure~\ref{fig:DisjointSaddles}. This ensures that
the saddles of~$\gamma$ are disjoint from those of~$\gamma'$.
\begin{figure}[h]
\centering
\begin{overpic}[width=5cm]{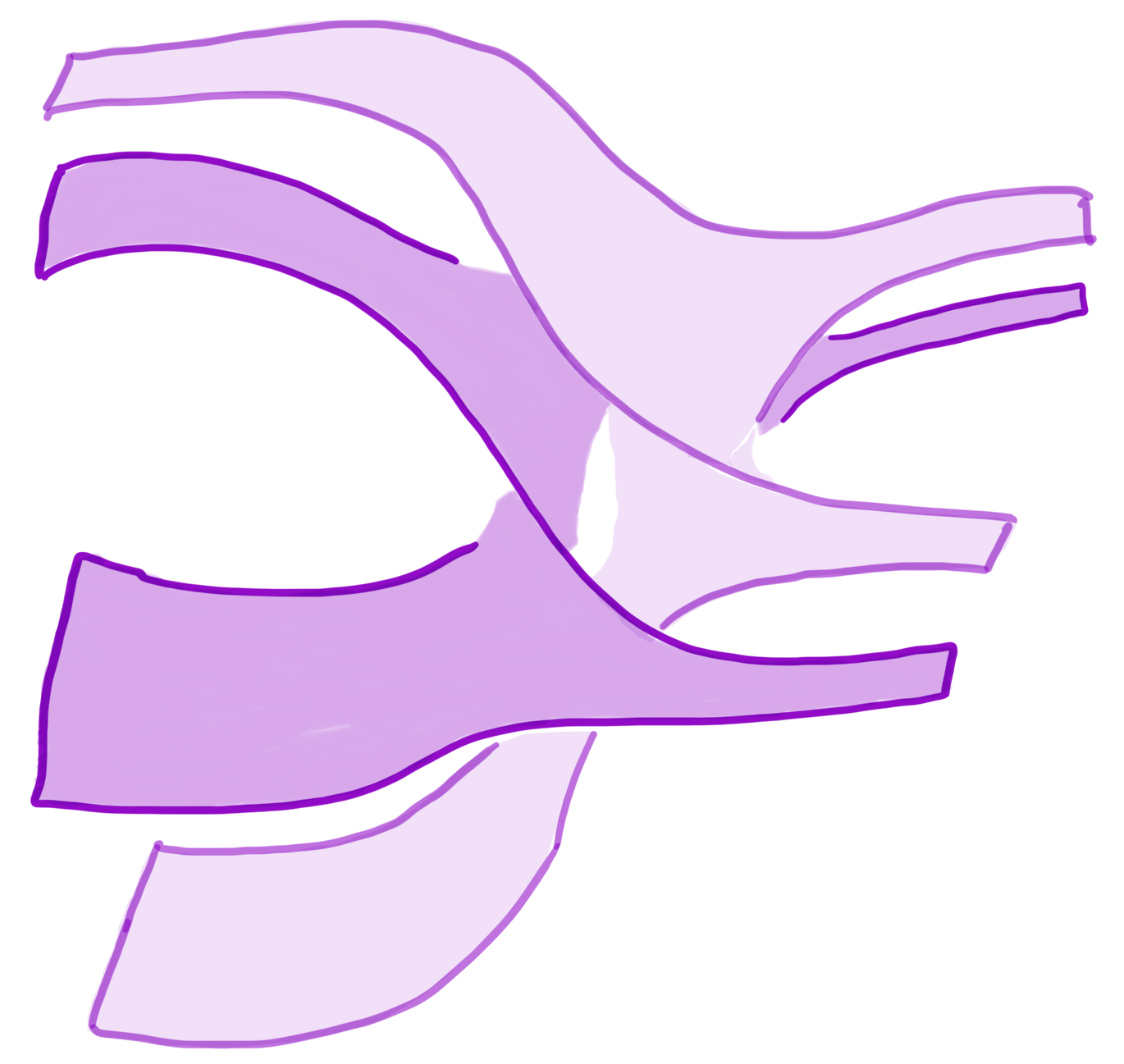}\end{overpic}
\caption{After a small isotopy, the saddles can be made disjoint.}
\label{fig:DisjointSaddles}
\end{figure}

As for the intersections of bands, they consist of arcs of the form~$x\times J$ with~$x\in\gamma\cap\gamma'$
away from the clasps.
Via shrinking, these can be traded for pairs of intersection points in the interior of~$\gamma'\times J$.
Note that shrinking also ensures that the boundary curves~$\gamma^\varepsilon$ and~$\gamma'^{\varepsilon'}$ are disjoint.

We now assert that the contribution of such a pair of points to the equivariant intersection is trivial.
As the two points~$p^+,p^-$ clearly contribute with opposite signs, we only need to check that~$[\xi_{p^+}]=[\xi_{p^-}]$. To do so, we assume without loss of generality~$x\in F_1$, fix arbitrary~$\varepsilon_2,\varepsilon'_2\in\{\pm\}$,
and consider
\[
p^+ \in \partial D_{\gamma^{+,\varepsilon_2}} \cap (\gamma'^{+, \varepsilon_2'} \times J), \quad p^- \in \partial D_{\gamma^{-,\varepsilon_2}} \cap (\gamma'^{-, \varepsilon_2'} \times J)\,.
\]
The only portion of the loop~$\xi_{p^\pm}$ which contributes non-trivially to its homotopy class is the part linking the basepoint of~$U_\gamma$ (which we can assume to be in~$\gamma^{--}$) to the basepoint of~$U_{\gamma'}$ (in~$\gamma'^{--}$) via the point~$p^{\pm}$. This readily yields~$[\xi_{p^\pm}] = t_2^{\frac{\varepsilon_2-\varepsilon_2'}{2}}$, proving the assertion and therefore the claim.

Two applications of Lemma~\ref{lemma:Pi1Term} show that for all~$p\in D_{\gamma^\varepsilon} \cap D_{\gamma'^{\varepsilon'}}$, we have~$[\xi_p] = t_1^{\frac{\varepsilon_1 - \varepsilon_1'}{2}} t_2^{\frac{\varepsilon_2 - \varepsilon_2'}{2}}$.
Proposition~\ref{prop:FL1} and Lemma~\ref{lemma:AlgebraicIntersectionNumber} then yield
\[
\lambda_{W_F}(\tilde U_{\gamma},\tilde U_{\gamma'}) = \sum_{\varepsilon,\varepsilon'} \sum_{p \in D_{\gamma^\varepsilon} \cap D_{\gamma'^{\varepsilon'}}} [\xi_p] \,Q_p(D_{\gamma}^\varepsilon, D_{\gamma'^{\varepsilon'}})
=\sum_{\varepsilon,\varepsilon'}t_1^{\frac{\varepsilon_1 - \varepsilon_1'}{2}} t_2^{\frac{\varepsilon_2 - \varepsilon_2'}{2}}\lk(\partial D_{\gamma^\varepsilon},\partial D_{\gamma'^{\varepsilon'}})\,.
\]
Finally, since~$\gamma \times J$ is shrunk at each intersection point, half of the resulting portion of the band gives an isotopy between~$\partial D_{\gamma^\varepsilon}$ and~$\gamma$ in~$S^3 \setminus \gamma'^{\varepsilon'}$. Together with our orientation conventions, this yields
\[
\lk(\partial D_{\gamma^\varepsilon},\partial D_{\gamma'^{\varepsilon'}}) = \varepsilon_1\varepsilon_2\varepsilon_1'\varepsilon_2' \, \lk(\gamma,\gamma'^{\varepsilon'})\,.
\]
The first equality in the statement now follows from the identity
\[
\sum_{\varepsilon} \varepsilon_1\varepsilon_2 \varepsilon_1'\varepsilon_2'\, t_1^{\frac{\varepsilon_1 - \varepsilon_1'}{2}} t_2^{\frac{\varepsilon_2 - \varepsilon_2'}{2}}=(t_1^{-\varepsilon'_1}-1)(t_2^{-\varepsilon'_2}-1)\,,
\]
and the second equality from the definition of~$H_F$.
\end{proof}

\subsection{Proof of Theorem~\ref{thm:ExistenceOfW}~(3)}
\label{sub:EquivIntForm}

In this section, we use the computations from Section~\ref{sub:CalculInt} to check that~$\mathcal{C}$ forms a basis of~$H_2(W_F;\Lambda)$, and to prove the third item of Theorem~\ref{thm:ExistenceOfW}.
This relies on the following elementary but crucial result; see e.g.\cite[Lemma 2.4]{JuhaszPowell} or~\cite{FL26}.
We refer to the upcoming article~\cite{FL26} for the easy proof.

\begin{lemma}
\label{lemma:FL2}
Let~$R$ be a commutative ring with involution, and let~$\lambda \colon R^n \times R^n \to R$ be a non-degenerate hermitian form represented by a matrix~$A$. Fix~$\{v_1, \ldots, v_n\} \subset R^n$ and define the matrix~$H$ by~$H_{ij} := \lambda(v_i,v_j)$. If~$\det(H) = u \cdot \det(A)$ for some unit~$u \in R^*$, then~$\{v_1, \ldots, v_n\}$ forms a basis of~$R^n$, in which~$H$ represents $\lambda$.\qed
\end{lemma}

We now prove the third item of Theorem~\ref{thm:ExistenceOfW}.

\begin{proof}[Proof of Theorem~\ref{thm:ExistenceOfW}~(3)]
Let~$F$ be a nice~C-complex for a~$\mu$-component link~$L$ with~$\mu\le 2$ and~$\Delta_L\neq 0$, and let~$W_F$ be the~$4$-manifold constructed in the proof of the first item.
First note that the ranks of the~$\Lambda$-module~$H_2(V_F;\Lambda)$ and of the~$\Z$-module~$H_1(F)$ coincide:
this can be checked using an Euler characteristic argument, and is also a direct consequence of~\cite[Proposition~4.1]{CFT18}. 
By Lemma~\ref{lem:V-WoverQ}, the cardinality of~$\mathcal C$ hence coincides with the rank of~$H_2(W_F;\Lambda)$.
By Lemmas~\ref{lemma:IntTR}--\ref{lemma:IntUU} together with the fact that~$\lambda_{W_F}$ is hermitian, the matrix~$H_F$ satisfies~$H_F(\Sigma_1,\Sigma_2) = \lambda_{W_F}(\tilde \Sigma_1,\tilde \Sigma_2)$ for all~$\Sigma_1,\Sigma_2\in \mathcal C$.
Thus, by Lemma~\ref{lemma:FL2}, it suffices to verify that~$\det(H_F)$ equals the determinant of a matrix representing the equivariant intersection form up to multiplication by a unit of~$\Lambda$, a relation which we denote by~$\dot{=}$.

By Proposition~\ref{prop:BlanchfieldRepresentation}, any matrix representing the equivariant intersection form on~$H_2(W_F;\Lambda)$ is also a presentation matrix for the~$\Lambda$-module~$H_1(\partial W_F;\Lambda)$. 
Therefore, the determinant of such a matrix is the order of~$H_1(M_L;\Lambda)$. 
The result therefore reduces to proving that 
$$ \operatorname{Ord} H_1(M_L;\Lambda)
\doteq 
\det(H_F).
$$
If $\mu=1$,  this follows readily (for the first equality, see e.g.~\cite[Proposition~3.3]{CFP}):
$$
\operatorname{Ord} H_1(M_L;\Lambda)
\doteq  \Delta_L(t)
\doteq \det(tA-A^{\T})
\doteq \det(H_F).
$$
We therefore assume that $\mu=2$.
Combining~\cite[Proposition 3.3]{CFP} with~\cite[Corollary~2.2]{Cooper} and a tedious calculation using the definition of $H_F$, we obtain 
\begin{align*}
\operatorname{Ord} H_1(M_L;\Lambda)
&\doteq \left((t_1-1)(t_2-1)\right)^{|\lk(K_1,K_2)|-1} \Delta_L(t_1,t_2) \\
&\doteq \left((t_1-1)(t_2-1)\right)^{|\lk(K_1,K_2)|-1} (t_1-1)^{-2g_2}(t_2-1)^{-2g_1}\det(t_1 t_2 A - t_1 B - t_2 B^{\T} + A^{\T}) \\
&\doteq \left((t_1-1)(t_2-1)\right)^{l-n} (t_1-1)^{-2g_2}(t_2-1)^{-2g_1}\det(t_1 t_2 A -  t_1 B - t_2 B^{\T} + A^{\T}) \\
&\doteq \det(H_F).
\end{align*}
In the penultimate equality, we used the equality~$|\lk(K_1,K_2)| - 1 = l - n$ which follows from the fact that the total number~$N$ of clasps satisfies~$N-1=l+n$ (the rank of~$H_1(\Gamma_F)$) as well as~$N = 2n + |\lk(K_1,K_2)|$. 
This completes the proof.
\end{proof}

We conclude this section with the proof of Corollary~\ref{cor:IntroSeifert}.

\begin{proof}[Proof of Corollary~\ref{cor:IntroSeifert}]
Let~$L$ be a~$\mu$-component link with~$\mu\le 2$ and~$\Delta_L\neq 0$.
For any nice C-complex~$F$ for~$L$, the fact that~$[H_F]\in L^4(Q)$ is invariant by concordance follows from
Theorems~\ref{thm:Calculation},~\ref{thm:ConcordanceInvarianceGeneral} and~\ref{thm:ExistenceOfW}~(3).
Moreover, the extended multivariable signature~$\sigma_L\colon(S^1)^\mu\to\Z$ evaluated at~$\omega$ is defined as follows (for~$\mu\le 2$): it is the signature of the intersection form on~$H_2(W;\C^\omega)$, where~$(W,\Phi)$ is any~$4$-manifold bounding~$(M_L,\varphi)$ over~$\Z^\mu$ with~$\varphi$ meridional, and where the coefficient system is defined via the composition of the morphism~$\Z[H_1(W)]\to\Z[\Z^\mu]=\Lambda$ induced by~$\Phi$ with the evaluation map~$\Lambda\to\C$ given by~$t_i\mapsto\omega_i$.
By Theorem~\ref{thm:ExistenceOfW}, the matrix~$H_F$ represents the intersection form on~$H_2(W_F;\Lambda)$,
with~$(W_F,\Phi)$ bounding~$(M_L,\varphi)$ over~$\Z^\mu$.
By naturality of the intersection forms (see e.g.~\cite[Appendix~B]{CMP}), this implies that~$\sigma_L(\omega)$
is equal to the signature of the matrix~$H_F$ evaluated at~$t_i=\omega_i$.
\end{proof}

\appendix

\section{An exact sequence in~$L$-theory}
\label{sub:Ltheory}

This section reviews some $L$-theory as well as an associated exact sequence.
We also briefly mention Cappell and Shaneson's $\Gamma$-groups.
There are numerous references on these topics, but we mostly follow~\cite{Karoubi,CappellShaneson,CarlssonMilgram,Ranicki}.

\medbreak

Given a ring~$R$ with involution, we write~$L^{2n}(R)$ for the Witt group of nonsingular~$(-1)^n$-hermitian forms~$(P,\lambda)$, where~$P$ is a projective~$R$-module. 
More explicitly,  two hermitian forms are equivalent if they become isometric after some number of metabolic forms is added to each.
Here, a hermitian form~$(P,\lambda)$ is \emph{metabolic} if there is a direct summand~$F \subset P$ with~$F=F^\perp$.

\begin{remark}
When treating the long exact sequence in~$L$-theory, it is common to work with projective modules~\cite{CarlssonMilgram,Ranicki}.
In surgery theoretic applications,~$L$-groups are instead typically defined using stably free modules, and the notation~$L^n_p$ is used to emphasise the use of projective modules.
Following~\cite{Ranicki}, we will not use this convention. Note that since we will almost exclusively be using~$R=\Lambda$ and~$R=Q$, these distinctions do not matter: over these rings, a module is (stably) free if and only if it is projective. 
\end{remark}

Let~$S \subset R$ be a multiplicative subset.
We write~$L^{2n}(R,S)$ for the Witt group of nonsingular~$S^{-1}R/R$-valued~$(-1)^n$-hermitian even linking forms~$(T,b)$, where~$T$ is an~$R$-module of projective dimension~$1$ that is~$S$-torsion, i.e. satisfies~$S^{-1}R \otimes_R T=0$.
A linking form is \emph{even} if~$b(x,x)=\alpha$ mod~$R$ for some~$\alpha \in S^{-1}R$ with~$\overline{\alpha}=(-1)^n\alpha$.
In Carlsson-Milgram, all linking forms are assumed to be even~\cite[Definition 4 (c)]{CarlssonMilgram} but following~\cite[page~223]{Ranicki}, we prefer to isolate the condition as it plays an important role in the long exact sequence of~$L$-groups.
Again,  two even linking forms are equivalent in~$L^{2n}(R,S)$ if they become isometric after some number of even metabolic forms is added to each.
Here, a linking form~$(T,b)$ is \emph{metabolic} if there is a projective dimension one submodule~$F \subset T$ with~$F=F^\perp$.

\begin{remark}
In order to relate~$L^{2n}(S^{-1}R)$ to~$L^{2n}(R)$,  it is necessary to work with the subset~$L^{2n}_S(S^{-1}R)$ that consists of hermitian forms defined on modules of the form~$P=P' \otimes S^{-1}R$ with~$P'$ projective, see~\cite[page 172]{Ranicki} and~\cite{CarlssonMilgram}.
When~$S^{-1}R$ is the field of fractions,  this condition is automatic, yielding~$L_S^4(Q)=L^4(Q)$.
As we will explain below,~$L^{2n}_S(S^{-1}R)$ is isomorphic to the Witt group of hermitian forms over~$R$ that become nonsingular over~$S^{-1}R$; following~\cite[page 150]{Ranicki}, we call such forms~\emph{$S$-nonsingular}.
In particular, elements of~$L^{2n}_S(S^{-1}R)$ can be represented by~$R$-valued hermitian forms defined on~$R$-modules.
Note that when $S=R \setminus \{0\}$, being $S$-nonsingular is the same as being nondegenerate.
\end{remark}

The main point of this section is to recall the following more or less well known fact; see e.g.~\cite[Lemma~1]{CarlssonMilgram} or~\cite[page~172]{Ranicki}.

\begin{proposition}
\label{prop:LTheoryMap}
There is an exact sequence
$$
L^4(R) \xrightarrow{i} L^4_S(S^{-1}R) \xrightarrow{\partial} L^4(R,S)\,,
$$
where the map~$i \colon L^4(R) \to L^4_{S}(S^{-1}R)$ is induced by the inclusion~$R \to S^{-1}R$,
and the map $\partial \colon L^4_S(S^{-1}R) \to L^4(R,S)$ takes an $S$-nonsingular form to its boundary linking form.
%
\end{proposition}

We now recall the definition of the boundary linking form of an $S$-nonsingular hermitian form following~\cite[page 243]{Ranicki}, but note that the definition can also be expressed in terms of the ``dual lattice"; see e.g.~\cite[pages 244-245]{Ranicki} and~\cite[Section 4]{CarlssonMilgram}.

\begin{definition}
\label{def:BoundaryLinking}
Let~$(P,\lambda)$ be an~$S$-nonsingular hermitian form, and let~$\Ad_\lambda\colon P\to\overline{\Hom_R(P,R)}$ be its adjoint.
 The \emph{boundary linking form} of~$(P,\lambda)$ is the even linking form
\begin{align*}
 \partial \lambda \colon \coker(\Ad_\lambda) \times \coker(\Ad_\lambda) &\to S^{-1}R/R  \\
 ([x],[y]) &\mapsto  \frac{y(z)}{s},
\end{align*}
where since~$\coker(\Ad_{\lambda})$ is~$S$-torsion, there exists~$s \in S$ and~$z \in P$ such that~$sx=\Ad_\lambda(z)$.
\end{definition}

It is not difficult to show that~$\partial \lambda$ is independent of the choices involved, sesquilinear and hermitian.
We emphasise that the boundary linking form is necessarily even.

The next remark gives a description of the boundary linking form in terms of matrices in the case~$R=\Lambda$ and~$S^{-1}R=Q$.
This is somewhat well known, see e.g.~\cite[Remark 2.4]{ConwayPowell},  and is used in the proof of Theorem~\ref{thm:RelatingInvariantsIntro}.

\begin{remark}
\label{rem:AlternativeDefinitionBoundaryLinking}
Set~$S=\Lambda\setminus\{0\}\subset\Lambda=R$, and let~$(P,\lambda)$ be an~$S$-nonsingular hermitian form.
By hypothesis, the adjoint
\[
\Ad_{\lambda_Q}:=\Ad_\lambda\otimes \id_Q\colon P\otimes_\Lambda Q\to \overline{\Hom_\Lambda(P,\Lambda)}\otimes_\Lambda Q=\overline{\Hom_Q(P\otimes_\Lambda Q,Q)}
\]
 is an isomorphism, where the identification follows from the fact that~$P$ is free (recall that projective $\Lambda$-modules are free).
 By definition, the boundary linking form of~$(P,\lambda)$ can be described as~$\partial \lambda([x],[y])=y (\Ad_{\lambda_Q}^{-1}(x))$ for~$[x],[y] \in \coker(\Ad_\lambda)$.
Choose a basis~$(e_i)_{i=1}^n$ for the free module~$P$ and endow~$\overline{\Hom_\Lambda(P,\Lambda)}$ with the dual basis.
If~$H_{ij}=\lambda(e_i,e_j)$ is a matrix representing~$\lambda$, then~$\overline{H}$ is a matrix for~$\Ad_\lambda$ and
\[
\partial \lambda([x],[y])
=y (\Ad_{\lambda_Q}^{-1}(x))
=(\overline{H}^{-1}x)^T\overline{y}
=x^TH^{-1}\overline{y}\,.
\]
 \end{remark}
 
 We conclude by discussing~$\Gamma$-groups following~\cite{CappellShaneson} and~\cite[page 150]{Ranicki}, as they provide a conceptually helpful way of thinking of~$L_S^{2n}(S^{-1}R)$ and the long exact sequence in~$L$-theory.
 The\emph{~$\Gamma$-group}~$\Gamma^{2n}(A \to B)$ of a ring homomorphisms~$f \colon A \to B$ is the Witt group of~$(-1)^n$-hermitian forms~$(P,\lambda)$ over~$A$ such that~$(P,\lambda) \otimes B$ is nonsingular as a hermitian form over~$B$.
 Again, we assume that $P$ is projective.
As mentioned on~\cite[page 209]{Ranicki},  by ``clearing denominators",  the map~$(P,\lambda) \mapsto (P,\lambda) \otimes_R S^{-1}R$ gives rise to an isomorphism
$$\Gamma^{2n}(R \to S^{-1}R) \cong L_S^{2n}(S^{-1}R).$$
With this identification in mind, the long exact sequence in~$L$-theory can be viewed as a special case of the long exact sequence of a pair for~$\Gamma$-groups~\cite{CappellShaneson,Ranicki}.

\begin{remark}
As we mentioned in the introduction, our torsion modules do not always have projective dimension $1$ and our linking forms are sometimes only nondegenerate.
For this reason we write $L^4_\textit{nd}(R,S)$ for the Witt group of nondegenerate linking forms.
This group has already been considered in~\cite{Hillman,BorodzikFriedlPowell}.
\end{remark}

\section{The maximal pseudonull module}
\label{sub:PseudoNull}

In this section, we recall the notion of maximal pseudonull submodule of a given module following~\cite{hillman2012}, and give a characterisation of its generators.
This submodule plays a crucial role in our work, since it appears as the kernel of the adjoint of the Blanchfield pairing.

\medskip

Throughout this section,~$M$ denotes a finitely generated module over a noetherian ring~$R$ which is assumed to be a unique factorisation domain (or UFD).

Following~\cite[$\S$4.4]{Bourbaki} and~\cite[p.32]{HillmanAlexanderIdeals}, such a module~$M$ is said to be {\em pseudonull\/} if~$M_{\wp}=0$ for all prime ideal~$\wp$ of height at most~$1$.
We will neither employ this classical definition nor recall the ingredients involved in it, but rather make use of an equivalent definition. To state it, recall that for any subset~$S\subset M$, the ideal
\[
\Ann(S)=\{\alpha\in R\mid \alpha m=0\text{ for all~$m\in S$}\}
\]
of~$R$ is called the \emph{annihilator} of~$S$.

\begin{remark}
\label{rem:pseudonull}
A module~$M$ is pseudonull if and only if~$\Ann(M)$ is not contained in a proper principal ideal of~$R$.
Indeed, by~\cite[p.33]{HillmanAlexanderIdeals} (see also~\cite[Proposition~9]{Bourbaki}), an~$R$-module~$M$
is pseudonull if and only if the intersection of the principal ideals of~$R$ which contain~$\Ann(M)$ is equal to~$R$.
The claim follows readily.
\end{remark}

Let~$M$ be an~$R$-module, and let~$TM$ be its torsion submodule.
The following terminology can be found in~\cite[Section~2.5]{hillman2012}.

\begin{terminology}[\cite{hillman2012}]
\label{def:MaximalPseudonull}
The \emph{maximal pseudonull submodule}~$zM$ of~$M$ is the submodule of~$TM$ generated by all elements whose annihilator ideal is not contained in a proper principal ideal of~$R$.
\end{terminology}

This terminology is justified by the following fact.

\begin{remark}
\label{rem:N-in-zM}
The maximal pseudonull submodule~$zM$ contains every pseudonull submodule of~$M$.
Indeed, if a submodule~$N$ of~$M$ is pseudonull, then~$\Ann(N)=\bigcap_{m\in N}\Ann(m)$ is not 
contained in a proper principal ideal of~$R$ by Remark~\ref{rem:pseudonull}, so the same holds for
each~$\Ann(m)$ with~$m\in N$, leading to the inclusion~$N\subset zM$. 
It is however not clear that~$zM$ is pseudonull, but we do not need this fact.
\end{remark}

\medskip

We now proceed towards an alternative characterisation of~$zM$.
Recall that elements in a UFD~$R$ are said to be {\em relatively prime\/} (or {\em coprime\/})
if their common divisors are the units of~$R$.
The following result is probably known, but we have not been able to find it in the literature.

\begin{proposition}
\label{prop:pseudonull}
Let~$R$ be a UFD. Given an ideal~$I$ in~$R$, the following are equivalent:
\begin{enumerate}[(i)]
\item The ideal~$I$ is not contained in any proper principal ideal.
\item There exist~$\alpha,\beta\in I$ which are coprime.
\end{enumerate}
\end{proposition} 

For the proof of Proposition~\ref{prop:pseudonull}, we will need the following lemma.

\begin{lemma}
\label{lemma:gcd}
Let~$I$ be an ideal in a GCD domain~$R$. If there exists~$\alpha,\beta_1,\dots,\beta_n \in I$ that are relatively prime for some~$n\ge 1$, then there exists~$\alpha, \beta \in I$ that are relatively prime.
\end{lemma}

\begin{proof}[Proof of Lemma~\ref{lemma:gcd}]
We proceed by induction on~$n\ge 1$. The case~$n=1$ being trivial, let us fix~$n>1$, assume that the statement holds up to~$n-1$, and consider relatively prime elements~$\alpha,\beta_1,\dots,\beta_n\in I$.
The strategy of the proof is to construct an explicit element~$\beta^*_{n-1}$ of~$I$ such that~$\alpha,\beta_1, \ldots, \beta_{n-2},\beta^*_{n-1}$ are relatively prime, and to conclude by induction.

Let~$\gamma$ be a greatest common divisor of~$\alpha,\beta_1,\ldots, \beta_{n-1}$. We write~$\alpha = \gamma\alpha'$ and~$\beta_i = \gamma\beta'_i$. We fix a greatest common divisor~$\delta$ of~$\alpha',\beta'_1, \ldots, \beta'_{n-2}$  and decompose it as~$\delta = \delta'\delta''$, with~$\delta'$ a greatest common divisor of~$\delta$ and~$\gamma$. By construction,~$\delta''$ and~$\gamma$ are coprime, while~$\delta'$ divides~$\gamma$. Note that~$\beta^*_{n-1}:=\beta_{n-1}+\beta_n\delta''$ belongs to the ideal~$I$.
It remains to check that the elements~$\alpha,\beta_1, \ldots, \beta_{n-2},\beta^*_{n-1}$ of~$I$ are relatively prime.

Let us assume by contradiction that some prime~$p\in R$ divides~$\alpha,\beta_1, \ldots, \beta_{n-2},\beta^*_{n-1}$.
First, we check that~$p$ divides~$\beta_{n-1}$. If~$p$ divides~$\gamma$, then we are done, as~$\gamma$ divides~$\beta_{n-1}$,
so we can assume that~$p$ does not divide~$\gamma$. Since the prime~$p$ divides~$\alpha=\gamma\alpha',\beta_1=\gamma\beta'_{1}, \ldots, \beta_{n-2}=\gamma\beta'_{n-2}$ but not~$\gamma$, it divides~$\alpha',\beta'_{1}, \ldots, \beta'_{n-2}$ and their greatest common divisor~$\delta = \delta' \delta''$.
Since~$p$ is prime, it divides~$\delta'$ or~$\delta''$. If~$p$ divides~$\delta'$, then~$p$ divides~$\gamma$, so~$p$ divides~$\beta_{n-1}$. If~$p$ divides~$\delta''$, since~$p$ also divides~$\beta^*_{n-1}=\beta_{n-1}+\beta_n \delta''$, it divides~$\beta_{n-1}$.
We have checked that~$p$ divides~$\beta_{n-1}$.
Since~$p$ divides~$\beta_{n-1}$ and~$\beta_{n-1}+\beta_n \delta''$, it divides~$\beta_n \delta''$, so  it divides~$\beta_n$ or~$\delta''$. If~$p$ divides~$\beta_n$, it would divide the relatively prime elements~$\alpha,\beta_1, \ldots, \beta_{n}$, a contradiction.
If~$p$ divides~$\delta''$, then it would also divides~$\delta=\delta'\delta''$, a (greatest) common divisor of~$\alpha',\beta'_1, \ldots, \beta'_{n-2}$. Thus,~$p$ divides~$\alpha',\beta'_1, \ldots, \beta'_{n-2}$, so it divides their multiples by~$\gamma$,
namely~$\alpha,\beta_1, \ldots, \beta_{n-2}$. Since~$p$ also divides~$\beta_{n-1}$, it divides their greatest common divisor~$\gamma$. Thus,~$p$ divides the coprime elements~$\gamma$ and~$\delta''$, a contradiction.
\end{proof}

\begin{proof}[Proof of Proposition~\ref{prop:pseudonull}]
Let us first assume that~$I$ contains coprime elements~$\alpha,\beta$. Let~$x\in R$ be such that~$I\subset (x)$. If~$r\in R$ divides~$x$, then~$r$ divides~$\alpha$ and~$\beta$, so it is a unit.
Since~$x\neq 0$ and~$R$ is a UFD, it follows that~$x$ is a unit. Hence, we have~$(x)=R$, showing the first implication.

Let us now assume that~$I$ is not contained in any principal ideal~$(x)\neq R$, and show that~$I$
contains two coprime elements~$\alpha,\beta$. By assumption, the ideal~$I$ is not trivial (recall that a UFD is nontrivial). If~$I$ contains an invertible element, then~$I=R$ and we can set~$\alpha=0_R$ and~$\beta=1_R$. Therefore, we can assume that~$I$ contains an element~$\alpha$ which is neither~$0$ nor a unit.
Since~$R$ is a UFD, we have a decomposition~$\alpha=\prod_{i=1}^n p_i^{m_i}$ with~$p_1,\dots,p_n$ distinct irreducible
elements of~$R$ and~$m_i\ge 1$ for all~$i$. For all~$i$, there exists~$\beta_i\in I$ such that~$p_i$ does not divide~$\beta_i$: otherwise, we would have~$I\subset (p_i)\neq R$ contradicting the assumption.
By construction, no prime that divides~$\alpha$ divides all the~$\beta_i$s, so the elements~$\alpha,\beta_1,\dots,\beta_n$ of~$I$ are relatively prime. By Lemma~\ref{lemma:gcd},
there exists~$\alpha,\beta\in I$ coprime. This concludes the proof.
\end{proof}

Proposition~\ref{prop:pseudonull} immediately implies the following result, which is used in Section~\ref{sec:Blanchfield}.

\begin{corollary}
\label{cor:max-pn}
If~$R$ is a UFD, then the maximal pseudonull submodule~$zM$ of an~$R$-module~$M$ is generated by the
elements of~$TM$ whose annihilator ideal contains two coprime elements.\qed
\end{corollary}

The proof of the following fact is usually assumed in the literature. We include a detailed argument
for the reader's convenience, and as an illustration of the above corollary.

\begin{example}
\label{ex:NoPN}
If~$R$ is a UFD, then the~$R$-module~$R_0/R$ has trivial maximal pseudonull submodule.
To show this, consider~$m\in R_0/R$ admitting two coprime elements~$\alpha,\beta\in\Ann(m)$. We write~$m\in R_0/R$ as the class of a fraction~$\frac{r}{s}\in R_0$ with~$r,s \in R$ coprime and~$s\neq 0$. By assumption, both~$\frac{\alpha r}{s}$ and~$\frac{\beta r}{s}$ are elements of~$R$, so~$s$ divides~$\alpha r$ and~$\beta r$ in~$R$.
Hence, any prime factor~$p$ of~$s$ divides both~$\alpha r$ and~$\beta r$ in~$R$; since~$r$ and~$s$ are coprime,~$p$ divides the coprime elements~$\alpha$ and~$\beta$, a contradiction.
Therefore, the element~$s\in R$ is a unit, showing that~$m=0\in R_0/R$.
By Corollary~\ref{cor:max-pn}, this implies that the maximal pseudonull submodule of the~$R$-module~$R/R_0$ is trivial.
\end{example}

We also record the following example.

\begin{example}
\label{ex:Z-in-zM}
Let~$M$ be a module over~$\Lambda=\Z[\Z^\mu]$ with~$\mu\ge 2$, and let~$N$ be a submodule of~$M$
isomorphic to~$\Lambda/(t_1-1,\dots,t_\mu-1)\cong\Z$. Then, we have the inclusion~$N\subset zM$.
Indeed, the annihilator~$\Ann(N)$ contains the coprime elements~$t_1-1$ and~$t_2-1$ of~$\Lambda$. By Proposition~\ref{prop:pseudonull} and Remark~\ref{rem:pseudonull}, it follows that~$N$ is pseudonull, and therefore contained in~$zM$ by Remark~\ref{rem:N-in-zM}.
\end{example}

For the remainder of this section, we only assume that~$R$ is an integral domain.
Following~\cite{hillman2012}, we denote the quotient module~$TM/zM$ by~$\hat{t}M$.
We also need to consider the full quotient~$M/zM$, which we denote by~$\hat{s}M$.
Finally, we write~$\shat \colon M \to \shat M$ for the canonical projection. The following elementary lemma addresses the functoriality of this construction.

\begin{lemma}
\label{lemma:PnMapInduction}
Any~$R$-linear map~$f\colon M \to N$ satisfies~$f(zM) \subset zN$, and therefore induces a unique~$R$-linear map~$\hat{f} \colon \shat M \to \shat N$ such that~$\hat f(\shat(x)) = \shat (f(x))$ for~$x\in M$. Moreover, for any two composable~$R$-linear map~$f$ and~$g$, we have~$\widehat{g\circ f} = \hat g \circ \hat f$.
\end{lemma}

\begin{proof}
Observe that for any~$m\in M$, the linearity of~$f$ implies the inclusion~$\Ann(m)\subset\Ann(f(m))$. It follows that
if~$\Ann(m)$ is not contained in a proper principal ideal of~$R$, then neither is~$\Ann(f(m))$.
This shows that~$f$ maps the generators of~$zM$ to generators of~$zN$, implying~$f(zM)\subset zN$.
The rest of the proof is left to the reader.
\end{proof}

The following two lemmas will be used in Section~\ref{sec:Blanchfield}.

\begin{lemma}
\label{lemma:ExactSequencePseudonull}
Let~$M_1 \stackrel{\varphi}{\to} M_2 \stackrel{\psi}{\to} M_3$ be an exact sequence of~$R$-modules. If the restriction
$\psi \vert_{zM_2} \colon zM_2 \to zM_3$ is surjective, then the sequence $\shat M_1 \stackrel{\hat \varphi}{\to} \shat M_2 \stackrel{\hat \psi}{\to} \shat M_3$ is exact.
\end{lemma}

\begin{proof}
We need to check the equality~$\ker ( \hat \psi) = \operatorname{Im}(\hat \varphi)$. By Lemma~\ref{lemma:PnMapInduction} and the exactness of the first sequence, we have~$\hat \psi \circ \hat \varphi = \widehat{\psi \circ \varphi}=\widehat{0}=0$, implying the inclusion~$\operatorname{Im} (\hat \varphi ) \subset \ker(\hat \psi)$.
For the other inclusion, fix an arbitrary element~$\shat(x)\in \ker(\hat \psi)\subset\shat M_2=M_2/zM_2$ given by the class of~$x\in M_2$. By assumption, we have~$\shat (\psi(x)) = \hat\psi(\shat(x)) = 0$, so~$\psi(x) \in zM_3$. Since~$\psi\vert_{zM_2}\colon zM_2 \to zM_3$ is surjective, there exists~$y\in zM_2$ such that~$\psi(y) = \psi(x)$. Therefore, 
the difference~$x-y$ belongs to~$\ker(\psi)=\operatorname{Im}(\varphi)$: there exists~$w\in M_1$ such that~$\varphi(w) = x-y$. It follows that
\[
\hat \varphi( \shat(w)) =  \shat(\varphi(w)) = \shat(x-y)= \shat(x)-\shat(y)= \shat(x)\,,
\]
so~$\shat(x)$ belongs to~$\operatorname{Im}(\hat\varphi)$ as claimed.
\end{proof}

Here is the last lemma.

\begin{lemma}
\label{lemma:HomPN}
If~$M,H$ are~$R$-modules and~$N$ is an~$(R,R)$-bimodule with~$zN=0$, then:
\begin{enumerate}[(i)]
\item The module~$\overline{\Hom_R(M,N)}$ has trivial maximal pseudonull submodule.
\item Every~$R$-linear map~$f\colon H\to \overline{\Hom_R(M,N)}$ induces an~$R$-linear map~$\hat f \colon \shat H \to \overline{\Hom_R(M,N)}$.
\item The projection~$M\to \shat M$ induces an isomorphism~$\overline{\Hom_R(\shat M,N)} \cong \overline{\Hom_R(M,N)}$.
\end{enumerate}
\end{lemma}

\begin{proof}
First note that since~$R$ is commutative, we need not worry about left vs right-module structures.
To show the first point, consider~$g\in\overline{\Hom_R(M,N)}$ with~$\Ann(g)$ not contained in any proper principal ideal of~$R$. For any fixed~$m\in M$, we have~$\Ann(g)\subset\Ann(g(m))$, so~$\Ann(g(m))$ is not contained in any proper principal ideal of~$R$ either. Hence, we have~$g(m)\in zN=0$ for all~$m\in~M$, yielding~$g=0$.
The second point is a direct consequence of Lemma~\ref{lemma:PnMapInduction} together with the first point.
For the third point, apply the left-exact contravariant functor~$\overline{\Hom_R(-,N)}$ to the
exact sequence~$0\to zM\to M\to \shat M\to 0$ and note that~$\overline{\Hom_R(zM,N)}=\overline{\Hom_R(zM,zN)}=~0$.
\end{proof}

\bibliographystyle{alpha}
\bibliography{bibliothesis}

\begin{thebibliography}{AADG20}

\bibitem[AADG20]{Davis}
Jonah Amundsen, Eric Anderson, Christopher~William Davis, and Daniel Guyer.
\newblock {The C-complex clasp number of links}.
\newblock {\em Rocky Mountain Journal of Mathematics}, 50(3):839 -- 850, 2020.

\bibitem[BF15]{BorodzikFriedl}
Maciej Borodzik and Stefan Friedl.
\newblock The unknotting number and classical invariants, {I}.
\newblock {\em Algebr. Geom. Topol.}, 15(1):85--135, 2015.

\bibitem[BFP16a]{BorodzikFriedlPowell}
Maciej Borodzik, Stefan Friedl, and Mark Powell.
\newblock Blanchfield forms and {G}ordian distance.
\newblock {\em J. Math. Soc. Japan}, 68(3):1047--1080, 2016.

\bibitem[BFP16b]{BFP}
Maciej Borodzik, Stefan Friedl, and Mark Powell.
\newblock Blanchfield forms and gordian distance.
\newblock {\em Journal of the Mathematical Society of Japan.}, 68:1047--1080,
  2016.
\newblock EPrint Processing Status: Full text deposited in DRO.

\bibitem[Bla57]{Blanchfield}
Richard~C. Blanchfield.
\newblock Intersection theory of manifolds with operators with applications to
  knot theory.
\newblock {\em Ann. of Math. (2)}, 65:340--356, 1957.

\bibitem[BLLV74]{LambdaSpheres}
Jean Barge, Jean Lannes, Fran\c{c}ois Latour, and Pierre Vogel.
\newblock $\lambda $-sph\`eres.
\newblock {\em Annales scientifiques de l'\'Ecole Normale Sup\'erieure}, 4e
  s{\'e}rie, 7(4):463--505, 1974.

\bibitem[Bou65]{Bourbaki}
N.~Bourbaki.
\newblock {\em \'El\'ements de math\'ematique. {F}asc. {XXXI}. {A}lg\`ebre
  commutative. {C}hapitre 7: {D}iviseurs}, volume No. 1314 of {\em Actualit\'es
  Scientifiques et Industrielles [Current Scientific and Industrial Topics]}.
\newblock Hermann, Paris, 1965.

\bibitem[BPR21]{DET}
Stefan Behrens, Mark Powell, and Arunima Ray.
\newblock Context for the disc embedding theorem.
\newblock In {\em The disc embedding theorem}, pages 1--26. Oxford Univ. Press,
  Oxford, 2021.

\bibitem[CF08]{CimasoniFlorens}
David Cimasoni and Vincent Florens.
\newblock Generalized {S}eifert surfaces and signatures of colored links.
\newblock {\em Trans. Amer. Math. Soc.}, 360(3):1223--1264, 2008.

\bibitem[CFP25]{CFP}
David Cimasoni, Livio Ferretti, and Iuliia Popova.
\newblock Extended signatures and link concordance.
\newblock {\em Proceedings of the Edinburgh Mathematical Society}, page 1–35,
  2025.

\bibitem[CFT18]{CFT18}
Anthony Conway, Stefan Friedl, and Enrico Toffoli.
\newblock The {B}lanchfield pairing of colored links.
\newblock {\em Indiana Univ. Math. J.}, 67(6):2151--2180, 2018.

\bibitem[CG86]{CassonGordonCobordism}
A.~J. Casson and C.~McA. Gordon.
\newblock Cobordism of classical knots.
\newblock In {\em \`{A} la recherche de la topologie perdue}, volume~62 of {\em
  Progr. Math.}, pages 181--199. Birkh\"auser Boston, Boston, MA, 1986.
\newblock With an appendix by P. M. Gilmer.

\bibitem[Cha10]{ChaHirzebruch}
Jae~Choon Cha.
\newblock Link concordance, homology cobordism, and {H}irzebruch-type defects
  from iterated {$p$}-covers.
\newblock {\em J. Eur. Math. Soc. (JEMS)}, 12(3):555--610, 2010.

\bibitem[Cha14]{ChaSymmetric}
Jae~Choon Cha.
\newblock Symmetric {W}hitney tower cobordism for bordered 3-manifolds and
  links.
\newblock {\em Trans. Amer. Math. Soc.}, 366(6):3241--3273, 2014.

\bibitem[Cim04]{Cim04}
David Cimasoni.
\newblock A geometric construction of the {C}onway potential function.
\newblock {\em Comment. Math. Helv.}, 79(1):124--146, 2004.

\bibitem[CK25]{CK}
Anthony Conway and Daniel Kasprowski.
\newblock 4-manifolds with a given boundary, November 2025.

\bibitem[CM80]{CarlssonMilgram}
Gunnar~E. Carlsson and R.~James Milgram.
\newblock Some exact sequences in the theory of hermitian forms.
\newblock {\em Journal of Pure and Applied Algebra}, 18:233--252, 1980.

\bibitem[CMP25]{CMP}
David Cimasoni, Maciej Markiewicz, and Wojciech Politarczyk.
\newblock {Torres-Type Formulas for Link Signatures}.
\newblock {\em Michigan Mathematical Journal}, pages 1 -- 76, 2025.

\bibitem[CNT20]{ConwayNagelToffoli}
Anthony Conway, Matthias Nagel, and Enrico Toffoli.
\newblock Multivariable signatures, genus bounds, and 0.5-solvable cobordisms.
\newblock {\em Michigan Math. J.}, 69(2):381--427, 2020.

\bibitem[CO22]{ConwayOrson}
Anthony Conway and Patrick Orson.
\newblock Abelian invariants of doubly slice links.
\newblock {\em Enseign. Math. (2)}, 68(3-4):243--290, 2022.

\bibitem[Con18]{ConwayBlanchfield}
Anthony Conway.
\newblock An explicit computation of the {B}lanchfield pairing for arbitrary
  links.
\newblock {\em Canad. J. Math.}, 70(5):983--1007, 2018.

\bibitem[Coo82a]{Cooper}
D.~Cooper.
\newblock The universal abelian cover of a link.
\newblock In {\em Low-dimensional topology ({B}angor, 1979)}, volume~48 of {\em
  London Math. Soc. Lecture Note Ser.}, pages 51--66. Cambridge Univ. Press,
  Cambridge-New York, 1982.

\bibitem[Coo82b]{CooperThesis}
Daryl Cooper.
\newblock Signatures of surfaces with applications to knot and link cobordism.
\newblock 1982.
\newblock University of Warwick.

\bibitem[COT03]{CochranOrrTeichner}
Tim~D. Cochran, Kent~E. Orr, and Peter Teichner.
\newblock Knot concordance, {W}hitney towers and {$L^2$}-signatures.
\newblock {\em Ann. of Math. (2)}, 157(2):433--519, 2003.

\bibitem[CP23]{ConwayPowell}
Anthony Conway and Mark Powell.
\newblock Embedded surfaces with infinite cyclic knot group.
\newblock {\em Geom. Topol.}, 27(2):739--821, 2023.

\bibitem[CS69]{CrowellStrauss}
R.~H. Crowell and D.~Strauss.
\newblock On the elementary ideals of link modules.
\newblock {\em Trans. Amer. Math. Soc.}, 142:93--109, 1969.

\bibitem[CS74]{CappellShaneson}
Sylvain~E. Cappell and Julius~L. Shaneson.
\newblock The codimension two placement problem and homology equivalent
  manifolds.
\newblock {\em Ann. of Math. (2)}, 99:277--348, 1974.

\bibitem[CS80]{CappellShanesonLinkCobordism}
Sylvain~E. Cappell and Julius~L. Shaneson.
\newblock Link cobordism.
\newblock {\em Comment. Math. Helv.}, 55(1):20--49, 1980.

\bibitem[Dav06]{DavisConcordant}
James~F. Davis.
\newblock A two component link with {A}lexander polynomial one is concordant to
  the {H}opf link.
\newblock {\em Math. Proc. Cambridge Philos. Soc.}, 140(2):265--268, 2006.

\bibitem[DNOP20]{DavisNagelOrsonPowell}
Christopher~W. Davis, Matthias Nagel, Patrick Orson, and Mark Powell.
\newblock Surface systems and triple linking numbers.
\newblock {\em Indiana Univ. Math. J.}, 69(7):2505--2547, 2020.

\bibitem[FL26]{FL26}
Peter Feller and Lukas Lewark.
\newblock Balanced gordian distance, bilinear forms, and cobordisms.
\newblock {\em upcoming paper}, 2026.

\bibitem[FLNP17]{FriedlLeidyNagelPowell}
Stefan Friedl, Constance Leidy, Matthias Nagel, and Mark Powell.
\newblock Twisted {B}lanchfield pairings and decompositions of 3-manifolds.
\newblock {\em Homology Homotopy Appl.}, 19(2):275--287, 2017.

\bibitem[FP14]{FriedlPowellHopf}
Stefan Friedl and Mark Powell.
\newblock Links not concordant to the {H}opf link.
\newblock {\em Math. Proc. Cambridge Philos. Soc.}, 156(3):425--459, 2014.

\bibitem[FQ90]{FQ90}
Michael~H. Freedman and Frank Quinn.
\newblock {\em Topology of 4-manifolds}, volume~39 of {\em Princeton
  Mathematical Series}.
\newblock Princeton University Press, Princeton, NJ, 1990.

\bibitem[Fre82]{Freedman}
Michael~Hartley Freedman.
\newblock The topology of four-dimensional manifolds.
\newblock {\em J. Differential Geometry}, 17(3):357--453, 1982.

\bibitem[Har08]{HarveyHomology}
Shelly~L. Harvey.
\newblock Homology cobordism invariants and the {C}ochran-{O}rr-{T}eichner
  filtration of the link concordance group.
\newblock {\em Geom. Topol.}, 12(1):387--430, 2008.

\bibitem[Hil81]{HillmanAlexanderIdeals}
Jonathan~A. Hillman.
\newblock {\em Alexander ideals of links}, volume 895 of {\em Lecture Notes in
  Mathematics}.
\newblock Springer-Verlag, Berlin-New York, 1981.

\bibitem[Hil12a]{hillman2012}
J.~Hillman.
\newblock {\em Algebraic Invariants Of Links (2nd Edition)}.
\newblock Series On Knots And Everything. World Scientific Publishing Company,
  2012.

\bibitem[Hil12b]{Hillman}
Jonathan Hillman.
\newblock {\em Algebraic invariants of links}, volume~52 of {\em Series on
  Knots and Everything}.
\newblock World Scientific Publishing Co. Pte. Ltd., Hackensack, NJ, second
  edition, 2012.

\bibitem[HKT09]{HambletonKreckTeichner}
Ian Hambleton, Matthias Kreck, and Peter Teichner.
\newblock Topological 4-manifolds with geometrically two-dimensional
  fundamental groups.
\newblock {\em J. Topol. Anal.}, 1(2):123--151, 2009.

\bibitem[Hom17]{HomSurvey}
Jennifer Hom.
\newblock A survey on {H}eegaard {F}loer homology and concordance.
\newblock {\em J. Knot Theory Ramifications}, 26(2):1740015, 24, 2017.

\bibitem[JP24]{JuhaszPowell}
Andr\'{a}s Juh\'{a}sz and Mark Powell.
\newblock Examples of topologically unknotted tori.
\newblock {\em Trans. Amer. Math. Soc. Ser. B}, 11:1266--1293, 2024.

\bibitem[Kar74]{Karoubi}
Max Karoubi.
\newblock Localisation de formes quadratiques. {I}.
\newblock {\em Ann. Sci. \'Ecole Norm. Sup. (4)}, 7:359--403 (1975); ibid. (4)
  8 (1975), 99--155, 1974.

\bibitem[Kea75]{KeartonCobordism}
C.~Kearton.
\newblock Cobordism of knots and {B}lanchfield duality.
\newblock {\em J. London Math. Soc. (2)}, 10(4):406--408, 1975.

\bibitem[Kim15]{KimWhitney}
Min~Hoon Kim.
\newblock Whitney towers, gropes and {C}asson-{G}ordon style invariants of
  links.
\newblock {\em Algebr. Geom. Topol.}, 15(3):1813--1845, 2015.

\bibitem[KM13]{KronheimerMrowkaGaugeS}
P.~B. Kronheimer and T.~S. Mrowka.
\newblock Gauge theory and {R}asmussen's invariant.
\newblock {\em J. Topol.}, 6(3):659--674, 2013.

\bibitem[{Ko}89]{Ko2}
Ki~Hyoung {Ko}.
\newblock {A Seifert-matrix interpretation of Cappell and Shaneson's approach
  to link cobordisms.}
\newblock {\em {Math. Proc. Camb. Philos. Soc.}}, 106(3):531--545, 1989.

\bibitem[KPRT24]{KPRT24}
Daniel Kasprowski, Mark Powell, Arunima Ray, and Peter Teichner.
\newblock Embedding surfaces in 4-manifolds.
\newblock {\em Geom. Topol.}, 28(5):2399--2482, 2024.

\bibitem[Lam06]{Lam06}
T.~Y. Lam.
\newblock {\em Serre's problem on projective modules}.
\newblock Springer Monographs in Mathematics. Springer-Verlag, Berlin, 2006.

\bibitem[LD88]{LeDimetThesis}
Jean-Yves Le~Dimet.
\newblock Cobordisme d'enlacements de disques.
\newblock {\em M\'{e}m. Soc. Math. France (N.S.)}, (32):ii+92, 1988.

\bibitem[Let00]{Letsche}
Carl~F. Letsche.
\newblock An obstruction to slicing knots using the eta invariant.
\newblock {\em Math. Proc. Cambridge Philos. Soc.}, 128(2):301--319, 2000.

\bibitem[Lev69]{LevineKnotCob}
J.~Levine.
\newblock Knot cobordism groups in codimension two.
\newblock {\em Comment. Math. Helv.}, 44:229--244, 1969.

\bibitem[Lev82]{LevineModule2}
Jerome~P. Levine.
\newblock The module of a 2-component link.
\newblock {\em Commentarii Mathematici Helvetici}, 57:377--399, 1982.

\bibitem[Lit84]{LitherlandCobordism}
Richard Litherland.
\newblock {Cobordism of satellite knots.}
\newblock {Four-manifold theory, Proc. AMS-IMS-SIAM Joint Summer Res. Conf.,
  Durham/N.H. 1982, Contemp. Math. 35, 327-362 (1984).}, 1984.

\bibitem[LO00]{LevineOrrSurvey}
Jerome Levine and Kent~E. Orr.
\newblock A survey of applications of surgery to knot and link theory.
\newblock In {\em Surveys on surgery theory, {V}ol. 1}, volume 145 of {\em Ann.
  of Math. Stud.}, pages 345--364. Princeton Univ. Press, Princeton, NJ, 2000.

\bibitem[MM03]{M-M}
Blake Mellor and Paul Melvin.
\newblock A geometric interpretation of {M}ilnor's triple linking numbers.
\newblock {\em Algebr. Geom. Topol.}, 3:557--568 (electronic), 2003.

\bibitem[MR90]{MilgramRanicki}
R.~J. Milgram and A.~A. Ranicki.
\newblock The {$L$}-theory of {L}aurent extensions and genus {$0$} function
  fields.
\newblock {\em J. Reine Angew. Math.}, 406:121--166, 1990.

\bibitem[Nic03]{Nicolaescu}
Liviu~I. Nicolaescu.
\newblock {\em The {R}eidemeister torsion of 3-manifolds}, volume~30 of {\em de
  Gruyter Studies in Mathematics}.
\newblock Walter de Gruyter \& Co., Berlin, 2003.

\bibitem[NP17]{NagelPowell}
Matthias Nagel and Mark Powell.
\newblock Concordance invariance of {L}evine-{T}ristram signatures of links.
\newblock {\em Doc. Math.}, 22:25--43, 2017.

\bibitem[OS03]{OzsvathSzaboTau}
Peter Ozsv\'{a}th and Zolt\'{a}n Szab\'{o}.
\newblock Knot {F}loer homology and the four-ball genus.
\newblock {\em Geom. Topol.}, 7:615--639, 2003.

\bibitem[Pow16]{Powell}
Mark Powell.
\newblock Twisted {B}lanchfield pairings and symmetric chain complexes.
\newblock {\em Q. J. Math.}, 67(4):715--742, 2016.

\bibitem[PP22]{PP22}
JungHwan Park and Mark Powell.
\newblock A ribbon obstruction and derivatives of knots.
\newblock {\em Israel J. Math.}, 250(1):265--305, 2022.

\bibitem[Ran81]{Ranicki}
Andrew Ranicki.
\newblock {\em Exact sequences in the algebraic theory of surgery}, volume~26
  of {\em Mathematical Notes}.
\newblock Princeton University Press, Princeton, NJ; University of Tokyo Press,
  Tokyo, 1981.

\bibitem[Ran02]{RanickiAlgebraicAndGeometric}
Andrew Ranicki.
\newblock {\em Algebraic and geometric surgery}.
\newblock Oxford Mathematical Monographs. The Clarendon Press, Oxford
  University Press, Oxford, 2002.
\newblock Oxford Science Publications.

\bibitem[Ras10]{Rasmussen}
Jacob Rasmussen.
\newblock Khovanov homology and the slice genus.
\newblock {\em Invent. Math.}, 182(2):419--447, 2010.

\bibitem[Ser80]{Serre}
Jean-Pierre Serre.
\newblock {\em Trees}.
\newblock Springer-Verlag, Berlin-New York, 1980.
\newblock Translated from the French by John Stillwell.

\bibitem[She03]{SheihamThesis}
Desmond Sheiham.
\newblock Invariants of boundary link cobordism.
\newblock {\em Mem. Amer. Math. Soc.}, 165(784):x+110, 2003.

\bibitem[Sim26a]{Simian}
Ga\"etan Simian.
\newblock An algebraic concordance group for links.
\newblock {\em upcoming paper}, 2026.

\bibitem[Sim26b]{GaetanPhD}
Ga\"etan Simian.
\newblock {\em On algebraic concordance of links}.
\newblock PhD thesis, Universit{\'e} de Genève, 2026.

\bibitem[Sti22]{StirlingThesis}
Scott Stirling.
\newblock Applications of noncommutative intersection forms to linking.
\newblock 2022.
\newblock
  \url{https://www.maths.gla.ac.uk/~mpowell/Scott_Stirling_PhD_thesis.pdf}.

\bibitem[Tof19]{ToffoliThesis}
Enrico Toffoli.
\newblock Invariants for manifolds with boundary and low-dimensional topology.
\newblock 2019.
\newblock University of Regensburg.

\bibitem[Tof22]{Toffoli}
Enrico Toffoli.
\newblock The {Atiyah}-{Patodi}-{Singer} rho invariant and signatures of links.
\newblock {\em Proc. Edinb. Math. Soc., II. Ser.}, 65(2):404--440, 2022.

\bibitem[Wal69]{Wall}
C.~T.~C. Wall.
\newblock Non-additivity of the signature.
\newblock {\em Invent. Math.}, 7:269--274, 1969.

\bibitem[Wei94]{Weibel}
Charles~A. Weibel.
\newblock {\em An introduction to homological algebra}, volume~38 of {\em
  Cambridge Studies in Advanced Mathematics}.
\newblock Cambridge University Press, Cambridge, 1994.

\end{thebibliography}

\end{document}